\documentclass[a4paper, 11pt]{article}

\usepackage[utf8]{inputenc}
\usepackage[T1]{fontenc}

\usepackage{amsthm}
\usepackage{amssymb}
\usepackage{mathtools}
\usepackage{stmaryrd}	\SetSymbolFont{stmry}{bold}{U}{stmry}{m}{n}	
\usepackage{bm}
\usepackage{mathrsfs}
\usepackage{euscript}
\usepackage{tikz-cd}
\usetikzlibrary{positioning, shapes.geometric, patterns, graphs, decorations.pathmorphing}

\usepackage{paralist}

\usepackage[colorlinks]{hyperref}
\hypersetup{
	linkcolor=blue,
	citecolor=red,
	urlcolor=teal,
	pdftitle = {},
	pdfauthor = {Wen-Wei Li}
}

\newcommand{\Z}{\ensuremath{\mathbb{Z}}}

\newcommand{\R}{\ensuremath{\mathbb{R}}}
\newcommand{\CC}{\ensuremath{\mathbb{C}}}

\newcommand{\Aut}{\operatorname{Aut}}

\newcommand{\Sym}{\operatorname{Sym}}
\newcommand{\Ext}{\operatorname{Ext}}
\newcommand{\Tor}{\operatorname{Tor}}

\newcommand{\dd}{\mathop{}\!\mathrm{d}}

\newcommand{\lrangle}[1]{\ensuremath{\langle #1 \rangle}}
\newcommand{\sgn}{\ensuremath{\mathrm{sgn}}}
\newcommand{\Stab}{\ensuremath{\mathrm{Stab}}}

\newcommand{\identity}{\ensuremath{\mathrm{id}}}

\newcommand{\Hom}{\operatorname{Hom}}
\newcommand{\RHom}{\operatorname{RHom}}

\newcommand{\End}{\operatorname{End}}
\newcommand{\rightiso}{\ensuremath{\stackrel{\sim}{\rightarrow}}}
\newcommand{\leftiso}{\ensuremath{\stackrel{\sim}{\leftarrow}}}
\newcommand{\Rder}{\operatorname{R}\!}	
\newcommand{\Lder}{\operatorname{L}\!}	
\newcommand{\Hm}{\operatorname{H}}	

\newcommand{\Ker}{\operatorname{ker}}

\newcommand{\Image}{\operatorname{im}}
\newcommand{\dotimes}[1]{\ensuremath{\underset{#1}{\otimes}}}

\newcommand{\otimesL}[1][]{\ensuremath{\underset{#1}{\overset{\mathrm{L}}{\otimes}}}}

\newcommand{\Lie}{\operatorname{Lie}}
\newcommand{\Ad}{\operatorname{Ad}}
\newcommand{\Spec}{\operatorname{Spec}}
\newcommand{\Gm}{\ensuremath{\mathbb{G}_\mathrm{m}}}

\newcommand{\utimes}[1]{\ensuremath{\overset{#1}{\times}}}

\newcommand{\GL}{\operatorname{GL}}
\newcommand{\SO}{\operatorname{SO}}

\newcommand{\SL}{\operatorname{SL}}

\theoremstyle{plain}
\newtheorem{proposition}{Proposition}
\newtheorem{lemma}[proposition]{Lemma}
\newtheorem{theorem}[proposition]{Theorem}
\newtheorem{corollary}[proposition]{Corollary}
\theoremstyle{definition}
\newtheorem{definition}[proposition]{Definition}
\newtheorem{definition-theorem}[proposition]{Definition--Theorem}
\newtheorem{definition-proposition}[proposition]{Definition--Proposition}
\newtheorem{remark}[proposition]{Remark}

\newtheorem{example}[proposition]{Example}

\theoremstyle{definition}

\theoremstyle{plain}

\theoremstyle{plain}
\newtheorem{Thm}{Theorem}

\numberwithin{equation}{subsection}
\numberwithin{proposition}{subsection}
\numberwithin{conj}{subsection}	

\newcommand{\cate}[1]{\ensuremath{\mathsf{#1}}}	
\newcommand{\dcate}[1]{\ensuremath{\text{-}\mathsf{#1}}}	

\usepackage{amsmath}
\usepackage[nottoc]{tocbibind}

\usepackage{geometry}
\usepackage{lmodern}

\title{Higher localization and higher branching laws}
\author{Wen-Wei Li}
\date{}

\makeatletter
\renewcommand{\l@section}{\@dottedtocline{1}{1.5em}{2.0em}}
\renewcommand{\l@subsection}{\@dottedtocline{2}{4.0em}{3.0em}}
\makeatother

\begin{document}

\maketitle

\begin{abstract}
	For a connected reductive group $G$ and an affine smooth $G$-variety $X$ over the complex numbers, the localization functor takes $\mathfrak{g}$-modules to $D_X$-modules. We extend this construction to an equivariant and derived setting using the formalism of h-complexes due to Beilinson--Ginzburg, and show that the localizations of Harish-Chandra $(\mathfrak{g}, K)$-modules onto $X = H \backslash G$ have regular holonomic cohomologies when $H, K \subset G$ are both spherical reductive subgroups. The relative Lie algebra homologies and $\Ext$-branching spaces for $(\mathfrak{g}, K)$-modules are interpreted geometrically in terms of equivariant derived localizations. As direct consequences, we show that they are finite-dimensional under the same assumptions, and relate Euler--Poincaré characteristics to local index theorem; this recovers parts of the recent results of M.\ Kitagawa. Examples and discussions on the relation to Schwartz homologies are also included.
\end{abstract}


\tableofcontents

\section{Introduction}\label{sec:intro}
\subsection{Backgrounds}
We work over the complex numbers. Let $G$ be a connected reductive group with Lie algebra $\mathfrak{g}$. The main theme of this work begins with the functor
\[ \mathscr{D}_X \dotimes{U(\mathfrak{g})} (\cdot): \mathfrak{g}\dcate{Mod} \to \mathscr{D}_X\dcate{Mod} \]
where $X$ is a smooth variety with right $G$-action (we say that $X$ is a $G$-variety), with sheaf of algebraic differential operators $\mathscr{D}_X$, and $\mathfrak{g}\dcate{Mod}$ (resp.\ $\mathscr{D}_X\dcate{Mod}$) is the abelian category of $\mathfrak{g}$-modules (resp.\ left $\mathscr{D}_X$-modules that are $\mathscr{O}_X$-quasi-coherent). The tensor product is taken using the homomorphism $j: U(\mathfrak{g}) \to D_X := \Gamma(X, \mathscr{D}_X)$ induced by $G$-action.

This functor is called \emph{localization}. Its significance can be partly explained by the fact that when $X$ is a homogeneous $G$-space, the fiber of $\mathscr{D}_X \dotimes{U(\mathfrak{g})} V$ at $x \in X$ is isomorphic to $V/\mathfrak{h} V$ where $H := \Stab_G(x)$; therefore, localization organizes \emph{various} spaces of co-invariants into a geometric object over $X$.

The celebrated Beilinson--Bernstein localization is the special case when $X$ is the flag variety of $G$, but allowing some twists on $\mathscr{D}_X$. It has tremendous consequences in the study of $\mathfrak{g}\dcate{Mod}$ and related structures. In this work we try to use localization in another way. Namely, we start with a homogeneous $G$-space $X$ and study aspects of representation theory of $G$ or its real form that are ``relative'' to $X$, by means of the localization functors.

Localization has been utilized in a similar flavor in the ``group case'' $G = H \times H$ and $X = H$ by D.\ Ben-Zvi and I.\ Ganev \cite{BZG19}, who applied it to study the asymptotics of matrix coefficients of admissible representations. They also related this construction to Beilinson--Bernstein localization by specialization at infinity. We remark that V.\ Ginzburg's work \cite{Gin89} played an important role there; specifically, it ensures that the localized $D_H$-modules are regular holonomic and puts some control on its characteristic variety --- in fact, their irreducible constituents are character $D_H$-modules.

A similar result of regularity is contained in \cite{Li22} for localizations of Harish-Chandra $(\mathfrak{g}, K)$-modules to $X := H \backslash G$, where $H, K \subset G$ are reductive spherical subgroups. By a \emph{Harish-Chandra $(\mathfrak{g}, K)$-module}, we mean a $(\mathfrak{g}, K)$-module that is finitely generated over $\mathfrak{g}$ and locally $\mathcal{Z}(\mathfrak{g})$-finite, where $\mathcal{Z}(\mathfrak{g})$ is the center of $U(\mathfrak{g})$. By sphericity of $H$ (ditto for $K$), we mean that the homogeneous $G$-space $H \backslash G$ has an open Borel orbit. The proof is based on a variant of Ginzburg's criterion of regularity in \cite{Gin89}. Note that reductivity implies $X$ is affine, so we may work with $D_X\dcate{Mod}$ instead of $\mathscr{D}_X\dcate{Mod}$.

Nonetheless, the usage of localization is made complicated by the following two facts, at least.
\begin{itemize}
	\item Given an equivariant morphism $f: Y_1 \to Y_2$ between smooth $G$-varieties, there is no obvious way to relate the corresponding localizations via inverse or direct images;
	\item Localization is a right exact functor, but it is non-exact in many cases of interest. This will be shown in Remark \ref{rem:non-exactness}.
\end{itemize}

\subsection{Higher localization}
In the first part of this work, we generalize these ideas to the equivariant and derived setting, i.e.\ to higher localizations.

First off, take a reductive subgroup $K \subset G$ and let $(\mathfrak{g}, K)\dcate{Mod}$ be the abelian category of $(\mathfrak{g}, K)$-modules. In representation theory, one is especially interested in its subcategory of Harish-Chandra $(\mathfrak{g}, K)$-modules.

Take an affine homogeneous space $X = H \backslash G$ and let $(D_X, K)\dcate{Mod}$ be the abelian category of $K$-equivariant $D_X$-modules. Since $j: U(\mathfrak{g}) \to D_X$ is equivariant, it is easily seen that the localization lifts to the equivariant level:
\[ D_X \dotimes{U(\mathfrak{g})} (\cdot): (\mathfrak{g}, K)\dcate{Mod} \to (D_X, K)\dcate{Mod}. \]
We wish to left-derive this right exact functor in order to obtain ``higher localizations'', and apply them to Harish-Chandra modules. Two issues arise immediately.
\begin{enumerate}
	\item Homological algebra for $(\mathfrak{g}, K)$-modules is done in the classical way in \cite{KV95}. On the other hand, the correct $K$-equivariant derived category of $D_X$-modules requires less naive techniques when $K \neq \{1\}$, such as that of Bernstein--Lunts \cite{BL94} which replaces $X$ by various resolutions $P \to X$, on which $K$ acts freely (see also \cite{Ac21}).
	
	This kind of ``type mismatch'' makes it non-trivial to derive equivariant localization. Nor is there any obvious way to define it as a compatible family on resolutions.
	
	\item Another option is to consider $(\mathfrak{g}, K)\dcate{Mod} \to D_X\dcate{Mod}$ instead, in order to avoid equivariant derived categories. This discards too many structures, thus impedes applications to higher branching laws. Moreover, since nonzero projectives in $(\mathfrak{g}, K)\dcate{Mod}$ are never Harish-Chandra, it seems difficult to prove regularity or holonomicity in cohomologies, unless one addresses the first issue simultaneously.
	
	For similar reasons, it seems unfeasible to define the derived localization through realization functor \cite[Theorem  A.7.16]{Ac21}.
\end{enumerate}

Our approach to these problems is based on the formalism of \emph{h-complexes}, reviewed in \S\ref{sec:dg}. This theory is presented in \cite{BL95}, attributed to Beilinson--Ginzburg and the earlier work of Duflo--Vergne \cite{DV87}; for subsequent developments, see \cite{Pan95, Pan05, Pan07, Ki12} and \cite[Chapter 7]{BD}. Let $A \in \{U(\mathfrak{g}), D_X\}$, so there is a natural homomorphism $j: \mathfrak{k} \to A$ of Lie algebras (the Lie brackets in $A$ are commutators). A weak $(A, K)$-module is a vector space $M$ with
\begin{itemize}
	\item a left $A$-module structure, say through $\alpha: A \to \End_{\CC}(M)$
	\item an algebraic $K$-action, say through $\rho: K \to \Aut_{\CC}(M)$, compatibly with $A$,
\end{itemize}
but the $\mathfrak{k}$-actions $\dd\rho$ and $\alpha j$ on $M$ are not required to agree; if they agree, we obtain an $(A, K)$-module. An h-complex over $(A, K)$ is a complex $(C, d)$ of weak $(A, K)$-modules together with a family of maps $i_\xi \in \End^{-1}(C) := \End^{-1}_{\CC}(C)$, linear in $\xi \in \mathfrak{k}$, such that
\begin{enumerate}[(i)]
	\item $k i_\xi k^{-1} = i_{\Ad(k)\xi}$ for all $k \in K$;
	\item $i_\xi$ is $A$-linear for all $\xi \in \mathfrak{k}$;
	\item $i_\xi i_\eta + i_\eta i_\xi = 0$ for all $\xi, \eta \in \mathfrak{k}$;
	\item $d i_\xi + i_\xi d = (\dd\rho - \alpha j)(\xi)$.
\end{enumerate}
The last homotopy condition implies that $\Hm^n(C)$ is an $(A, K)$-module for all $n$. There are also translation functors $C \mapsto C[1]$, mapping cones and $\Hom$-complexes for h-complexes. This leads to the h-derived category ${}^{\mathrm{h}} \cate{D}(A, K)$: it is triangulated and endowed with a $t$-structure whose heart is $(A, K)\dcate{Mod}$. For every subgroup $T \subset K$ there are functors of oblivion fitting into commutative diagrams up to isomorphism:
\[\begin{tikzcd}
	{}^{\mathrm{h}} \cate{D}(A, K) \arrow[r, "{\Hm^n}"] \arrow[d] & (A, K)\dcate{Mod} \arrow[d] \\
	{}^{\mathrm{h}} \cate{D}(A, T) \arrow[r, "{\Hm^n}"] & (A, T)\dcate{Mod}.
\end{tikzcd}\]

Therefore we obtain the bounded h-derived categories
\[ {}^{\mathrm{h}} \cate{D}^{\mathrm{b}}(\mathfrak{g}, K), \quad {}^{\mathrm{h}} \cate{D}^{\mathrm{b}}(D_X, K). \]
In \cite{BL95, Pan05}, these are shown to be equivalent to the usual bounded versions
\[ \cate{D}^{\mathrm{b}}(\mathfrak{g}, K), \quad \cate{D}^{\mathrm{b}}_K(X) \]
of the derived category of $(\mathfrak{g}, K)$-modules and the $K$-equivariant derived category of $D_X$-modules, respectively; in particular, the $\Ext$'s are the same. They proved the second case for non-affine $X$ and non-reductive $K$ as well.

The best way to understand the h-construction is to do homological algebra over \emph{Harish-Chandra dg-algebras}; this is just an equivariant version of the theory of dg-modules over dg-algebras (``dg'' = differential graded), which is nowadays standard; see \cite{BL94} or \cite{Yek20}. Consequently, left and right h-derived functors are defined, upon replacing projective (resp.\ injective) resolutions by K-projective (resp.\ K-injective) resolutions. As we assumed $X$ is affine and $K$ is reductive, there are enough K-projective bounded-above h-complexes over $(D_X, K)$.

We adopt this toolbox to deduce the following main theorem. Let
\[ \mathbf{Loc}_X = \mathbf{Loc}_{X, K}: {}^{\mathrm{h}} \cate{D}^{\mathrm{b}}(\mathfrak{g}, K) \to {}^{\mathrm{h}} \cate{D}^{\mathrm{b}}(D_X, K) \]
be the equivariant derived localization functor so obtained; see \S\ref{sec:Loc-functor}. An advantage of the h-construction is that for all reductive subgroups $T$ of $K$, one easily obtains commutative diagrams up to isomorphisms (Proposition \ref{prop:Loc-oblv}):
\[\begin{tikzcd}[column sep=large]
	{}^{\mathrm{h}} \cate{D}^{\mathrm{b}}(\mathfrak{g}, K) \arrow[r, "{\mathbf{Loc}_{X, K}}"] \arrow[d] & {}^{\mathrm{h}} \cate{D}^{\mathrm{b}}(D_X, K) \arrow[d] \\
	{}^{\mathrm{h}} \cate{D}^{\mathrm{b}}(\mathfrak{g}, T) \arrow[r, "{\mathbf{Loc}_{X, T}}"'] & {}^{\mathrm{h}} \cate{D}^{\mathrm{b}}(D_X, T).
\end{tikzcd}\]

A closely related approach to $\mathbf{Loc}_{X, K}$ can be found in \cite[\S 7.8]{BD}, as a part of the ``Hecke patterns''.

\begin{Thm}
	Let $H$ and $K$ be spherical reductive subgroups of $G$. Let $V$ be a Harish-Chandra $(\mathfrak{g}, K)$-module, then the cohomologies of $\mathbf{Loc}_X(V)$ are all regular holonomic as $D_X$-modules, and their characteristic varieties are included in the nilpotent locus.
\end{Thm}

We refer to Theorem \ref{prop:regularity} and the results in that section for complete statements. The proof is based on a criterion (Theorem \ref{prop:reg-criterion}) from \cite{Li22}. One must show that the cohomologies of $\mathbf{Loc}_X(V)$ are (R1) finitely generated over $D_X$, (R2) carry $K$-equivariant structures, and (R3) locally $\mathcal{Z}(\mathfrak{g})$-finite through
\[\mathcal{Z}(\mathfrak{g}) \subset U(\mathfrak{g}) \xrightarrow{j} D_X. \]

Note that (R2) is immediate in our setting. For the remaining conditions, one passes to the case $K = \{1\}$ by commuting $\mathbf{Loc}_X$ and $\Hm^n$ with oblivion. The hardest part is (R3): to prove it, we need deep results of F.\ Knop \cite{Kn94} on the structure of the algebra $D_X^G$ of invariant differential operators in the spherical case. This is the content of Proposition \ref{prop:local-Zg-finiteness}.

We remark that $D_X^G$ also acts on the right of the functor $D_X \dotimes{U(\mathfrak{g})} (\cdot)$, by letting $z \in D_X^G$ send $D \otimes w$ to $D z\otimes w$. This induces a right $D_X^G$-action on $\mathbf{Loc}_X$. As a by-product, Proposition \ref{prop:ZX-locally-finite} shows that $D_X^G$ acts locally finitely on the cohomologies of $\mathbf{Loc}_X(V)$, for all Harish-Chandra $(\mathfrak{g}, K)$-modules $V$.

\subsection{Higher branching}
We hope that the results on $\mathbf{Loc}_X$ will have some use in geometric representation theory. However, this work is originally motivated by \emph{branching laws} in harmonic analysis.

In the setting of $p$-adic groups $H \subset I$, the branching law studies the spaces $\Hom_H(V|_H, W)$ where $V$ (resp.\ $W$) is an admissible representation of $I$ (resp.\ $H$). In particular, one is interested in the dimension of the $\Hom$-space\footnote{We do not consider the case of $\Hom_H(V, W|_H)$ in this work.}. The $\Ext$-analogue or \emph{higher branching law}, proposed by D.\ Prasad \cite{Pra18} for $p$-adic groups, considers the corresponding problem for $\Ext^n_H(V|_H, W)$ for general $n \geq 0$. One is also interested in the Euler--Poincaré characteristic $\sum_n (-1)^n \dim \Ext^n_H(V|_H, W)$, well-defined as long as $\dim \Ext^n_H(V|_H, W)$ is finite and vanishes for $n \gg 0$. By replacing $I$ by $H \times I$ in which $H$ embeds diagonally, the problem can be reduced to the case $W = \CC$, the trivial representation.

This work is motivated by higher branching laws in the Archimedean case. We formulate it in the algebraic framework, namely by considering subgroups
\begin{equation*}\begin{tikzcd}
	H \arrow[phantom, r, "\subset" description] & G \\
	K^H \arrow[phantom, u, "\subset" description, sloped] \arrow[phantom, r, "\subset" description] & K \arrow[phantom, u, "\subset" description, sloped].
\end{tikzcd}\end{equation*}
where all groups are assumed to be complex reductive. We are led to study
\[ \Ext^n_{\mathfrak{h}, K^H}(V|_H, \CC), \quad n \in \Z_{\geq 0} \]
where $V$ is a $(\mathfrak{g}, K)$-module and $V|_H$ denotes its restriction to $(\mathfrak{h}, K^H)$. The $\Ext$ can be computed either in the h-derived category or in the classical one.

Let $X = H \backslash G$ and let $x$ be the point $H \cdot 1$. The inclusion map $i_x: \mathrm{pt} \to X$ is a morphism between $K^H$-varieties. By using the relation between localization and co-invariants, and working systematically in the h-derived categories, one obtains the following canonical isomorphisms. For all $(\mathfrak{g}, K)$-modules $V$ and all $n \in \Z$, we have
\begin{align}
	\label{eqn:Ext-interpretation}
	\Ext^n_{\mathfrak{h}, K^H}(V|_H, \CC) & \simeq \Ext^n_{D_{\mathrm{pt}}, K^H}\left( i_x^\bullet( \mathbf{Loc}_X(V) ), \CC \right), \\
	\label{eqn:Hm-interpretation}
	\Hm_n(\mathfrak{h}, K^H; V|_H) & \simeq \Hm^{-n}\Lder\left( \mathrm{coInv}^{\CC, K^H}_{\CC, \{1\}} \right) \left(i_x^\bullet \mathbf{Loc}_X(V)\right).
\end{align}
For the complete statements, see Propositions \ref{prop:RHom-RHom} and \ref{prop:H-coInv}. Here:
\begin{itemize}
	\item $\Hm_n(\mathfrak{h}, K^H; \cdot)$ are the relative Lie algebra homologies of $(\mathfrak{h}, K^H)$-modules (see \cite[p.157]{KV95}), and note that
	\[ \Ext^n_{\mathfrak{h}, K^H}(V|_H, \CC) \simeq \Hm_n(\mathfrak{h}, K^H; V|_H)^* ; \]
	\item $\mathbf{Loc}_X(V) := \mathbf{Loc}_{X, K^H}(V)$, but it is also isomorphic to the oblivion of $\mathbf{Loc}_{X, K}(V)$ via $K^H \subset K$;
	\item $i_x^\bullet: {}^{\mathrm{h}} \cate{D}^-(D_X, K^H) \to {}^{\mathrm{h}} \cate{D}^-(D_{\mathrm{pt}}, K^H)$ is the h-version of the inverse image functor for $D$-modules;
	\item $\Lder\left( \mathrm{coInv}^{\CC, K^H}_{\CC, \{1\}}\right): {}^{\mathrm{h}} \cate{D}^-(\CC, K^H) \to \cate{D}^-(\CC)$ is the left h-derived functor of taking co-invariants of h-complexes over $(\CC, K^H)$, see \S\ref{sec:inv-coinv}, and $\cate{D}(\CC)$ is the the derived category of $\CC = D_{\mathrm{pt}}$.
\end{itemize}

Note that $i_x^\bullet$ coincides with the inverse image functor under the equivalence ${}^{\mathrm{h}} \cate{D}^{\mathrm{b}}(D_X, K^H) \simeq \cate{D}^{\mathrm{b}}_{K^H}(X)$; see Proposition \ref{prop:inverse-image-compatibility}. The same can also be said for $\Lder\left( \mathrm{coInv}^{\CC, K^H}_{\CC, \{1\}}\right)$, but truncation is needed since this functor does not land in $\cate{D}^{\mathrm{b}}(\CC)$ in general.

Actually, the deductions of \eqref{eqn:Ext-interpretation} and \eqref{eqn:Hm-interpretation} are routine once the basic formalism of \S\S\ref{sec:review-dg}---\ref{sec:Loc-coinv} is set up. The h-formalism is needed only when $K^H \neq \{1\}$.

The result below is a direct consequence of the earlier theorem on regularity.

\begin{Thm}
	Assume that $H, K \subset G$ are both spherical reductive subgroups. For all Harish-Chandra $(\mathfrak{g}, K)$-modules $V$ and all $n \in \Z$, there are canonical isomorphisms
	\begin{align*}
		\Ext^n_{\mathfrak{h}, K^H}(V|_H, \CC) & \simeq \Ext^n_{D_{\mathrm{pt}}, K^H}\left( i_x^! \mathbf{Loc}_X(V)[\dim X], \CC \right), \\
		\Hm_n(\mathfrak{h}, K^H; V|_H) & \simeq \Hm^{- n + \dim X}\Lder\left( \mathrm{coInv}^{\CC, K^H}_{\CC, \{1\}} \right) \left(i_x^! \mathbf{Loc}_X(V)\right).
	\end{align*}
	When $K^H = \{1\}$ we also have
	\begin{align*}
		\Ext^n_{\mathfrak{h}}(V|_H, \CC) & \simeq \Ext^n_{D_X}\left( \mathbf{Loc}_X(V), i_{x, *}(\CC)[\dim X] \right), \\
		\Hm_n(\mathfrak{h}; V|_H) & \simeq \Hm^{-n - \dim X} \left(i_x^* \mathbf{Loc}_X(V)\right).
	\end{align*}
	Here $i_x^!$, $i_x^*$ and $i_{x, *}$ are defined to match the synonymous functors between the constructible equivariant derived categories via Riemann--Hilbert correspondence.
	
	All the complexes of $D$-modules above are bounded with regular holonomic cohomologies.	Consequently, all these vector spaces are finite-dimensional.
\end{Thm}

We refer to Theorem \ref{prop:RHom-Loc} and Proposition \ref{prop:H-Loc} for the complete statements.

The finiteness of $\Hm_n(\mathfrak{h}, K^H; V|_H)$ is covered by a recent work of M.\ Kitagawa \cite{Ki21}, who considers a broader setting of branching laws and obtains uniform bounds. His approach is to interpret relative Lie algebra cohomologies in terms of Zuckerman functors, and then pass to homologies by Poincaré duality. The approach in this work relies on standard results about constructible or regular holonomic equivariant derived categories; although this is conceptually straightforward, the resulting bound is less effective and involves more machinery.

In the special case $K^H = \{1\}$, we can connect the Euler--Poincaré characteristic
\[ \mathrm{EP}_{\mathfrak{h}}(V|_H, \CC) := \sum_n (-1)^n \dim \Ext^n_{\mathfrak{h}}(V|_H, \CC) \]
to a celebrated topological counterpart, namely the local Euler--Poincaré characteristic at $x$ of the \emph{solution complex} of $\mathbf{Loc}_{X, \{1\}}(V)$. See Theorem \ref{prop:local-index} for the complete statement.

\begin{Thm}
	Retain the previous assumptions on $H, K \subset G$ and assume $K^H = \{1\}$. Let $V$ be a Harish-Chandra $(\mathfrak{g}, K)$-module, and set $\mathcal{L} := \mathbf{Loc}_{X, \{1\}}(V)$. Then $\mathrm{EP}_{\mathfrak{h}}(V|_H, \CC)$ equals the local Euler--Poincaré characteristic
	\[ \chi_x\left( \mathrm{Sol}_X(\mathcal{L}) \right) \]
	of the solution complex of $\mathcal{L}$ at $x$, which is expressible in terms of characteristic cycles of $\mathcal{L}$ and Euler obstructions by Kashiwara's local index theorem \cite[Theorem 4.6.7]{HTT08}.
\end{Thm}

We remark that the reductivity of $H$ and $K$ makes homological algebra over $(D_X, K)$ and $(\mathfrak{g}, K)$ easier; it is also involved with the criterion of regularity and the results of Knop in \S\S\ref{sec:regularity-criterion}---\ref{sec:end-of-regularity}. It is still unclear whether our approach can extend to non-reductive subgroups.

It should be emphasized that the finiteness of $\Ext^n$ in the spherical case is just a first example of the usage of higher localization, and the computation of $\chi_x(\mathrm{Sol}_X(\mathcal{L}))$ requires knowledge about characteristic cycles. To gain more applications in higher branching laws, one needs a deeper understanding of $\mathbf{Loc}_X(V)$. Regularity is merely the first step.

\subsection{Twist by characters}
In arithmetic applications, one often has to study branching laws in which the trivial module $\CC$ of $H$ is replaced by some character $\chi$, say for $p$-adic groups. In the setting of $(\mathfrak{g}, K)$-modules, one should consider a $1$-dimensional $(\mathfrak{h}, K^H)$-module; by abusing notation, we identify it with the underlying character $\chi: \mathfrak{h} \to \CC$. Denote by $\chi^\vee$ its contragredient $(\mathfrak{h}, K^H)$-module\footnote{In the study of branching laws of real groups, many authors allow $\chi$ to be a finite-dimensional module. It seems hard to incorporate this case into our formalism.}.

Since $H$ is assumed to be reductive, $\chi$ factors through the Lie algebra of some torus $S = H/\underline{H}$ and can be identified with an element of $\mathfrak{s}^*$. The spaces
\[ \Ext^n_{\mathfrak{h}, K^H}(V|_H, \chi), \quad \Hm_n(\mathfrak{h}, K^H; V|_H \otimes \chi^\vee) \]
can still be interpreted via localization to $X = H \backslash G$, provided that we work with twisted differential operators (TDO's) on $X$.

Specifically, we have an $S$-torsor (on the left)
\[ \pi: \tilde{X} := \underline{H} \backslash G \to H \backslash G = X \]
that is also $G$-equivariant (on the right). This situation is already considered by \cite{BL95}, where it is called an $S$-monodromic $G$-variety.

Let $\mathfrak{m}_\chi \subset \Sym(\mathfrak{s})$ be the maximal ideal corresponding to $\chi$. The sheaf of $\chi$-twisted differential operators on $X$ is
\[ \mathscr{D}_{X, \chi} := (\pi_* \mathscr{D}_{\tilde{X}})^S \big/ \mathfrak{m}_\chi (\pi_* \mathscr{D}_{\tilde{X}})^S. \]

In the affine setting, $D_{X, \chi} = \Gamma(X, \mathscr{D}_{X, \chi})$ equals $D_{\tilde{X}}^S / \mathfrak{m}_\chi D_{\tilde{X}}^S$. There is again a $G$-equivariant homomorphism
\[ j: U(\mathfrak{g}) \to D_{X, \chi}, \]
by which we define the $\chi$-twisted localization
\[ \mathbf{Loc}_{X, \chi} = \mathbf{Loc}_{X, K, \chi}: {}^{\mathrm{h}} \cate{D}^{\mathrm{b}}(\mathfrak{g}, K) \to {}^{\mathrm{h}} \cate{D}^{\mathrm{b}}(D_{X, \chi}, K). \]

All the earlier results have monodromic counterparts in this context recorded in \S\ref{sec:monodromic}. For example, Proposition \ref{prop:coinv-Loc-monodromic} relates the fibers of $D_{X, \chi} \dotimes{U(\mathfrak{g})} V$ to $(\mathfrak{h}, \chi)$-co-invariants of $V$, i.e.\ the quotient of $V$ by the span of all $\eta v - \chi(\eta)v$ where $\eta \in \mathfrak{h}$ and $v \in V$. On the other hand, the $K^H$-equivariant sheaf $\CC$ on $\mathrm{pt}$ must be replaced by the monodromic variant $\CC_\chi$, and this gives rise to monodromic versions of \eqref{eqn:Ext-interpretation} and \eqref{eqn:Hm-interpretation}.

The key ingredient here is a monodromic version of the regularity Theorem \ref{prop:regularity-monodromic} in the spherical case. One has to extend Knop's results \cite{Kn94} to the twisted setting of $D_{X, \chi}$. This is done in Lemma \ref{prop:Knop-monodromic}, by reducing to the spherical homogeneous $S \times G$-space $\tilde{X}$.

We remark that the Riemann--Hilbert correspondence also holds in the twisted (i.e.\ monodromic) and equivariant setting; one has to consider twisted constructible sheaves, though. We refer to \cite{Ka08}.

\subsection{On the analytic picture}
The discussions thus far focus on the algebraic setting. In the representation theory of real groups, one often takes $G$ and $H$ to be real, and take $K$ (resp.\ $K^H$) corresponding to a maximal compact subgroup of $G$ (resp.\ $H$). In studying $\Hom_{\mathfrak{h}, K^H}(V|_H, \CC)$, one is more interested in those element which extend continuously to the Casselman--Wallach globalization $E$ of $V$. This is the same as the continuous linear functionals on the space of co-invariants $E_H := E \big/ \sum_h (h-1)E$ endowed with the quotient topology.

For higher branching laws in the analytic setting, it is thus natural to replace relative Lie algebra homologies by the \emph{Schwartz homologies}
\[ \Hm^{\mathcal{S}}_n(H; E|_H), \quad n \in \Z \]
defined by Y.\ Chen and B.\ Sun \cite{BC21}; they are locally convex topological spaces, not necessarily Hausdorff, and equal to $E_H$ when $n=0$. In fact, they agree with the \emph{smooth homologies} constructed earlier by P.\ Blanc and D.\ Wigner \cite{BW83}.

In \S\ref{sec:comparison} we will define the comparison maps
\begin{equation*}
	c_n(E): \Hm_n\left(\mathfrak{h}, K^H; E^{K\text{-fini}}\right) \to \Hm^{\mathcal{S}}_n(H; E|_H), \quad n \in \Z_{\geq 0}
\end{equation*}
from the algebraic to the analytic picture, for every Casselman--Wallach representation $E$ of $G$, where $E^{K\text{-fini}}$ is the $(\mathfrak{g}, K)$-module of $K$-finite vectors in $E$.

It is natural to ask if $c_n(E)$ is an isomorphism, at least when $H$ is a reductive spherical subgroup of $G$ after extension of scalars to $\CC$. If it is indeed the case, $\Hm^{\mathcal{S}}_n(H; E|_H)$ will be finite-dimensional since $\Hm_n(\mathfrak{h}, K^H; E^{K\text{-fini}})$ is; in turn, this will imply $\Hm^{\mathcal{S}}_n(H; E|_H)$ is Hausdorff.

It seems too early to pass verdict on the question of comparison. Its validity for $n=0$ is equivalent to the automatic continuity theorem, which is open except when $H$ is a symmetric subgroup \cite{BD88} or $H=G$; another (non-reductive) example is when $H$ is maximal unipotent \cite{HT98, LLY21}. Nonetheless, we will settle two cases in affirmative.
\begin{itemize}
	\item In Example \ref{eg:admissible-restriction}, we show that $c_n(E)$ is always an isomorphism when $E$ is $K^H$-admissible in the sense of T.\ Kobayashi. This includes the case when $H$ is symmetric, $H/K^H \to G/K$ is a holomorphic embedding of Hermitian symmetric domains and $E$ is a unitary highest weight module. We refer to \cite{Ko15} for an overview of admissible restrictions.
	
	\item In Examples \ref{eg:SL2}, \ref{eg:SL2-more}, we prove that $c_n(E)$ is always an isomorphism when $G = \SL(2)$ and $H$ is a reductive spherical subgroup. This is done by an explicit computation with principal series, which suffices by a general argument of Hecht--Taylor (Proposition \ref{prop:HT-reduction}). In fact, these computations show that $\Hm_1\left(\mathfrak{h}, K^H; E^{K\text{-fini}}\right)$ can indeed be nonzero in this case.
\end{itemize}

In general, neither the finiteness nor Hausdorffness of $\Hm^{\mathcal{S}}_n(H; E|_H)$ is known for $n > 0$, even when $H$ is reductive and spherical. See \cite{BC21} for related discussions.

\subsection{Structure of this article}
In \S\ref{sec:dg} we collect the necessary definitions and backgrounds on Harish-Chandra dg-algebras, explain the derived categories and derived functors in this generality, and then proceed to the case of h-complexes and h-derived categories. Our references on Harish-Chandra dg-algebras are \cite{BL95, Pan95, Pan05, Pan07}.

In \S\ref{sec:gK-mod}, the formalism of h-complexes is applied to a pair $(\mathfrak{g}, K)$, and a generalized notion of Harish-Chandra modules is presented in Definition \ref{def:HC-module}. The results are either taken from \cite{BL95, Pan05}, or are straightforward analogues from the classical picture, cf.\ \cite{KV95}.

The \S\ref{sec:D-basic} concerns the geometric counterpart. We review the definition of h-derived category of $D$-modules, focusing on the affine case. Most of the materials are from \cite{BL95}. In addition, we show that Beilinson's equivalence is compatible with various operations on equivariant derived categories and their h-analogues, especially for the case of inverse images.

In \S\ref{sec:Loc}, we define the localization functor $\mathbf{Loc}_X$ in the equivariant h-derived setting, relate its fibers to co-invariants, and prove the main Theorem \ref{prop:regularity} about regularity.

In \S\ref{sec:Ext-application}, we begin with the general algebraic framework of $\Ext$-branching laws, relate both the $\Ext$ spaces and relative Lie algebra homologies to localizations, and draw some consequences from regularity in the spherical case (Theorems \ref{prop:RHom-Loc}, \ref{prop:local-index} and Corollaries \ref{prop:Ext-consequence-1}, \ref{prop:Ext-consequence-2}). The monodromic setting discussed in \S\ref{sec:monodromic} is similar.

Finally, in \S\ref{sec:analytic} we review the Schwartz homologies and connect them to the algebraic version via a family of maps $c_n(E)$. After several sundry results, we present three cases in which the comparison maps $c_n(E)$ are isomorphisms: admissible restriction (Example \ref{eg:admissible-restriction}), the diagonal torus in $\SL(2)$ (Example \ref{eg:SL2}), and more general reductive spherical subgroups of $\SL(2)$ (Example \ref{eg:SL2-more}). This section is largely independent of the previous ones.

\subsection{Conventions}
Unless otherwise specified (such as in \S\ref{sec:analytic}), all vector spaces, varieties and algebraic groups are defined over $\CC$. Points of a variety are assumed to be closed, and we write $\mathrm{pt} := \Spec \CC$. By default, groups act on the right of varieties.

We use the abbreviation $\otimes = \otimes_{\CC}$. The dual of a vector space $V$ is denoted by $V^*$; the invariants under a group $\Gamma$ is denoted by $V^\Gamma$; the symmetric algebra is denoted by $\Sym(V)$.

Differential operators on a smooth variety $X$ are assumed to be algebraic. The structure sheaf of $X$ is denoted by $\mathscr{O}_X$. The algebra (resp.\ sheaf of algebras) of differential operators on $X$ is denoted by $D_X$ (resp.\ $\mathscr{D}_X$).

The identity connected component of a Lie group or an algebraic group $G$ is denoted by $G^\circ$; the opposite of $G$ is denoted by $G^{\mathrm{op}}$. Subgroups are always assumed to be closed.

An affine group $K$ is said to be \emph{reductive} if $K^\circ$ is a connected reductive group. Since we are in characteristic zero, this is equivalent to the linear reductivity of $K$.

Lie algebras of Lie groups or algebraic groups are denoted by Gothic letters, such as $\mathfrak{g} = \Lie G$. The universal enveloping algebra of $\mathfrak{g}$ is denoted as $U(\mathfrak{g})$, and the center of $U(\mathfrak{g})$ is denoted as $\mathcal{Z}(\mathfrak{g})$. The adjoint action of a group is denoted by $\Ad$.

Categories and functors are all $\CC$-linear. The opposite of a category $\mathcal{C}$ is denoted by $\mathcal{C}^{\mathrm{op}}$. We neglect all set-theoretic issues and employ only $1$-categories in this work, although dg-enrichment still plays a vital role.

Complexes are assumed to be cochain complexes. Given such a complex $(C, d)$ or simply $C$, we have the corresponding chain complex given by $C_n := C^{-n}$. We denote by $\cate{C}(\CC)$ (resp.\ $\cate{K}(\CC)$, $\cate{D}(\CC)$) the category of complexes of $\CC$-vector spaces (resp.\ its homotopy category, its derived category), and put the superscripts $+, -, \mathrm{b}$ to denote the full subcategories with boundedness conditions. More generally, for an abelian category $\mathcal{A}$, we have the categories $\cate{C}(\mathcal{A})$, $\cate{K}(\mathcal{A})$, $\cate{D}(\mathcal{A})$ and so forth.

For an algebra $A$ (resp.\ Lie algebra $\mathfrak{g}$), we denote by $A\dcate{Mod}$ (resp.\ $\mathfrak{g}\dcate{Mod}$) the abelian category of left $A$-modules (resp.\ $\mathfrak{g}$-modules); the opposite algebra of $A$ is denoted by $A^{\mathrm{op}}$. These notations will be generalized to dg-modules over Harish-Chandra dg-algebras in \S\ref{sec:HC-dga}.

\subsection*{Acknowledgement}
The author would like to thank Yangyang Chen, Jan Frahm and Masatoshi Kitagawa for useful conversations and their expertise. This work is supported by NSFC-11922101.

\section{Homological algebra over Harish-Chandra pairs}\label{sec:dg}
\subsection{Review of dg-algebras and dg-modules}\label{sec:review-dg}
Below is a short review about dg-modules over dg-algebras. These materials are standard, and detailed expositions can be found in \cite[Part II]{BL94} or \cite{Yek20}.

First of all, we make $\cate{C}(\CC)$ into a symmetric monoidal category by using Koszul's sign rule, namely: the tensor product of complexes $M$ and $N$ are given by $(M \otimes N)^n = \bigoplus_{i+j=n} M^i \otimes N^j$, with $d(m \otimes n) = dm \otimes n + (-1)^i m \otimes dn$, and the braiding $M \otimes N \rightiso N \otimes M$ is $m \otimes n \mapsto (-1)^{ij} n \otimes m$ on $M^i \otimes N^j$.

\begin{definition}
	If a complex $A$ carries the structure of an algebra\footnote{Also known as a \emph{monoid} in some textbooks.} in the symmetric monoidal category $\cate{C}(\CC)$, then $A$ is said to be a \emph{dg-algebra}.
\end{definition}

Multiplication in a dg-algebra is given by a morphism $A \otimes A \to A$ in $\cate{C}(\CC)$. In concrete terms, this means that there are linear maps $A^i \otimes A^j \to A^{i+j}$ for all $i, j \in \Z$, satisfying
\[ d(xy) = dx \cdot y + (-1)^i x \cdot dy \]
and are associative and unital with respect to some element in $A^0$.

It is customary to identify the complex $A$ with $\bigoplus_n A^n$ and view $d$ as an endomorphism of that vector space. Note that dg-algebras in degree zero are simply algebras.

\begin{definition}
	By the abstract formalism of algebras in symmetric monoidal categories, it makes sense to define left (resp.\ right) $A$-modules $M$, with scalar multiplication given by morphisms $A \otimes M \to M$ (resp.\ $M \otimes A \to M$), which we call left (resp. right) \emph{dg-modules} over $A$. Homomorphisms between dg-modules are defined as morphisms between complexes that respect scalar multiplications.
\end{definition}

In concrete terms, a left dg-module over $A$ is a complex $M$ together with linear maps $A^i \otimes M^j \to M^{i+j}$ for all $i, j \in \Z$, satisfying
\[ d(am) = da \cdot m + (-1)^i a \cdot dm \]
and are associative and unital. Again, it is customary to identify $M$ with the vector space $\bigoplus_n M^n$, so that $M$ is also a left $A$-module in the non-dg sense. Right dg-modules over $A$ are described similarly, and one can also define bimodules.

When $A = \CC$ (in degree zero), left and right dg-modules are nothing but complexes.

We now review the tensor product of dg-modules; see \cite[Definition 3.3.23]{Yek20} for details.

\begin{definition}\label{def:tensor-dg}
	Let $M$ be a right dg-module and $N$ be a left dg-module over a dg-algebra $A$. Let us form the vector space $M \dotimes{A} N$ in the non-dg sense; the kernel of the natural linear map
	\[ \bigoplus_n (M \otimes N)^n \simeq (\bigoplus_i M^i) \otimes (\bigoplus_j M^j) \to M \dotimes{A} N \]
	is readily seen to be a subcomplex of $M \otimes N$. In this way, we see that $M \dotimes{A} N$ comes from a complex, which will also be denoted as $M \dotimes{A} N$.
	
	If $M$ (resp.\ $N$) is a dg-bimodule over $(R, A)$ (resp.\ over $(A, S)$), where $R$ and $S$ are dg-algebras, then $M \dotimes{A} N$ inherits a natural $(R, S)$-bimodule structure.
\end{definition}

The \emph{opposite dg-algebra} $A^{\text{op}}$ of $A$ is the same complex with the multiplication $x \overset{\mathrm{op}}{\cdot} y := (-1)^{ij} yx$ for all $x \in A^i$ and $y \in A^j$.
One passes between left dg-modules over $A$ and right dg-modules over $A^{\mathrm{op}}$ by the rule $ma := (-1)^{ij} am$, for $a \in A^i$ and $m \in M^j$. Henceforth, we will mainly consider left dg-modules.

\begin{definition}
	A \emph{dg-category} is a category $\mathcal{C}$ enriched over the symmetric monoidal category $\cate{C}(\CC)$. In other words, to each pair $(X, Y)$ of objects is assigned the $\Hom$-complex $\Hom^\bullet_{\mathcal{C}}(X, Y) \in \cate{C}(\CC)$, and there are morphisms of composition in $\cate{C}(\CC)$ of the form
	\[ \Hom^\bullet_{\mathcal{C}}(Y, Z) \otimes \Hom^\bullet_{\mathcal{C}}(X, Y) \to \Hom^\bullet_{\mathcal{C}}(X, Z), \]
	subject to strict associativty and unit laws. The non-dg version $\Hom_{\mathcal{C}}(X, Y)$ can be recovered by taking $0$-cocycles in $\Hom^\bullet_{\mathcal{C}}(X, Y)$. We also write $\End^\bullet_{\mathcal{C}}(X) := \Hom^\bullet_{\mathcal{C}}(X, X)$.
	
	A \emph{dg-functor} $\mathcal{C} \to \mathcal{D}$ between dg-categories is a functor whose action on $\Hom$ sets upgrades to $\Hom$-complexes.
\end{definition}

For example, for any abelian category $\mathcal{A}$, the category $\cate{C}(\mathcal{A})$ is naturally a dg-category. In the case of $\cate{C}(\CC)$, the $\Hom$-complexes will be denoted as $\Hom^\bullet$. Let $M$ and $N$ be complexes, typical elements in $\Hom^n(M, N)$ take the form $(f^k)_{k \in \Z}$ with $f^k: M^k \to N^{k+n}$.

\begin{definition}
	Let $\mathcal{C}$ be a dg-category. Its opposite $\mathcal{C}^{\mathrm{op}}$ is the dg-category with the same objects, and $\Hom_{\mathcal{C}^{\mathrm{op}}}^\bullet(X, Y) := \Hom_{\mathcal{C}}^\bullet(Y, X)$, with morphisms of composition
	\[ \Hom_{\mathcal{C}^{\mathrm{op}}}^p(Y, Z) \otimes \Hom_{\mathcal{C}^{\mathrm{op}}}^q(X, Y) \to \Hom_{\mathcal{C}^{\mathrm{op}}}^{p+q}(X, Z) \]
	defined as $(-1)^{pq}$ times the reversed composition in $\mathcal{C}$.
\end{definition}

We now extend the definition of $\Hom$-complexes to dg-modules.

\begin{definition}\label{def:Hom-dg}
	Let $M$ and $N$ be left dg-modules over a dg-algebra $A$. Define $\Hom^\bullet_A(M, N)$ as the subcomplex of $\Hom^\bullet(M, N)$ given by
	\[ \Hom^n_A(M, N) := \left\{\begin{array}{r|l}
		(f^k)_{k \in \Z} \in \Hom^n(M, N)& \forall i, j, \; \forall a \in A^i, \; m \in M^j \\
		& f^{i+j} (am) = (-1)^{ni} a f^j(m)
	\end{array}\right\}. \]

	If $M, N, L$ are left dg-modules over $A$, we have the morphism of composition
	\[ \Hom^\bullet_A(N, L) \otimes \Hom^\bullet_A(M, N) \to \Hom^\bullet_A(M, L) \]
	in $\cate{C}(\CC)$. We write $\End^\bullet_A(M) := \Hom^\bullet_A(M, M)$.
\end{definition}

\begin{remark}
	For every complex $M$, composition makes $\End^\bullet(M)$ into a dg-algebra. It is routine to check (see eg.\ \cite[Proposition 3.3.17]{Yek20}) that given a dg-algebra $A$, to promote $M$ into a left dg-module over $A$ amounts to prescribing a homomorphism of dg-algebras
	\[ A \to \End^\bullet(M). \]
\end{remark}

\begin{definition}
	For every dg-algebra $A$, we denote the abelian category of left dg-modules over $A$ by $A\dcate{dgMod}$ in order to emphasize that its objects are complexes.
\end{definition}

\begin{proposition}
	The construction of $\Hom^\bullet_A$ upgrades $A\dcate{dgMod}$ to a dg-category.
\end{proposition}
\begin{proof}
	The $0$-cocycles in $\Hom^\bullet_A(M, N)$ are nothing but morphisms between dg-modules.
\end{proof}

For any complex $M$, its translate $M[1]$ is given by $M[1]^n = M^{n+1}$ and $d_{M[1]}^n = -d_M^{n+1}$. The translation functor on $A\dcate{dgMod}$ is given as follows (see \cite[Definition 4.1.9]{Yek20}).

\begin{definition}\label{def:translation-dg}
	Let $M$ be a left dg-module over a dg-algebra $A$. We turn $M[1]$ into a left dg-module over $A$ by setting the scalar multiplication $\cdot$ of $A$ on $M[1]$ to be
	\[ a \cdot m = (-1)^i am, \quad a \in A^i, \quad m \in M[1]^j = M^{j+1}. \]
\end{definition}

\subsection{Harish-Chandra dg-algebras}\label{sec:HC-dga}
To begin with, let $K$ be an affine algebraic group. Set $O_K := \Gamma(K, \mathscr{O}_K)$. An abstract representation $V$ of $K$, possibly infinite-dimensional, is said to be \emph{algebraic} if it arises from a $O_K$-comodule structure on $V$; in other words, $V$ is the union of finite-dimensional subrepresentations $V_i$ on which $K$ acts through a homomorphism $\sigma_i: K \to \GL(V_i)$ of algebraic groups. In general, we say $v \in V$ is $K$-algebraic if it belongs to some algebraic subrepresentation, and define
\[ V^{K\text{-alg}} := \{v \in V: \text{algebraic} \}. \]
This is the maximal $K$-algebraic subrepresentation of $V$. The assignment $V \mapsto V^{K\text{-alg}}$ is functorial, and one readily checks that
\[ V^{K\text{-alg}} = V^{K^\circ \text{-alg}}. \]

The class of algebraic representations is closed under subquotients, direct sums and tensor products. If $\sigma: K \to \Aut_{\CC}(V)$ gives rise to an algebraic representation, then its derivative at identity
\[ \dd\sigma: \mathfrak{k} \to \End_{\CC}(V) \]
makes sense: all reduces to the finite-dimensional case.

\begin{definition}
	We say that $K$ acts algebraically on a complex $M$, if $M^n$ is an algebraic representation of $K$ and $d^n: M^n \to M^{n+1}$ is $K$-equivariant, for each $n$. In particular, this induces a homomorphism $K \to \Aut_{\cate{C}(\CC)}(M)$ of abstract groups.
\end{definition}

\begin{definition}[{\cite[1.9.1]{BL95}}]\label{def:HC-dga}
	By a \emph{Harish-Chandra dg-algebra} we mean a quadruplet
	\[ (A, K, \sigma, j), \]
	or more succinctly a pair $(A, K)$, where:
	\begin{itemize}
		\item $A$ is a dg-algebra and $K$ is as above,
		\item $K$ acts algebraically on $A$, and induces a homomorphism $\sigma: K \to \Aut_{\cate{dga}}(A)$, where $\cate{dga}$ is the category of dg-algebras;
		\item $j: \mathfrak{k} \to A^0 \subset A$ is a $K$-equivariant linear map satisfying
		\begin{gather*}
			d \circ j = 0, \\
			j([\xi_1, \xi_2]) = [j(\xi_1), j(\xi_2)] := j(\xi_1) j(\xi_2) - j(\xi_2) j(\xi_1)
		\end{gather*}
		for all $\xi_1, \xi_2 \in \mathfrak{k}$;
		\item furthermore, we impose the condition that for all $\xi \in \mathfrak{k}$,
		\begin{equation*}
			(\dd\sigma)(\xi) = [j(\xi), \cdot] \;\in \End_{\CC}(A).
		\end{equation*}
	\end{itemize}
\end{definition}

A morphism $(A, K) \to (A', K')$ between Harish-Chandra dg-algebras consists of a homomorphism $\varphi: K \to K'$ between algebraic groups together with a $\varphi$-equivariant homomorphism $\psi: A \to A'$ between dg-algebras, such that
\[\begin{tikzcd}
	\mathfrak{k} \arrow[r, "{j}"] \arrow[d, "{\Lie \varphi}"'] & A \arrow[d, "\psi"] \\
	\mathfrak{k}' \arrow[r, "{j'}"'] & A'
\end{tikzcd}\]
commutes.

\begin{definition}\label{def:HC-dg-module}
	Let $(A, K)$ be a Harish-Chandra dg-algebra. A (left) dg-module over $(A, K)$ is a left dg-module $M$ over $A$, say given by a homomorphism $\alpha: A \to \End^\bullet(M)$ between dg-algebras, together with a group homomorphism $\rho: K \to \Aut_{\cate{C}(\CC)}(M)$ making $K$ act algebraically on $M$, subject to the following compatibilities:
	\begin{itemize}
		\item $\alpha(\sigma(k)a) = \rho(k) \alpha(a) \rho(k)^{-1}$ in $\End^\bullet(M)$, for all $k \in K$ and $a \in A$;
		\item $\alpha \circ j = \dd\rho$ as linear maps $\mathfrak{k} \to \End_{\cate{C}(\CC)}(M)$.
	\end{itemize}

	Homomorphisms $M \to N$ between dg-modules over $(A, K)$ are $K$-equivariant homomorphisms in $\Hom_A(M, N)$. The abelian category so obtained is denoted by $(A, K)\dcate{dgMod}$.
\end{definition}

When $K = \{1\}$, this reverts to the dg-category $A\dcate{dgMod}$ reviewed in \S\ref{sec:review-dg}.

Every morphism $(\varphi, \psi): (A, K) \to (A', K')$ between Harish-Chandra dg-algebras induces a pullback functor
\begin{equation}\label{eqn:HC-pullback}
	(\varphi, \psi)^*: (A', K')\dcate{dgMod} \to (A, K)\dcate{dgMod}.
\end{equation}
This can be naturally upgraded to a dg-functor between dg-categories.

\begin{remark}\label{rem:AK-module}
	When $A$ is simply a $\CC$-algebra, we can define $(A, K)$-modules as dg-modules over $(A, K)$ in degree zero, and the resulting abelian category will be denoted by $(A, K)\dcate{Mod}$.
\end{remark}

\subsection{Derived categories and functors}\label{sec:derived-categories}
The basic reference for what follows is \cite[1.9]{BL95}. Let $(A, K)$ be a Harish-Chandra dg-algebra. Many operations on $A\dcate{dgMod}$ carry over to $(A, K)\dcate{dgMod}$ by imposing $K$-equivariance. Below are some key examples.

\begin{description}
	\item[Translation functor] Let $M$ be a dg-module over $(A, K)$. We make $M[1]$ into a dg-module over $(A, K)$ using the dg-module structure over $A$ from Definition \ref{def:translation-dg}, and leaving the $K$-action intact.
	\item[The $\Hom$-complex] Let $M$, $N$ be dg-modules over $(A, K)$. Define $\Hom^\bullet_{A, K}(M, N)$ to be the subcomplex of $\Hom^\bullet_A(M, N)$ (see Definition \ref{def:Hom-dg}) given by
	\[ \Hom^n_{A, K}(M, N) = \left\{ (f^k)_{k \in \Z} \in \Hom^n_A(M, N): \forall k, \; f^k \;\text{is $K$-equivariant} \right\}. \]
	Morphisms between dg-modules over $(A, K)$ are nothing but $0$-cocycles in $\Hom^\bullet_{A, K}$. This turns $(A, K)\dcate{dgMod}$ into a dg-category.
	\item[Homotopy category] Two morphisms $f, g \in \Hom_{A, K}(M, N)$ are said to be homotopic if $f-g$ is a coboundary in $\Hom^0_{A, K}(M, N)$. The homotopy category $\cate{K}(A, K)$ of $(A, K)\dcate{dgMod}$ has the same objects as $(A, K)\dcate{dgMod}$, whilst
	\[ \Hom_{\cate{K}(A, K)}(M, N) := \Hm^0 \Hom^\bullet_{A, K}(M, N) \]
	for all objects $M$ and $N$. The translation functor $M \mapsto M[1]$ passes to $\cate{K}(A, K)$, and we have $\Hom_{\cate{K}(A, K)}(M, N[n]) \simeq \Hm^n \Hom^\bullet_{A, K}(M, N)$.
	
	\item[Acyclic objects] We say $M$ is acyclic if it is acyclic as an object of $\cate{C}(\CC)$.
	
	\item[Quasi-isomorphisms] If $f \in \Hom_{A, K}(M, N)$ induces isomorphisms $\Hm^n(M) \to \Hm^n(N)$ for every $n$, we say $f$ is a quasi-isomorphism. This property depends only on the image of $f$ in the homotopy category.
	
	\item[Mapping cones] Let $f \in \Hom_{A, K}(M, N)$. The mapping cone $\mathrm{Cone}(f)$ taken in $\cate{C}(\CC)$ is made into a dg-module over $(A, K)$ by recalling that $\mathrm{Cone}(f)$ is $M[1] \oplus N$ as graded vector spaces, and we let $A$ and $K$ act on $M[1]$ and $N$ in the way prescribed before. The natural morphisms
	\[ N \to \mathrm{Cone}(f) \to M[1] \]
	are morphisms in $(A, K)\dcate{dgMod}$. Note that this is a $K$-equivariant version of \cite[Definition 4.2.1]{Yek20}.
\end{description}

The constructions above make $\cate{K}(A, K)$ into a triangulated category: the distinguished triangles are the ones isomorphic to
\[ M \xrightarrow{f} N \to \mathrm{Cone}(f) \xrightarrow{+1} \]
given before, where $f: M \to N$ is any morphism in $(A, K)\dcate{dgMod}$.

\begin{definition}
	Let $\cate{D}(A, K)$ be the Verdier quotient of $\cate{K}(A, K)$ by acyclic objects; equivalently, it is obtained from $\cate{K}(A, K)$ by inverting quasi-isomorphisms. We call $\cate{D}(A, K)$ the \emph{derived category} of $(A, K)$. It inherits the triangulated structure from $\cate{K}(A, K)$.
\end{definition}

When $K=\{1\}$, the above reduces to the well-known derived category of dg-modules: see \cite[\S 10]{BL94} or \cite[Chapter 7]{Yek20}.

In order to better understand $\cate{D}(A, K)$, one has to adapt the notion of K-projective and K-injective complexes into our context. Cf.\ \cite[\S\S 10.1 --- 10.2]{Yek20}.

\begin{definition}
	Let $M$ be a dg-module over $(A, K)$.
	\begin{itemize}
		\item We say $M$ is \emph{K-injective} if
		\[ Q\;\text{is acyclic} \implies \Hom^\bullet_{A, K}(Q, M)\;\text{is acyclic}; \]
		equivalently,
		\[ Q\;\text{is acyclic} \implies \Hom_{\cate{K}(A, K)}(Q, M) = 0. \]
		\item We say $M$ is \emph{K-projective} if
		\[ Q\;\text{is acyclic} \implies \Hom^\bullet_{A, K}(M, Q) \;\text{is acyclic}; \]
		equivalently,
		\[ Q\;\text{is acyclic} \implies \Hom_{\cate{K}(A, K)}(M, Q) = 0. \]
	\end{itemize}
\end{definition}

Being K-injective (resp.\ K-projective) is thus a property within $\cate{K}(A, K)$. If $M$ is K-injective (resp.\ K-projective), then
\begin{gather*}
	\Hom_{\cate{K}(A, K)}(X, M) \simeq \Hom_{\cate{D}(A, K)}(X, M) \\
	(\text{resp.}\; \Hom_{\cate{K}(A, K)}(M, X) \simeq \Hom_{\cate{D}(A, K)}(M, X))
\end{gather*}
for all dg-module $X$ over $(A, K)$. Cf.\ \cite[Theorems 10.1.13 and 10.2.9]{Yek20}.

We now move to $t$-structure. We say $A$ is \emph{non-positively graded} if $A^n = 0$ for all $n > 0$.

\begin{proposition}
	Suppose that $A$ is non-positively graded. For every dg-module $M$ over $(A, K)$, the ``smart truncation'' $\tau^{< 0} M$ in $\cate{C}(\CC)$ is actually a dg-submodule of $M$ over $(A, K)$, thus so is $\tilde{\tau}^{\geq 0} M := M/\tau^{< 0} M$. We have the canonical distinguished triangle
	\[ \tau^{< 0} M \to M \to \tilde{\tau}^{\geq 0} M \xrightarrow{+1} \]
	in $\cate{K}(A, K)$ as well as in $\cate{D}(A, K)$.
\end{proposition}
\begin{proof}
	Routine, see \cite[1.9.4 Lemma]{BL95}.
\end{proof}


Using the fact above, we equip $\cate{D}(A, K)$ with a $t$-structure and define its full triangulated subcategories $\cate{D}^{\star}(A, K)$ where $\star \in \{+, -, \mathrm{b}\}$. They are characterized by $\Hm^n = 0$ for $n \ll 0$, $n \gg 0$ and $|n| \gg 0$ respectively.

We also denote by $\cate{K}^{\star}(A, K)$ (where $\star \in \{+, -, \mathrm{b}\}$) the full triangulated subcategories of $\cate{K}(A, K)$ consisting of bounded below, bounded above, and bounded complexes respectively, so that $\cate{D}^{\star}(A, K)$ is equivalent to the corresponding Verdier quotients.

Now comes the definition of derived functors in an abstract context. More general formulations exist and can be found in \cite[Chapter 8]{Yek20}, for example.

\begin{definition}\label{def:derived-functor}
	Let $(A, K)$ (resp.\  $(A', K')$) be Harish-Chandra dg-algebras, and let $\cate{K}$ (resp.\ $\cate{K}'$) be a full triangulated subcategory of $(A, K)\dcate{dgMod}$ (resp.\ $(A', K')\dcate{dgMod}$). Denote the functor to Verdier quotients modulo acyclic objects as $Q: \cate{K} \to \cate{D}$ and $Q': \cate{K}' \to \cate{D}'$ respectively.
	
	Consider a triangulated functor $F: \cate{K} \to \cate{K}'$. A left (resp.\ right) \emph{derived functor} of $F$ is defined to be a triangulated functor $\Lder F$ (resp.\ $\Rder F$) from $\cate{D}$ to $\cate{D}'$ which is the right (resp.\ left) Kan extension of $Q' F$ along $Q: \cate{K} \to \cate{D}$.
\end{definition}

Whenever they exist, $\Lder F$ and $\Rder F$ fit into 2-cells:
\begin{equation}\label{eqn:2-cells-der}\begin{tikzcd}
	\cate{K} \arrow[r, "F"] \arrow[d, "Q"'] & \cate{K}' \arrow[d, "{Q'}"] \\
	\cate{D} \arrow[r, "{\Lder F}"'] \arrow[Rightarrow, ru] & \cate{D}'
\end{tikzcd}\quad\begin{tikzcd}	
	\cate{K} \arrow[r, "F"] \arrow[d, "Q"'] & \cate{K}' \arrow[d, "{Q'}"] \arrow[Rightarrow, ld] \\
	\cate{D} \arrow[r, "{\Rder F}"'] & \mathcal{D}'
\end{tikzcd}\end{equation}
The morphisms $\Rightarrow$ between functors are part of the data of derived functors. The universal property of Kan extension amounts to asserting that every 2-cell of the left (resp.\ right) form as \eqref{eqn:2-cells-der} uniquely ``retracts to'' (resp.\ is uniquely ``inflated from'') the 2-cell of $\Lder F$ (resp.\ $\Rder F$).

\begin{example}
	The easiest example is the case when $F$ is \emph{exact}, i.e.\ when $F$ preserves acyclicity (equivalently, preserves quasi-isomorphisms). For such functors, $\Lder F = \Rder F$ exists and is simply induced from the universal property of Verdier quotients; furthermore, the $\Rightarrow$ in \eqref{eqn:2-cells-der} are isomorphisms.
\end{example}

The pull-back functors $(\varphi, \psi)^*$ from \eqref{eqn:HC-pullback} are exact, since they preserve the underlying complexes.

For general $F$, the left (resp.\ right) derived functor can be accessed from K-injective (resp.\ K-projective) resolutions as in the usual setting, if they exist.

\begin{definition}
	Let $M$ be a dg-module over $(A, K)$. A K-injective (resp.\ K-projective) \emph{resolution} of $M$ is a quasi-isomorphism $M \to I$ (resp.\ $P \to M$) in $(A, K)\dcate{dgMod}$ where $I$ is K-injective (resp.\ $P$ is K-projective).
	
	Let $\cate{K}$ be a full triangulated subcategory of $\cate{K}(A, K)$. If every object $M$ of $\cate{K}$ admits a K-injective (resp.\ K-projective) resolution within $\cate{K}$, then $\cate{K}$ is said to have enough K-injectives (resp.\ K-projectives).
\end{definition}

\begin{theorem}\label{prop:enough-K-general}
	Assume that $(A, K)$ be non-positively graded. Then $\cate{K}^+(A, K)$ has enough K-injectives; if we assume moreover that $K$ is reductive, then $\cate{K}^-(A, K)$ has enough K-projectives.
\end{theorem}
\begin{proof}
	The first part follows from \cite[1.15.3]{BL95}. The second part follows from \cite[\S 5.6]{Pan95}.
\end{proof}

The following is a special instance of \cite[Theorems 10.1.20 and 10.2.15]{Yek20}.

\begin{proposition}\label{prop:derived-functor-resolution}
	Consider the situation of Definition \ref{def:derived-functor}.
	\begin{itemize}
		\item If $\cate{K}$ has enough K-injectives, then $\Rder F$ exists, and $(\Rder F)(M) \simeq Q' F(I)$ if $M \to I$ is a K-injective resolution within $\cate{K}$.
		\item If $\cate{K}$ has enough K-projectives, then $\Lder F$ exists, and $(\Lder F)(M) \simeq Q' F(P)$ if $P \to M$ is a K-projective resolution within $\cate{K}$.
	\end{itemize}
\end{proposition}

As a result, one can define $\RHom_{A, K}(X, Y) \in \cate{D}(\CC)$ for $X$ bounded above and $Y$ bounded below by the familiar recipe:
\begin{itemize}
	\item either as the derived functor of $\Hom^\bullet_{A, K}(X, \cdot)$ from $(A, K)\dcate{dgMod}^+$ to $\cate{C}(\CC)$,
	\item or as the derived functor of $\Hom^\bullet_{A, K}(\cdot, Y)$ from $(A, K)\dcate{dgMod}^-$ to $\cate{C}(\CC)^{\mathrm{op}}$.
\end{itemize}
The boundedness conditions are removable if the relevant K-injective or K-projective resolutions exist. In any case, for all $n \in \Z$ we have
\begin{equation}
	\Hm^n \RHom_{A, K}(X, Y) \simeq \Hom_{\cate{D}(A, K)}(X, Y[n]) =: \Ext^n_{A, K}(X, Y). 
\end{equation}

Given functors $\cate{K} \xrightarrow{F} \cate{K}' \xrightarrow{G} \cate{K}''$, the universal property furnishes canonical morphisms
\[ \Rder (GF) \to (\Rder G)(\Rder F), \quad (\Lder G)(\Lder F) \to \Lder (GF), \]
provided that all these derived functors exist. We record the following standard result (see eg.\ \cite[Remark 8.4.31]{Yek20}) which is also immediate from Proposition \ref{prop:derived-functor-resolution}. Recall that a functor is said to be exact if it preserves acyclicity.

\begin{corollary}\label{prop:derived-functor-composite}
	Keep the notations of Definition \ref{def:derived-functor} and consider functors $\cate{K} \xrightarrow{F} \cate{K}' \xrightarrow{G} \cate{K}''$.
	\begin{itemize}
		\item Suppose $\cate{K}$ and $\cate{K}'$ have enough K-injectives. If either $F$ preserves K-injectives or $G$ is exact, then $\Rder (GF) \rightiso (\Rder G)(\Rder F)$.
		\item Suppose $\cate{K}$ and $\cate{K}'$ have enough K-projectives. If either $F$ preserves K-projectives or $G$ is exact, then $(\Lder G)(\Lder F) \rightiso \Lder (GF)$.
	\end{itemize}
\end{corollary}

The standard tool to ensure the preservation of K-injectives or K-projectives is adjunction.

\begin{proposition}\label{prop:K-injectives-adjunction}
	Let $\mathcal{C}$ and $\mathcal{C}'$ be dg-categories, each is of the form $(A, K)\dcate{dgMod}$ for some $(A, K)$. Suppose that
	\[\begin{tikzcd}
		F: \mathcal{C} \arrow[shift left, r] & \mathcal{D}: G \arrow[shift left, l]
	\end{tikzcd}\]
	is a pair of adjoint functors, and both $F$ and $G$ upgrade to dg-functors (see the discussions after \eqref{eqn:HC-pullback}). Then
	\begin{itemize}
		\item $F$ preserves K-projectives if $G$ is exact;
		\item $G$ preserves K-injectives if $F$ is exact.
	\end{itemize}
\end{proposition}
\begin{proof}
	Adjunction is described by morphisms $\eta: \identity_{\mathcal{C}} \to GF$ and $\epsilon: FG \to \identity_{\mathcal{C}'}$ satisfying the triangular equalities. As $F$ and $G$ upgrade to dg-functors, $(\eta, \epsilon)$ induces not only functorial bijections
	\[ \Hom_{\mathcal{C}'}(FX, Y) \simeq \Hom_{\mathcal{C}}(X, GY), \]
	but also functorial isomorphisms of complexes
	\[ \Hom^\bullet_{\mathcal{C}'}(FX, Y) \simeq \Hom^\bullet_{\mathcal{C}}(X, GY), \]
	i.e.\ a dg-adjunction. The assertions follow at once.
\end{proof}
 
\subsection{h-construction}\label{sec:h-cplx}
Throughout this subsection, we consider a Harish-Chandra dg-algebra $(A, K)$ with $A$ concentrated at degree zero, i.e.\ $A$ is an algebra coming with compatible homomorphisms
\[ j: \mathfrak{k} \to A, \quad \sigma: K \to \Aut_{\cate{alg}}(A). \]

Recall from Remark \ref{rem:AK-module} that the $(A, K)$-modules are simply left $A$-modules equipped with compatible $K$-actions. The dg-category of complexes of $(A, K)$-modules is denoted by $\cate{C}(A, K)$.

\begin{definition}\label{def:weak-module}
	With the conventions above, a \emph{weak $(A, K)$-module} is a $\CC$-vector space $M$ equipped with homomorphisms
	\[ \alpha: A \to \End_{\CC}(M), \quad \rho: K \to \Aut_{\CC}(M), \]
	making $M$ into a left $A$-module and an algebraic representation of $K$ respectively, such that
	\[ \alpha(\sigma(k) a) = \rho(k) \alpha(a) \rho(k)^{-1} \]
	holds in $\End_{\CC}(M)$ for all $k \in K$ and $a \in A$, i.e.\ $\alpha$ is $K$-equivariant.
\end{definition}

Weak $(A, K)$-modules form an abelian category. Given a weak $(A, K)$-module $M$, define
\begin{equation}\label{eqn:weak-AK-w}
	w := \dd\rho - \alpha \circ j: \mathfrak{k} \to \End_{\CC}(M).
\end{equation}
This map is clearly $K$-equivariant. It is also $A$-linear: indeed,
\begin{align*}
	[\alpha (j(\xi)), \alpha(a)] & = \alpha\left( [j(\xi), a] \right) \\
	& = \alpha\left( \dd\sigma(\xi) \cdot a \right) = [\dd\rho(\xi), \alpha(a)]
\end{align*}
for all $\xi \in \mathfrak{k}$ and $a \in A$; the last equality follows from the $K$-equivariance of $\alpha$. This entails $[w(\xi), \alpha(a)] = 0$.

Therefore, a weak $(A, K)$-module is an $(A, K)$-module if and only if $w = 0$.

\begin{definition}
	Denote by ${}^{\mathrm{w}} \cate{C}(A, K)$ the category of complexes of weak $(A, K)$-modules. The $\Hom$-complex in this dg-category is denoted by ${}^{\mathrm{w}} \Hom^\bullet$.
\end{definition}

Our main reference for the h-construction below is \cite{BL95}; the idea is attributed to Duflo--Vergne \cite{DV87} and Beilinson--Ginzburg in \textit{loc.\ cit.}

\begin{definition}[{\cite[1.5]{BL95}}]\label{def:h-cplx}
	An \emph{h-complex} over $(A, K)$ is a complex $C = (C, d)$ of weak $(A, K)$-modules together with a linear map
	\[ i: \mathfrak{k} \to \End^{-1}(C) := \End^{-1}_{\CC}(C), \quad \xi \mapsto i_\xi, \]
	such that
	\begin{enumerate}[(i)]
		\item $k i_\xi k^{-1} = i_{\Ad(k)\xi}$ for all $k \in K$;
		\item $i_\xi$ is $A$-linear;
		\item $i_\xi i_\eta + i_\eta i_\xi = 0$ for all $\xi, \eta \in \mathfrak{k}$;
		\item $d i_\xi + i_\xi d = w(\xi)$ (recall \eqref{eqn:weak-AK-w}).
	\end{enumerate}
	The morphisms between h-complexes are morphisms between complexes of weak $(A, K)$-modules that commute with all $i_\xi$.
\end{definition}

The homotopy condition (iv) implies that the cohomologies of an h-complex are actually $(A, K)$-modules.

We remark that h-complexes are called \emph{equivariant complexes} in \cite{Pan95, Pan05, Pan07, Ki12}.

\begin{definition}
	Denote by ${}^{\mathrm{h}} \cate{C}(A, K)$ the category of h-complexes over $(A, K)$.
\end{definition}

Every h-complex is a complex of weak $(A, K)$-modules, whence the functor ${}^{\mathrm{h}} \cate{C}(A, K) \to {}^{\mathrm{w}} \cate{C}(A, K)$. We upgrade ${}^{\mathrm{h}} \cate{C}(A, K)$ into a dg-category by defining ${}^{\mathrm{h}} \Hom^\bullet(C_1, C_2)$ as the subcomplex of ${}^{\mathrm{w}} \Hom^\bullet(C_1, C_2)$:
\[ {}^{\mathrm{h}}\Hom^n(C_1, C_2) := \left\{ (f^l)_{l \in \Z} \in {}^{\mathrm{w}} \Hom^n(C_1, C_2) : \forall l, \xi, \; f^{l-1} i_\xi = (-1)^n i_\xi f^l \right\}. \]

As usual, we say a morphism (resp.\ an object) of ${}^{\mathrm{h}} \cate{C}(A, K)$ is a quasi-isomorphism (resp.\ acyclic) if it is so as a complex. The following definition thus makes sense.

\begin{definition}
	Using ${}^{\mathrm{h}} \Hom^\bullet$, we define the homotopy category ${}^{\mathrm{h}} \cate{K}(A, K)$ of ${}^{\mathrm{h}} \cate{C}(A, K)$, and the \emph{h-derived category} ${}^{\mathrm{h}} \cate{D}(A, K)$ is defined as its Verdier quotient by acyclic complexes, or equivalently by inverting quasi-isomorphisms.
	
	The K-injective and K-projective h-complexes are also defined in this way. The derived functors in this setting are also called \emph{h-derived functors}.
\end{definition}

For an h-complex $C$, we make $C[1]$ into an h-complex by setting
\[ i': \mathfrak{k} \to \End^{-1}(C[1]), \quad \xi \mapsto i'_\xi := -i_\xi. \]
Using this, mapping cones can be defined for any morphism $f: C_1 \to C_2$ in ${}^{\mathrm{h}} \cate{C}(A, K)$. In this way, ${}^{\mathrm{h}} \cate{K}(A, K)$ and ${}^{\mathrm{h}} \cate{D}(A, K)$ become triangulated categories.
	
For every h-complex $C$, the ``smart truncation'' $\tau^{< 0} C$ as a subcomplex of weak $(A, K)$-modules is an h-complex, and so is $\tilde{\tau}^{\geq 0} C := C/\tau^{< 0} C$. Consequently, ${}^{\mathrm{h}}\cate{D}(A, K)$ is endowed with a $t$-structure, whose heart is exactly the category $(A, K)\dcate{Mod}$ of $(A, K)$-modules.

We wish to compare the h-derived category with the naive one. There is an evident dg-functor
\begin{equation}\label{eqn:h-comparison-C}
	\cate{C}((A, K)\dcate{Mod}) \to {}^{\mathrm{h}} \cate{C}(A, K)
\end{equation}
by turning complexes over $(A, K)\dcate{Mod}$ into h-complexes with $i_\xi = 0$. It induces a functor between homotopy categories and preserves acyclicity.

Recall that the derived category $\cate{D}((A, K)\dcate{Mod})$ of $(A, K)\dcate{Mod}$ is a triangulated category with $t$-structure whose heart is $(A, K)\dcate{Mod}$. For $\star \in \{\; , +, -, \mathrm{b}\}$, from \eqref{eqn:h-comparison-C} we obtain a functor
\begin{equation}\label{eqn:h-comparison}
	\alpha: \cate{D}^{\star}((A, K)\dcate{Mod}) \to {}^{\mathrm{h}} \cate{D}^{\star}(A, K).
\end{equation}
It is obviously $t$-exact.

All the assertion above can be checked by hand. Nonetheless, the formalism of Harish-Chandra dg-algebras provides a conceptually more satisfactory approach to the h-construction. We present a summary below after some preparations.

First, suppose that $A_i$ are dg-algebras for $i=1,2$. By the general theory of algebras in a symmetric monoidal category, the tensor product $A_1 \otimes A_2$ of complexes underlies a dg-algebra.

Secondly, consider the dg Lie-algebra
\[ \overline{\mathfrak{k}} := \left[ \mathfrak{k} \xrightarrow{\identity} \mathfrak{k} \right], \quad \text{degrees:}\; -1, 0, \]
whose Lie bracket on $\overline{\mathfrak{k}}^0 \otimes \overline{\mathfrak{k}}^0$ and $\overline{\mathfrak{k}}^0 \otimes \overline{\mathfrak{k}}^{-1}$ equals the $[\cdot, \cdot]$ for $\mathfrak{k}$, and is zero otherwise. Its universal enveloping dg-algebra is
\[ U(\overline{\mathfrak{k}}) = \underbracket{\bigwedge^\bullet \mathfrak{k}}_{\deg \leq 0} \otimes \underbracket{U(\mathfrak{k})}_{\deg = 0}. \]
The differential $d$ is induced from the complex $\overline{\mathfrak{k}}$; explicit formulas will be given in \S\ref{sec:std-resolution}. We let $K$ act on $U(\overline{\mathfrak{k}})$ by adjoint actions on both $\otimes$-slots.

\begin{theorem}\label{prop:h-vs-dg}
	We have the following equivalences of dg-categories.
	\begin{enumerate}[(i)]
		\item Take $B = A$. Then
		\[ (B, K)\dcate{dgMod} \simeq \cate{C}((A, K)\dcate{Mod}). \]
		\item Take $B = U(\mathfrak{k}) \otimes A$ with diagonal $K$-action, where $U(\mathfrak{k})$ is viewed as a dg-algebra concentrated at degree zero. Define $j_B: \mathfrak{k} \to B$ by
		\[ j_B(\xi) = \xi \otimes 1 + 1 \otimes j(\xi). \]
		This makes $(B, K)$ into a Harish-Chandra dg-algebra, and
		\[ (B, K)\dcate{dgMod} \simeq {}^{\mathrm{w}} \cate{C}(A, K). \]
		Specifically, the action of $\xi \in \mathfrak{k} \subset U(\mathfrak{k})$ corresponds to $w(\xi)$; see \eqref{eqn:weak-AK-w}.
		\item Take $B = U(\overline{\mathfrak{k}}) \otimes A$ with diagonal $K$-action. Define $j_B: \mathfrak{k} \to B$ by
		\[ j_B(\xi) = j^\natural(\xi) \otimes 1 + 1 \otimes j(\xi) \]
		where $j^\natural$ is the embedding $\mathfrak{k} \hookrightarrow 1 \otimes U(\mathfrak{k}) \subset U(\overline{\mathfrak{k}})$. This makes $(B, K)$ into a Harish-Chandra dg-algebra, and
		\[ (B, K)\dcate{dgMod} \simeq {}^{\mathrm{h}} \cate{C}(A, K). \]
		Specifically, the action of $\xi \in \mathfrak{k} \subset \bigwedge^\bullet \mathfrak{k}$ corresponds to $i_\xi$.
	\end{enumerate}
	All these equivalences are identity on the underlying complexes; they preserve acyclicity, quasi-isomorphisms and homotopies.
\end{theorem}
\begin{proof}
	The case (i) is trivial. Case (iii) is \cite[1.11.1]{BL95}, and (ii) is explained in \cite[p.2201]{Pan07}.
\end{proof}

\begin{remark}
	The dg-algebra $B$ above is always non-positively graded. Hence Theorem \ref{prop:enough-K-general} provides enough K-injectives (resp.\ K-projectives, assuming $K$ reductive) for the bounded below (resp.\ bounded above) homotopy categories.
\end{remark}

\begin{remark}\label{rem:h-Ext}
	The h-version of $\Ext$ functors, denoted by ${}^{\mathrm{h}} \Ext_{A, K}$, can be defined through the equivalence in Theorem \ref{prop:h-vs-dg} (iii). As the familiar $\Ext$ functor, it can be computed in terms of ${}^{\mathrm{h}} \Hom^\bullet$ and K-injective or K-projective resolutions of h-complexes.
\end{remark}

In view of these identifications, the dg-functor \eqref{eqn:h-comparison-C} corresponds to the pullback induced by the $K$-equivariant homomorphism
\[ \epsilon \otimes \identity: U(\overline{\mathfrak{k}}) \otimes A \to \CC \otimes A \simeq A \]
of dg-algebras, where $\epsilon: U(\overline{\mathfrak{k}}) \to \CC$ is the augmentation homomorphism. On the other hand, pull-back along the inclusion $U(\mathfrak{k}) \otimes A \to U(\overline{\mathfrak{k}}) \otimes A$ amounts to forgetting the datum $i$ in h-complexes.

\subsection{Adjoint functors of oblivion}\label{sec:adjoint-oblv}
Consider Harish-Chandra dg-algebras $(A, K, \sigma, j)$, $(A', K', \sigma', j')$ and a morphism $(A, K) \to (A', K')$. In most of the scenarios, the maps $A \to A'$ and $K \to K'$ will be inclusions. For this reason, the corresponding exact functor from \eqref{eqn:HC-pullback} will be named as \emph{oblivion} instead of pullback. Denote it by
\[ \mathrm{oblv}: (A', K')\dcate{dgMod} \to (A, K)\dcate{dgMod}. \]

We begin with the change of dg-algebras, i.e.\ the case $K' = K$. Consider a $K$-equivariant homomorphism $\varphi: A \to A'$ between dg-algebras such that $\varphi j = j'$. Two constructions on a dg-module $M$ over $(A, K)$ will be needed. 

\begin{itemize}
	\item Definition \ref{def:tensor-dg} affords the dg-module $A' \dotimes{A} M$ over $A'$. Let $K$ acts diagonally on it; this action is algebraic.
	\item Since $A'$ is a dg-bimodule over $(A, A')$, the $\Hom$-complex $\Hom^\bullet_A(A', M)$ is actually a left dg-module over $A'$. It carries the standard $K$-action
	\begin{equation}\label{eqn:K-adjoint-action}
		f \xmapsto{k \in K} \underbracket{k}_{M} \circ f \circ \underbracket{k^{-1}}_{A'},
	\end{equation}
	which respects differentials and $A'$-action. Now take $\Hom^\bullet_A(A', M)^{K\text{-alg}}$.
\end{itemize}

We contend that they yield dg-modules over $(A', K)$.

\begin{lemma}\label{prop:oblv-adjunction-prepr}
	Both constructions above yield dg-functors
	\[ (A, K)\dcate{dgMod} \to (A', K)\dcate{dgMod}. \]
\end{lemma}
\begin{proof}
	We first check that $A' \dotimes{A} M$ is a dg-module over $(A', K)$. Writing the actions by $K$, $A$ and $A'$ as left multiplication, we have
	\begin{align*}
		(k a'_1) \cdot (a'_2 \otimes m) & = (k a'_1)a'_2 \otimes m \\
		& = k \cdot \left( (a'_1 \cdot k^{-1} a'_2) \otimes k^{-1} m \right) \\
		& = k \cdot a'_1 \cdot k^{-1} \cdot (a'_2 \otimes m)
	\end{align*}
	for all $a'_1, a'_2 \in A'$, $k \in K$ and $m \in M$. This is the first condition in Definition \ref{def:HC-dg-module}.
	
	As for the second condition, taking derivative of the diagonal $K$-action yields
	\begin{align*}
		\dd \rho(\xi) (a' \otimes m) & = (\dd\sigma'(\xi) a') \otimes m + a' \otimes (\dd\rho(\xi) m) \\
		& = [j'(\xi), a'] \otimes m + a' \otimes (j(\xi) m) \\
		& = \left( [j'(\xi), a'] + a' j'(\xi) \right) \otimes m \\
		& = (j'(\xi)a') \otimes m = j'(\xi) \cdot (a' \otimes m)
	\end{align*}
	for all $\xi \in \mathfrak{k}$ since $j' = \varphi j$, and we are done.
	
	Similarly, one readily checks that $\Hom^\bullet_A(A', M)^{K\text{-alg}}$ satisfies these two conditions.
	
	Standard facts from the theory of dg-modules (see eg.\ \cite[\S 9.1]{Yek20}) show that they are dg-functors if we forget $K$-actions. Our case follows by imposing $K$-equivariance.
\end{proof}

\begin{proposition}\label{prop:oblv-adjunction}
	The dg-functors from Lemma \ref{prop:oblv-adjunction-prepr} fit into adjunctions
	\begin{equation*}\begin{gathered}
		\begin{tikzcd}
			A' \dotimes{A} (\cdot): (A, K)\dcate{dgMod} \arrow[shift left, r] & (A', K)\dcate{dgMod} : \mathrm{oblv}, \arrow[shift left, l]
		\end{tikzcd} \\
		\begin{tikzcd}
			\mathrm{oblv}: (A', K)\dcate{dgMod} \arrow[shift left, r] & (A, K)\dcate{dgMod}: \Hom^\bullet_A(A', \cdot)^{K\text{-alg}}. \arrow[shift left, l]
		\end{tikzcd}
	\end{gathered}\end{equation*}
\end{proposition}
\begin{proof}
	Consider the first adjunction. It is a standard fact that
	\[ \Hom_{A'}(A' \dotimes{A} M, M') \simeq \Hom_A(M, \mathrm{oblv}(M')) \;\;\text{canonically.} \]
	Adding $K$-equivariance yields the desired adjunction for $\Hom_{A', K}$. The second adjunction can be deduced in a similar way, by using standard adjunctions for dg-modules.
\end{proof}

\begin{corollary}\label{prop:oblv-adjoint-K}
	The functor $A' \dotimes{A} (\cdot)$ preserves K-projectives and $\Hom^\bullet_A(A', \cdot)^{K\text{-alg}}$ preserves K-injectives.
\end{corollary}
\begin{proof}
	Oblivion is exact, so the assertions follow from Proposition \ref{prop:K-injectives-adjunction}.
\end{proof}

Next, we describe the adjoint functors of oblivion in h-construction (Definition \ref{def:h-cplx}).

\begin{proposition}\label{prop:adjoint-h-oblv}
	Suppose that $A$ and $A'$ are both in degree zero in the circumstance of Proposition \ref{prop:oblv-adjunction}. In what follows, $M$ stands for an arbitrary h-complex over $(A, K)$, say with $i^M_\xi \in \End^{-1}(M)$ in its data ($\xi \in \mathfrak{k}$).
	\begin{enumerate}[(i)]
		\item The left adjoint of
		\[ \mathrm{oblv}: {}^{\mathrm{h}}\cate{C}(A', K) \to {}^{\mathrm{h}}\cate{C}(A, K) \]
		is the dg-functor
		\[ A' \dotimes{A} (\cdot): {}^{\mathrm{h}}\cate{C}(A, K) \to {}^{\mathrm{h}}\cate{C}(A', K), \]
		where $K$ acts diagonally on $A' \dotimes{A} M$, and the degree $-1$ endomorphisms $i^{\otimes}_\xi$ of $A' \dotimes{A} M$ are given by
		\[ i^{\otimes}_\xi(a' \otimes m) = a' \otimes i^M_\xi(m). \]

		\item The right adjoint of $\mathrm{oblv}$ is the dg-functor
		\[ \Hom^\bullet_A(A', \cdot)^{K\text{-alg}}: {}^{\mathrm{h}}\cate{C}(A, K) \to {}^{\mathrm{h}}\cate{C}(A', K), \]
		where $K$ acts à la \eqref{eqn:K-adjoint-action} on $\Hom^\bullet_A(A', M)$ and the degree $-1$ endomorphisms $i^{\Hom}_\xi$ of $\Hom^\bullet_A(A', M)$ are given by
		\[ (i_\xi^{\Hom} f)(a') = i^M_\xi f(a'), \quad f \in \Hom^n_A(A', M). \]
	\end{enumerate}
\end{proposition}
\begin{proof}
	This is essentially a combination of Theorem \ref{prop:h-vs-dg} and Proposition \ref{prop:oblv-adjunction}, since
	\begin{align*}
		(U(\overline{\mathfrak{k}}) \otimes A') \dotimes{U(\overline{\mathfrak{k}}) \otimes A} M & \leftiso A' \dotimes{A} M, \\
		\Hom^\bullet_{U(\overline{\mathfrak{k}}) \otimes A}(U(\overline{\mathfrak{k}}) \otimes A', M) & \rightiso \Hom^\bullet_A(A', M)
	\end{align*}
	as dg-modules over $A'$, and these isomorphisms are $K$-equivariant. It remains to identify $i_\xi^{\Hom}$ and $i_\xi^{\otimes}$, and this is routine.
\end{proof}

\begin{remark}
	The functor $\Hom^\bullet_A(A', \cdot)^{K\text{-alg}}$ specializes to the co-induction in \cite[1.12]{BL95} in the special case $A = \CC$.
\end{remark}

Finally, we consider change of group, i.e.\ the case $A = A'$. For the sake of simplicity, we only treat the right adjoint of oblivion in the case of h-construction.

\begin{theorem}[P.\ Pandžić]\label{prop:equivariant-Zuckerman}
	Suppose that $A$ is concentrated at degree zero. Let $T \to K$ be a homomorphism of affine groups and assume $T$ is reductive. Then there is a pair of adjoint functors
	\[\begin{tikzcd}
		\mathrm{oblv}: {}^{\mathrm{h}} \cate{C}(A, K) \arrow[shift left, r] & {}^{\mathrm{h}} \cate{C}(A, T): \Gamma^{\mathrm{eq}}_{K, T}. \arrow[shift left, l]
	\end{tikzcd}\]

	Moreover, $\Gamma^{\mathrm{eq}}_{K, T}$ upgrades to a dg-functor and induces a $t$-exact functor between homotopy categories. The functor $\Gamma^{\mathrm{eq}}_{K, T}$ is exact (i.e.\ preserves acyclicity).
\end{theorem}
\begin{proof}
	This is \cite[Theorems 3.1.2 --- 3.1.5]{Pan07}. The assertion that $\Gamma^{\mathrm{eq}}_{K, T}$ upgrades to a dg-functor is contained in the cited proofs.
\end{proof}

In \cite{Pan07}, the functors $\Gamma^{\mathrm{eq}}_{K, T}$ are called \emph{equivariant Zuckerman functors}. A comparison with the classical Zuckerman functors is given in \cite[Theorem 3.3.2]{Pan07}.

\begin{corollary}\label{prop:change-group-K-proj}
	If $T$ is reductive, then $\mathrm{oblv}: {}^{\mathrm{h}} \cate{C}(A, K) \to {}^{\mathrm{h}} \cate{C}(A, T)$ preserves K-projectives.
\end{corollary}
\begin{proof}
	Apply Proposition \ref{prop:K-injectives-adjunction}.
\end{proof}

Finally, we record the easy observation that $\mathrm{oblv}$ induces functors between h-derived categories for all $T \to K$ (by exactness), and the diagrams
\begin{equation}\label{eqn:oblv-comm}
	\begin{tikzcd}
		{}^{\mathrm{h}} \cate{C}(A, K) \arrow[r, "{\Hm^n}"] \arrow[d, "{\mathrm{oblv}}"'] & (A, K)\dcate{Mod} \arrow[d] \\
		{}^{\mathrm{h}} \cate{C}(A, T) \arrow[r, "{\Hm^n}"] & (A, T)\dcate{Mod}
	\end{tikzcd}
	\quad
	\begin{tikzcd}
		{}^{\mathrm{h}} \cate{D}(A, K) \arrow[r, "{\Hm^n}"] \arrow[d, "{\mathrm{oblv}}"'] & (A, K)\dcate{Mod} \arrow[d] \\
		{}^{\mathrm{h}} \cate{D}(A, T) \arrow[r, "{\Hm^n}"] & (A, T)\dcate{Mod}
	\end{tikzcd}
\end{equation}
commute for all $n \in \Z$.

\subsection{Invariants and co-invariants}\label{sec:inv-coinv}
Consider a Harish-Chandra dg-algebra $(A, K)$ together with
\begin{itemize}
	\item $\mathfrak{a} \subsetneq A$: a $K$-invariant dg-ideal;
	\item $N \lhd K$: a subgroup acting trivially on $A/\mathfrak{a}$.
\end{itemize}
From this we obtain a morphism $(A, K) \to (A/\mathfrak{a}, K/N)$ between Harish-Chandra dg-algebras. The pull-back dg-functor in this case is called \emph{inflation}
\[ \mathrm{Infl}^{A, K}_{A/\mathfrak{a}, K/N}: (A/\mathfrak{a}, K/N)\dcate{dgMod} \to (A, K)\dcate{dgMod}. \]

\begin{proposition}\label{prop:Inv-coInv}
	There exists a left (resp.\ right) adjoint dg-functor $\mathrm{coInv}^{A, K}_{A/\mathfrak{a}, K/N}$ (resp.\ $\mathrm{Inv}^{A, K}_{A/\mathfrak{a}, K/N}$) of $\mathrm{Infl}^{A, K}_{A/\mathfrak{a}, K/N}$.
\end{proposition}
\begin{proof}
	The construction is standard: $\mathrm{Inv}^{A, K}_{A/\mathfrak{a}, K/N}$ (resp.\ $\mathrm{coInv}^{A, K}_{A/\mathfrak{a}, K/N}$) takes the maximal quotient dg-module (resp.\ dg-submodule) on which $\mathfrak{a}$ and $N$ act trivially.
\end{proof}

The transitivity $\mathrm{Infl}^{A, K}_{A/\mathfrak{a}, K/N} = \mathrm{Infl}^{A, K}_{A/\mathfrak{a}, K} \mathrm{Infl}^{A/\mathfrak{a}, K}_{A/\mathfrak{a}, K/N}$ implies
\begin{equation}\label{eqn:Inv-coInv-transitive}\begin{aligned}
	\mathrm{Inv}^{A, K}_{A/\mathfrak{a}, K/N} & \simeq \mathrm{Inv}^{A/\mathfrak{a}, K}_{A/\mathfrak{a}, K/N} \mathrm{Inv}^{A, K}_{A/\mathfrak{a}, K}, \\
	\mathrm{coInv}^{A, K}_{A/\mathfrak{a}, K/N} & \simeq \mathrm{coInv}^{A/\mathfrak{a}, K}_{A/\mathfrak{a}, K/N} \mathrm{coInv}^{A, K}_{A/\mathfrak{a}, K}.
\end{aligned}\end{equation}

Note that the left and right adjoint of $\mathrm{Infl}^{A, K}_{A/\mathfrak{a}, K}$ have been given in \S\ref{sec:adjoint-oblv}. The right adjoint of $\mathrm{Infl}^{A, K}_{A, K/N}$ is given in Theorem \ref{prop:equivariant-Zuckerman} when $K$ is reductive.

\begin{proposition}\label{prop:Inv-coInv-K}
	The functors $\mathrm{coInv}^{A, K}_{A/\mathfrak{a}, K/N}$ (resp.\ $\mathrm{Inv}^{A, K}_{A/\mathfrak{a}, K/N}$) preserves K-projectives (resp.\ K-injectives).
\end{proposition}
\begin{proof}
	This follows from the exactness of inflation and Proposition \ref{prop:K-injectives-adjunction}.
\end{proof}

In general, $\mathrm{Inv}$ (resp.\ $\mathrm{coInv}$) is not exact, and one has to consider its right (resp.\ left) h-derived functor, provided that $(A, K)\dcate{dgMod}$ or some suitable subcategory has enough K-injectives (resp.\ K-projectives). For these h-derived functors, the adjunction to $\mathrm{Infl}^{A, K}_{A/\mathfrak{a}, K/N}$ still holds on the level of derived categories.

\begin{example}\label{eg:h-inflation}
	Let us specialize to the h-construction. Assume $A$ is in degree zero, realized as the quotient of $B := U(\overline{\mathfrak{k}}) \otimes A$ modulo the augmentation ideal (Theorem \ref{prop:h-vs-dg}). The inflation
	\[ \cate{C}(A, K) = (A, K)\dcate{dgMod} \to (B, K)\dcate{dgMod} \simeq {}^{\mathrm{h}} \cate{C}(A, K) \]
	is the obvious inclusion. The functor of invariants (resp.\ co-invariants) extracts $\bigcap_\xi (\Ker(i_\xi) \cap \Ker(w(\xi)))$ (resp.\ takes the quotient modulo $\sum_\xi (\Image(i_\xi) + \Image(w(\xi))$).
	
	The left h-derived functor $\Lder(\mathrm{coInv})$ in this setting played a major role in \cite{Pan05}.
\end{example}

\section{\texorpdfstring{$(\mathfrak{g}, K)$}{(g, K)}-modules}\label{sec:gK-mod}
\subsection{Basic definitions}\label{sec:gK-basic}
Following \cite{BL95}, we consider the following data
\begin{itemize}
	\item $\mathfrak{g}$: finite-dimensional Lie algebra,
	\item $K$: affine algebraic group,
	\item $\Ad: K \to \Aut_{\text{Lie alg.}}(\mathfrak{g})$: a homomorphism of algebraic groups,
	\item $\iota: \mathfrak{k} \to \mathfrak{g}$: inclusion of Lie algebras,
\end{itemize}
subject to the conditions below
\begin{itemize}
	\item $\iota$ is $K$-equivariant,
	\item $(\dd\Ad(\xi))(x) = [\iota(\xi), x]$ for all $\xi \in \mathfrak{k}$ and $x \in \mathfrak{g}$.
\end{itemize}

Note that no grading is put on $U(\mathfrak{g})$. As $\iota$ extends to an inclusion $U(\mathfrak{k}) \to U(\mathfrak{g})$ and $\Ad$ extends to an algebraic action on $U(\mathfrak{g})$, the following is obvious.

\begin{proposition}
	The quadruplet $(U(\mathfrak{g}), K, \Ad, \iota)$ is a Harsh-Chandra dg-algebra in the sense of Definition \ref{def:HC-dga}, concentrated at degree zero.
\end{proposition}

We shall write $(\mathfrak{g}, K)$ instead of $(U(\mathfrak{g}), K)$. Thus we obtain the notions of $(\mathfrak{g}, K)$-modules, weak $(\mathfrak{g}, K)$-modules, and h-complexes over $(\mathfrak{g}, K)$ by the general formalism in \S\ref{sec:h-cplx}.

The constructions below for tensor products and internal $\Hom$'s are taken from \cite[p.80]{Pan05}. Let $M$ and $N$ be objects of ${}^{\mathrm{h}} \cate{C}(\mathfrak{g}, K)$, say with data $i^M_\xi$ and $i^N_\xi$ for all $\xi \in \mathfrak{k}$. We make $M \otimes N$ into an h-complex by letting
\begin{align*}
	\eta (x \otimes y) & = \eta x \otimes y + x \otimes \eta y, \quad \eta \in \mathfrak{g}, \\
	k (x \otimes y) & = kx \otimes ky, \quad k \in K, \\
	i^{M \otimes N}_\xi(x \otimes y) & = i^M_\xi(x) \otimes y + (-1)^p x \otimes i^N_\xi(y), \quad \xi \in \mathfrak{k},
\end{align*}
where $x \in M^p$ and $y \in N^q$, for all $(p, q) \in \Z^2$.

Let $\mathfrak{g}$ and $K$ act on the $\Hom$-complex $\Hom^\bullet_{\CC}(M, N)$ via
\begin{gather*}
	(\eta f)(m) = \eta (f(m)) - f(\eta m), \quad \eta \in \mathfrak{g}, \\
	(kf)(m) = k(f(k^{-1} m)), \quad k \in K,
\end{gather*}
where $f \in \Hom^n_{\CC}(M, N)$, $n \in \Z$. We take $\Hom^\bullet_{\CC}(M, N)^{K\text{-alg}}$ and define
\[ (i^{\Hom}_\xi f)(m) = i^N_\xi f(m) - (-1)^n f(i^M_\xi m), \quad f \in \Hom^n_{\CC}(M, N)^{K\text{-alg}}, \quad \xi \in \mathfrak{k}. \]

\begin{definition}\label{def:internal-Hom}
	The h-complex constructed above is called the \emph{internal $\Hom$} of $M$ and $N$. By taking $N = \CC$, one arrives at the notion of \emph{contragredient} h-complexes $M^\vee$ over $(\mathfrak{g}, K)$.
\end{definition}

The next notion applies only to $(\mathfrak{g}, K)$-modules. Recall that a $\mathcal{Z}(\mathfrak{g})$-module is called \emph{locally $\mathcal{Z}(\mathfrak{g})$-finite} if it is a union of finite-dimensional submodules.

\begin{definition}\label{def:HC-module}
	A $(\mathfrak{g}, K)$-module $M$ is said to be a \emph{Harish-Chandra module} if the following two conditions hold.
	\begin{enumerate}[(i)]
		\item $M$ is finitely generated as a $\mathfrak{g}$-module;
		\item $M$ is locally $\mathcal{Z}(\mathfrak{g})$-finite.
	\end{enumerate}
\end{definition}

Harish-Chandra modules form a Serre subcategory of $(\mathfrak{g}, K)\dcate{Mod}$. Cf.\ Lemma \ref{prop:Zg-finite-ses}.

\begin{remark}
	In the applications to Lie groups, $K$ is often taken to be a compact Lie group instead, and the algebraicity of the $K$-action on $M$ translates into local $K$-finiteness; see eg.\ \cite[p.45 and (1.64)]{KV95}.
\end{remark}

From \S\ref{sec:Loc} onward, we will specialize to the case where $\mathfrak{g} = \Lie G$ for some connected reductive group $G$ and $K$ is a reductive subgroup of $G$; the maps $\iota$ and $\Ad$ will then be the obvious ones. If $K$ is a symmetric subgroup, then a $(\mathfrak{g}, K)$-module $M$ is Harish-Chandra if and only if it is finitely generated over $\mathfrak{g}$ and admissible; see \cite[Corollary 7.223]{KV95} and \cite[3.4.1 Theorem]{Wa88} after taking suitable real forms.

\subsection{Standard resolutions}\label{sec:std-resolution}
Let $\mathfrak{g}$ be a finite-dimensional Lie algebra. The standard complex $N\mathfrak{g}$ is given by
\[ \cdots \to U(\mathfrak{g}) \otimes \bigwedge^2 \mathfrak{g} \to U(\mathfrak{g}) \otimes \bigwedge^1 \mathfrak{g} \to  U(\mathfrak{g}) \otimes \bigwedge^0 \mathfrak{g} \]
in degrees $\ldots, -2, -1, 0$ and zero elsewhere. The differentials
\[ \partial_n: U(\mathfrak{g}) \otimes \bigwedge^n \mathfrak{g} \to U(\mathfrak{g}) \otimes \bigwedge^{n-1} \mathfrak{g} \]
for $n \geq 1$ are
\begin{align*}
	u \otimes (\xi_1 \otimes \cdots \otimes \xi_n) \mapsto \sum_{i=1}^n (-1)^{i+1} u\xi_i \otimes (\xi_1 \wedge \cdots \widehat{\xi_i} \cdots \wedge \xi_n ) \\
	+ \sum_{p < q} (-1)^{p+q} u \otimes ( [\xi_p, \xi_q] \wedge \cdots \widehat{\xi_p} \cdots \widehat{\xi_q} \cdots )
\end{align*}
where $\widehat{\cdots}$ means terms to be neglected.

Let $U(\mathfrak{g})$ act on each $U(\mathfrak{g}) \otimes \bigwedge^n \mathfrak{g}$ by left multiplication, and take the augmentation homomorphism
\[ \epsilon: U(\mathfrak{g}) = U(\mathfrak{g}) \otimes \bigwedge^0 \mathfrak{g}\to \CC. \]

It is well known \cite[Theorem 2.122]{KV95} that $\epsilon: N\mathfrak{g} \to \CC$ furnishes a K-projective resolution of the trivial $\mathfrak{g}$-module $\CC$.

We now consider the data $(\mathfrak{g}, K)$ of \S\ref{sec:gK-basic}. View $\CC$ as the trivial $(\mathfrak{g}, K)$-module.

\begin{theorem}[P.\ Pandžić {\cite[Theorem 3.2.6]{Pan07}}]\label{prop:N-resolution}
	Let $K$ act diagonally on each term of $N\mathfrak{g}$, and define
	\begin{align*}
		i_\xi: U(\mathfrak{g}) \otimes \bigwedge^n \mathfrak{g} & \to U(\mathfrak{g}) \otimes \bigwedge^{n+1} \mathfrak{g} \\
		u \otimes \lambda & \mapsto -u \otimes (\xi \wedge \lambda)
	\end{align*}
	for all $\xi \in \mathfrak{k}$, $u \in U(\mathfrak{g})$, $\lambda \in \bigwedge^n \mathfrak{g}$ and $n \geq 0$. Then $N\mathfrak{g}$ is an h-complex over $(\mathfrak{g}, K)$.
	
	If $K$ is reductive, then $\epsilon: N\mathfrak{g} \to \CC$ is a K-projective resolution in the sense of h-complexes.
\end{theorem}

The following result is essentially a formal consequence of Theorem \ref{prop:N-resolution}.

\begin{theorem}[P.\ Pandžić {\cite[Theorem 3.5]{Pan05}}]\label{prop:N-resolution-gen}
	Assume that $K$ is reductive. For every h-complex $M$ over $(\mathfrak{g}, K)$, the morphism $\identity \otimes \epsilon: M \otimes N\mathfrak{g} \to M$ is a K-projective resolution in the h-sense.
\end{theorem}

Therefore, the homotopy category ${}^{\mathrm{h}} \cate{K}^{\star}(\mathfrak{g}, K)$ of h-complexes have enough K-projectives for $\star \in \{\; , +, -, \mathrm{b}\}$ as long as $K$ is reductive.

The upshot is that there are canonical K-projective resolutions $M \otimes N\mathfrak{g} \to M$ for all $M$ in ${}^{\mathrm{h}} \cate{C}(\mathfrak{g}, K)$ when $K$ is reductive. Moreover, these are also K-projective resolutions in $\cate{C}(\mathfrak{g}\dcate{Mod})$; this agrees with the assertion of Corollary \ref{prop:change-group-K-proj} (take $T = \{1\}$).

\subsection{Bernstein--Lunts equivalence}
Let $(\mathfrak{g}, K)$ be as in \S\ref{sec:gK-basic}. Let
\[ \cate{C}(\mathfrak{g}, K), \quad \cate{K}(\mathfrak{g}, K), \quad \cate{D}(\mathfrak{g}, K) \]
be the category of complexes of $(\mathfrak{g}, K)$-modules, its homotopy category and derived category, respectively. Recall that they have the h-avatars
\[ {}^{\mathrm{h}} \cate{C}(\mathfrak{g}, K), \quad {}^{\mathrm{h}} \cate{K}(\mathfrak{g}, K), \quad {}^{\mathrm{h}} \cate{D}(\mathfrak{g}, K) \]
and so on for the variants with superscripts $+, -, \mathrm{b}$ to denote full subcategories with boundedness conditions. The comparison functor \eqref{eqn:h-comparison} becomes
\begin{equation}\label{eqn:h-comparison-gK}
	\alpha: \cate{D}^{\star}(\mathfrak{g}, K) \to {}^{\mathrm{h}} \cate{D}^{\star}(\mathfrak{g}, K), \quad \star \in \{\; , +, -, \mathrm{b}\}.
\end{equation}

The following equivalences are due to Bernstein--Lunts \cite{BL95} and P.\ Pandžić \cite{Pan05}.

\begin{theorem}\label{prop:BL-equiv}
	The functor $\alpha: \cate{D}^{\mathrm{b}}(\mathfrak{g}, K) \to {}^{\mathrm{h}} \cate{D}^{\mathrm{b}}(\mathfrak{g}, K)$ is an equivalence. If $K$ is reductive, then so is $\alpha: \cate{D}(\mathfrak{g}, K) \to {}^{\mathrm{h}} \cate{D}(\mathfrak{g}, K)$. In both cases, the functors are triangulated and preserve $t$-structures.
\end{theorem}
\begin{proof}
	The first part is \cite[Theorem 1.3]{BL95} and the second is \cite[Theorem 1.1]{Pan05}. Moreover, we have noted after \eqref{eqn:h-comparison} that $\alpha$ is triangulated and preserves $t$-structures.
\end{proof}

Denote the $\Ext$ functors for $\cate{C}(\mathfrak{g}, K)$ and ${}^{\mathrm{h}} \cate{C}(\mathfrak{g}, K)$ by $\Ext^\bullet_{\mathfrak{g}, K}$ and ${}^{\mathrm{h}} \Ext^\bullet_{\mathfrak{g}, K}$, respectively. See Remark \ref{rem:h-Ext}.

\begin{corollary}\label{prop:BL-equiv-Ext}
	There are canonical isomorphisms
	\[ \Ext^n_{\mathfrak{g}, K}(M, N) \rightiso {}^{\mathrm{h}} \Ext^n_{\mathfrak{g}, K}(\alpha M, \alpha N) \]
	for all objects $M, N$ of $\cate{D}^{\mathrm{b}}(\mathfrak{g}, K)$ and all $n \in \Z$, compatibly with long exact sequences for $\Ext$. If $K$ is reductive, the same holds for all objects of $\cate{D}(\mathfrak{g}, K)$.
\end{corollary}
\begin{proof}
	For the left-hand side we have $\Ext^n_{\mathfrak{g}, K}(M, N) \simeq \Hom_{\cate{D}(\mathfrak{g}, K)}(M, N[n])$, and similarly for the right-hand side.
\end{proof}

In view of these results, we will often omit $\alpha$ from the formulas, and denote the $\Ext$ functor for both $\cate{D}^{\mathrm{b}}(\mathfrak{g}, K)$ and ${}^{\mathrm{h}} \cate{D}^{\mathrm{b}}(\mathfrak{g}, K)$ by $\Ext^\bullet_{\mathfrak{g}, K}$.

\subsection{A lemma on restriction}\label{sec:restriction-lemma}
Here we consider two pairs $(\mathfrak{g}, K)$ and $(\mathfrak{h}, T)$ as in \S\ref{sec:gK-basic}, assuming:
\begin{itemize}
	\item $\mathfrak{h}$ is a Lie subalgebra of $\mathfrak{g}$;
	\item $T$ is a subgroup of $K$;
	\item the maps $K \to \Aut(\mathfrak{g})$, $T \to \Aut(\mathfrak{h})$, $\mathfrak{k} \to \mathfrak{g}$ and $\mathfrak{t} \to \mathfrak{h}$ are compatible with the inclusions above.
\end{itemize}
 
Therefore we obtain a morphism of pairs $(U(\mathfrak{h}), T) \to (U(\mathfrak{g}), K)$ in the sense of \S\ref{sec:HC-dga}, and similarly for the Harish-Chandra dg-algebras obtained by h-construction.

Below is the h-counterpart of a standard result \cite[Proposition 2.57 (c)]{KV95} about restriction (or oblivion); the proof is also similar to the one in \textit{loc.\ cit.}

\begin{lemma}\label{prop:restriction-lemma}
	Suppose that $T$ is reductive. Then the functor ${}^{\mathrm{h}}\cate{C}(\mathfrak{g}, K) \to {}^{\mathrm{h}}\cate{C}(\mathfrak{h}, T)$ is exact and preserves K-projectives.
\end{lemma}
\begin{proof}
	Exactness is known. As for the preservation of K-projectives, we break the restriction into
	\[ {}^{\mathrm{h}}\cate{C}(\mathfrak{g}, K) \to {}^{\mathrm{h}}\cate{C}(\mathfrak{g}, T) \to {}^{\mathrm{h}}\cate{C}(\mathfrak{h}, T). \]
	
	The first step preserves K-projectives by Corollary \ref{prop:change-group-K-proj}. In view of Proposition \ref{prop:K-injectives-adjunction}, it suffices to show that ${}^{\mathrm{h}}\cate{C}(\mathfrak{g}, T) \to {}^{\mathrm{h}}\cate{C}(\mathfrak{h}, T)$ has a right adjoint which is exact and upgrades to a dg-functor.
	
	The desired right adjoint dg-functor is given by Proposition \ref{prop:adjoint-h-oblv} (ii), namely
	\[ \Hom^\bullet_{U(\mathfrak{h})}(U(\mathfrak{g}), \cdot)^{T\text{-alg}}: {}^{\mathrm{h}}\cate{C}(\mathfrak{h}, T) \to {}^{\mathrm{h}}\cate{C}(\mathfrak{g}, T). \]
	
	It remains to show that $\Hom^\bullet_{U(\mathfrak{h})}(U(\mathfrak{g}), M)^{T\text{-alg}}$ is acyclic for every acyclic h-complex $M$ over $(\mathfrak{h}, T)$.
	
	By reductivity, we may choose an $T^\circ$-invariant subspace $\mathfrak{q} \subset \mathfrak{g}$ under $\Ad$, such that $\mathfrak{g} = \mathfrak{h} \oplus \mathfrak{q}$. Next, apply \cite[Lemma 2.56]{KV95} (with $L = \{1\}$ and swapping left and right) to obtain an isomorphism
	\[ U(\mathfrak{h}) \otimes \Sym(\mathfrak{q}) \rightiso U(\mathfrak{g}) \quad \text{as $\mathfrak{h}$-modules.} \]
	In \textit{loc.\ cit.} the map is $u \otimes v \mapsto u \sigma(v)$ where $\sigma: \Sym(\mathfrak{q}) \to U(\mathfrak{g})$ is symmetrization. Hence it is also $T^\circ$-equivariant.

	Hence there is a $T^\circ$-equivariant (relative to \eqref{eqn:K-adjoint-action}) isomorphism in $\cate{C}(\CC)$:
	\begin{align*}
		\Hom^\bullet_{U(\mathfrak{h})}(U(\mathfrak{g}), M)^{T^\circ\text{-alg}} & \rightiso \Hom^\bullet_{\CC}(\Sym(\mathfrak{q}), M)^{T^\circ\text{-alg}} \\
		\varphi & \mapsto \varphi|_{1 \otimes \Sym(\mathfrak{q})}.
	\end{align*}
	
	Consider any algebraic representation $W$ of $T^\circ$ and a short exact sequence $0 \to V' \to V \to V'' \to 0$ of algebraic representations of $T^\circ$. Taking a compact real form of $T^\circ$ and passing from algebraic to locally $T^\circ(\R)$-finite modules, by \cite[Proposition 1.18b]{KV95} the short exact sequence splits. As a consequence,
	\[ 0 \to \Hom_{\CC}(W, V')^{T^\circ\text{-alg}} \to \Hom_{\CC}(W, V)^{T^\circ\text{-alg}} \to \Hom_{\CC}(W, V'')^{T^\circ\text{-alg}} \to 0 \]
	is acyclic in $\cate{C}(\CC)$.
	
	From this fact, $\Hom^\bullet_{\CC}(\Sym(\mathfrak{q}), M)^{T^\circ\text{-alg}}$ and $\Hom^\bullet_{U(\mathfrak{h})}(U(\mathfrak{g}), M)^{T^\circ\text{-alg}}$ are seen to be acyclic. It remains to recall that $(\cdots)^{T\text{-alg}} \simeq (\cdots)^{T^\circ\text{-alg}}$.
\end{proof}

\section{\texorpdfstring{$D$}{D}-modules}\label{sec:Dmod}
\subsection{Basic definitions}\label{sec:D-basic}
Let $K$ be an affine algebraic group. A smooth variety $X$ is said to be a \emph{$K$-variety} if $K$ acts algebraically on the right of $X$. Morphisms of $K$-varieties are defined to be $K$-equivariant morphisms of varieties.

Let $X$ be a smooth variety. Unless otherwise specified, $D_X$-modules (resp.\ $\mathscr{D}_X$-modules) will mean left modules (resp.\ $\mathscr{O}_X$-quasi-coherent left modules). We denote by
\[ \cate{D}^{\star}(X), \quad \star \in \{\;, +, -, \mathrm{b}\} \]
the derived category of $\mathscr{D}_X$-modules with boundedness condition $\star$; for affine $X$, they are just $\cate{D}^{\star}(D_X\dcate{Mod})$.

If $X$ is a smooth affine $K$-variety, then we obtain the homomorphism
\[ j: U(\mathfrak{k}) \to D_X \]
of algebras, mapping $\xi \in \mathfrak{k}$ to the vector field on $X$ generated by $\xi$. On the other hand, $K$ acts on $D_X$ in the following way. The right $K$-action on $X$ induces a left $K$-action on regular functions; let $P \in D_X$ and $k \in K$, then ${}^k P \in D_X$ is the operator
\[ \varphi \mapsto k (P (k^{-1} \varphi)) \]
for all regular function $\varphi$. It belongs to $D_X$: in fact ${}^k P$ is just the transport of structure by $k$ acted on $P$. Denote by $\sigma: K \to \Aut(D_X)$ the resulting homomorphism.

\begin{proposition}
	Give a smooth affine $K$-variety $X$, the quadruplet $(D_X, K, \sigma, j)$ is a Harsh-Chandra dg-algebra in the sense of Definition \ref{def:HC-dga}, concentrated at degree zero.
\end{proposition}
\begin{proof}
	This is stated in \cite[2.1]{BL95} (see the references therein). See also \cite{Li22}.
\end{proof}

Fix a smooth affine $K$-variety $X$ hereafter. We obtain the notions of $(D_X, K)$-modules, weak $(D_X, K)$-modules, and h-complexes over $(D_X, K)$ by the general formalism in \S\ref{sec:h-cplx}. In particular, we obtain the categories
\[ {}^{\mathrm{h}}\cate{C}(D_X, K), \quad {}^{\mathrm{h}}\cate{K}(D_X, K), \quad {}^{\mathrm{h}}\cate{D}(D_X, K) \]
as well as the versions with boundedness conditions.

In comparison with the standard terminologies, we have:
\begin{itemize}
	\item $(D_X, K)$-modules are exactly the (strongly) \emph{$K$-equivariant $D_X$-modules};
	\item weak $(D_X, K)$-modules are exactly the  \emph{weakly $K$-equivariant $D_X$-modules}.
\end{itemize}
See also the explanations after \cite[Definition 2.2]{Li22}.

For all subgroup $T$ of $K$ and $n \in \Z$, it is clear that
\[\begin{tikzcd}
	{}^{\mathrm{h}} \cate{D}(D_X, K) \arrow[r] \arrow[d, "{\Hm^n}"'] & {}^{\mathrm{h}}\cate{D}(D_X, T) \arrow[d, "{\Hm^n}"] \\
	(D_X, K)\dcate{Mod} \arrow[r] & (D_X, T)\dcate{Mod}
\end{tikzcd}\]
commutes, where the horizontal arrows are the evident ones, and the upper one is $t$-exact triangulated. Taking $T = \{1\}$, the second column becomes $\cate{D}(X) \xrightarrow{\Hm^n} D_X\dcate{Mod}$.

\begin{remark}
	If $x \in X$ is fixed by $K$, we can also talk about $(D_{X, x}, K)$-modules, etc., where $D_{X, x}$ is the stalk of $D_X$ at $x$. Taking stalks yield the dg-functors
	\[ {}^{\mathrm{h}}\cate{C}(D_X, K) \to {}^{\mathrm{h}}\cate{C}(D_{X, x}, K) \]
	and so forth. In fact it is $D_{X, x} \dotimes{D_X} (\cdot)$; it preserves K-projectives by Corollary \ref{prop:oblv-adjoint-K}.
\end{remark}

\begin{remark}
	The assumption that $X$ is affine is artificial here. It is clear that the whole formalism can be sheafified. For any smooth $K$-variety $X$, there is a theory of $(\mathscr{D}_X, K)$-modules, weak $(\mathscr{D}_X, K)$-modules, h-complexes over $(\mathscr{D}_X, K)$ and the corresponding derived categories. See \cite[2.16.1]{BL95}, or \cite{Ki12} for partial flag varieties. Nevertheless, this level of generality is needed in this work.

	Several constructions from the theory of $D$-modules generalize to the h-setting. For example, the same ideas from \S\ref{sec:std-resolution} upgrades the \emph{Spencer complex} \cite[(1.5.6)]{HTT08}
	\[ \cdots \to \mathscr{D}_X \dotimes{\mathscr{O}_X} \bigwedge^2 \mathscr{T}_X \to \mathscr{D}_X \dotimes{\mathscr{O}_X} \bigwedge^1 \mathscr{T}_X \to \mathscr{D}_X \dotimes{\mathscr{O}_X} \bigwedge^0 \mathscr{T}_X \]
	into an h-complex over $(\mathscr{D}_X, K)$, where $\mathscr{T}_X$ stands for the tangent sheaf of $X$. Denote this h-complex as $NX$, then the morphism $\mathscr{D}_X \to \mathscr{O}_X$ given by $D \mapsto D(1)$ yields a quasi-isomorphism $NX \to \mathscr{O}_X$ in ${}^{\mathrm{h}} \cate{C}(\mathscr{D}_X, K)$.
\end{remark}

\subsection{Inverse images}\label{sec:inverse-image}
There are two basic operations on derived categories of $D$-modules: inverse image (suitably shifted) and de Rham push-forward. The goal here is to define inverse images for h-complexes on smooth affine $K$-varieties. We do not consider de Rham push-forward in this work.

For every morphism $f: X \to Y$ between smooth varieties, we set
\[ \mathrm{rd}(f) := \dim X - \dim Y. \]

Assuming that there are enough flat $\mathscr{D}_Y$-modules (eg.\ when $Y$ is quasi-projective), the basic theory of $D$-modules provides a functor
\[ f^\bullet = f^![-\mathrm{rd}(f)] : \cate{D}^-(Y) \to \cate{D}^-(X). \]
\begin{itemize}
	\item The $f^\bullet$ above is the inverse image functor denoted by $\Lder f^*$ in \cite[p.33]{HTT08}, the left derived functor of $\mathscr{D}_{X \to Y} \dotimes{f^{-1} \mathscr{D}_Y} f^{-1}(\cdot)$; 
	\item The $f^!$ is denoted by $f^\dagger$ in \textit{loc.\ cit.}, and $f^!$ corresponds to the $!$-pullback under Riemann--Hilbert correspondence when applied to regular holonomic complexes.
\end{itemize}

The effect of $\mathscr{D}_{X \to Y} \dotimes{f^{-1} \mathscr{D}_Y} f^{-1}(\cdot)$ is to take the pullback of a $\mathscr{D}_Y$-module as a quasi-coherent sheaf, and then equip it with a natural $\mathscr{D}_X$-module structure.

If $f$ is a morphism between $K$-varieties, then $\mathscr{D}_{X \to Y}$ is $K$-equivariant, and this functor can be lifted to the level of h-complexes, cf.\ Proposition \ref{prop:adjoint-h-oblv} (i). The point is to take its left h-derived functor. This splits into two cases.

\begin{description}
	\item[Smoooth inverse image] Suppose $f$ is smooth, then $\mathscr{D}_{X \to Y} \dotimes{f^{-1} \mathscr{D}_Y} f^{-1}(\cdot)$ preserves acyclic objects, hence it induces a $t$-exact functor
	\[ f^\bullet: {}^{\mathrm{h}} \cate{D}(D_Y, K) \to {}^{\mathrm{h}} \cate{D}(D_X, K) \]
	as well as $f^! := f^\bullet[\mathrm{rd}(f)]$.
	
	\item[General inverse image] For general $f$, we shall assume $K$ is reductive. Theorem \ref{prop:enough-K-general} ensures that ${}^{\mathrm{h}} \cate{C}(D_Y, K)$ has enough K-projectives. Hence we may define the left h-derived functor
	\[ f^\bullet: {}^{\mathrm{h}} \cate{D}^-(D_Y, K) \to {}^{\mathrm{h}} \cate{D}^-(D_X, K) \]
	and $f^! := f^\bullet[\mathrm{rd}(f)]$. Note that $f^\bullet$ can be computed from K-flat resolutions, and K-projective implies K-flat (cf.\ \cite[Proposition 10.3.4]{Yek20}).
\end{description}

Clearly, these two definitions agree in overlapping cases. The following results are trivial for smooth inverse images, thus we state them only for the general ones.

\begin{lemma}\label{prop:pullback-composite}
	Suppose $K$ is reductive. Consider morphisms $X \xrightarrow{f} Y \xrightarrow{g} Z$ of smooth affine $K$-varieties. There is a canonical isomorphism $(gf)^\bullet \simeq f^\bullet g^\bullet$.
\end{lemma}
\begin{proof}
	Analogous to the non-equivariant version; see the proof of \cite[Propositionp 1.5.11]{HTT08}. It boils down to standard properties of derived tensor products, but for our setting of h-complexes, the dg-counterparts are needed (cf.\ \cite[\S 12.3]{Yek20}).
\end{proof}

\begin{lemma}\label{prop:pullback-amplitude}
	Suppose $K$ is reductive. Let $f: X \to Y$ be a morphism between smooth affine $K$-varieties. Then $f^\bullet$ restricts to ${}^{\mathrm{h}} \cate{D}^{\mathrm{b}}(D_Y, K) \to {}^{\mathrm{h}} \cate{D}^{\mathrm{b}}(D_X, K)$.
\end{lemma}
\begin{proof}
	It suffices to show that $f^\bullet$ has bounded amplitude. The functor of forgetting $K$-equivariance
	\[ {}^{\mathrm{h}} \cate{C}(D_Y, K) \to {}^{\mathrm{h}} \cate{C}(D_Y, \{1\}) = \cate{C}(D_Y\dcate{Mod}) \]
	is exact and preserves K-projectives by Corollary \ref{prop:change-group-K-proj}. Ditto for $Y$ replaced by $X$. Therefore the left square of the diagram
	\[\begin{tikzcd}
		{}^{\mathrm{h}} \cate{D}^-(D_Y, K) \arrow[d] \arrow[r, "{f^\bullet}"] & {}^{\mathrm{h}} \cate{D}^-(D_X, K) \arrow[r, "{\Hm^n}"] \arrow[d] & (D_X, K)\dcate{Mod} \arrow[d] \\
		\cate{D}^-(Y) \arrow[r, "{f^\bullet}"'] & \cate{D}^-(X) \arrow[r, "{\Hm^n}"'] & D_X\dcate{Mod}
	\end{tikzcd}\]
	commutes up to a canonical isomorphism --- to see this, compute left h-derived functors via Corollary \ref{prop:derived-functor-composite}. So does the right square by \eqref{eqn:oblv-comm}. The amplitude of the upper $f^\bullet$ can thus be bounded by that of the lower one, which is finite by \cite[Corollary 1.4.20]{HTT08}.
\end{proof}

The non-affine case is not needed in this work, and is left to the interested reader.

\subsection{Comparison with the equivariant derived category}
To begin with, suppose $X$ is a smooth $K$-variety.

\begin{definition}
	Let $\cate{D}^{\mathrm{b}}_K(X)$ be the \emph{bounded equivariant derived category} of $\mathscr{D}_X$-modules under $K$-action defined by Bernstein and Lunts \cite[4.2]{BL94}. It is a triangulated category endowed with a $t$-structure, whose heart is the category of $K$-equivariant $\mathscr{D}_X$-modules.
\end{definition}

For details about the definition of $\cate{D}^{\mathrm{b}}_K(X)$, we refer to \cite{BL94}, \cite[2.11]{BL95}, or to the more systematic \cite[Chapter 6]{Ac21} in the setting of constructible sheaves. Below is only a sketch.

\begin{definition}
	A \emph{resolution} of a $K$-variety $X$ is a pair $(P, p)$ where $P$ is a $K$-torsor and $p: P \to X$ is a smooth affine $K$-equivariant morphism. We also set $\overline{P} := P/K$. Morphisms $f: (Q, q) \to (P, p)$ are defined as commutative diagrams
	\[\begin{tikzcd}
		Q \arrow[r, "f"] \arrow[rd, "q"'] & P \arrow[d, "p"] \\
		& X
	\end{tikzcd}\]
	where $f$ is a smooth $K$-equivariant morphism. From $f$ we obtain a smooth morphism $\overline{f}: \overline{Q} \to \overline{P}$.
\end{definition}

\begin{example}
	The \emph{trivial resolution} of $X$ is $P := X \times K$ and $p := \mathrm{pr}_1$, where $K$ acts on $P$ by $(x, h)k = (xk, k^{-1}h)$. Then $X \rightiso \overline{P}$ by mapping $x \in X$ to the orbit of $(x, 1)$.
\end{example}

For the construction of sufficiently many resolutions of $X$ (namely, $n$-acyclic ones \cite[Definition 6.1.18]{Ac21} for every $n \in \Z_{\geq 0}$), we refer to \cite[\S 6.1]{Ac21}.

\begin{itemize}
	\item The objects in $\cate{D}^{\mathrm{b}}_K(X)$ are collections $M$ of objects $M_P$ in $\cate{D}^{\mathrm{b}}(\overline{P})$ for every resolution $(P, p)$, together with isomorphisms
	\[ \alpha_f: \overline{f}^\bullet M_P \rightiso M_Q \]
	for each morphism $f: (Q, q) \to (P, p)$, compatibly with compositions in the sense that
	\[ \alpha_g \circ \overline{g}^\bullet(\alpha_f) = \alpha_{fg}. \]
	\item The morphisms $M \to N$ in $\cate{D}^{\mathrm{b}}_K(X)$ are collections of morphisms $\varphi_P: M_P \to N_P$ compatibly with various $\alpha_f$.
\end{itemize}

For the description of the following structures of $\cate{D}^{\mathrm{b}}_K(X)$, see also see \cite[pp.286--287]{Ac21}; note that one works with constructible sheaves in \textit{loc.\ cit.} and the smooth inverse images differ by a shift by $\mathrm{rd}(f)$.

\begin{description}
	\item[Forgetting equivariance] The functor $\mathbf{oblv}: \cate{D}^{\mathrm{b}}_K(X) \to \cate{D}^{\mathrm{b}}(X)$ sends an object $M$ to $M_{X \times K}$, by taking the trivial resolution.
	\item[Translation functor] Define $M[1]$ to be the collection $(M_P[1])_{(P, p)}$.
	\item[Distinguished triangles] A triangle $M_1 \to M_2 \to M_3 \xrightarrow{+1}$ is said to be distinguished if $M_{1, P} \to M_{2, P} \to M_{3, P} \xrightarrow{+1}$ is, for every resolution $(P, p)$.
	\item[Bounded $t$-structure] For any bounded interval $I$, an object $M$ lies in $\cate{D}^I_K(X)$ if and only if $\mathbf{oblv}(M)$ lies in $\cate{D}^I(X)$.
\end{description}

More generally, let $T \subset K$ be a subgroup. There is a $t$-exact functor
\[ \mathbf{oblv}^K_T: \cate{D}^{\mathrm{b}}_K(X) \to \cate{D}^{\mathrm{b}}_T(X) \]
whose definition à la \cite[Definition 6.5.2]{Ac21} is sketched below.

Any resolution $(P, p)$ can be viewed as a resolution of $X$ as a $T$-variety. Let $\cate{D}^{\mathrm{b}}_{T \subset K}(X)$ be the category defined in the same way as $\cate{D}^{\mathrm{b}}_T(X)$, but allowing only the resolutions so obtained. This gives rise to a functor
\[ \mathbf{oblv}^K_{T \subset K}: \cate{D}^{\mathrm{b}}_K(X) \to \cate{D}^{\mathrm{b}}_{T \subset K}(X). \]
On the other hand, \cite[Lemmas 6.4.7, 6.4.8]{Ac21} realizes $\cate{D}^{\mathrm{b}}_T(X)$ as a full subcategory of $\cate{D}^{\mathrm{b}}_{T \subset K}(X)$. All these constructions being triangulated and $t$-exact, one can verify that $\mathbf{oblv}^K_{T \subset K}$ lands in $\cate{D}^{\mathrm{b}}_T(X)$ by checking on hearts, as in \cite[p.292]{Ac21}. This yields $\mathbf{oblv}^K_T$.

Moreover, it is proved in \textit{loc.\ cit.} that $\cate{D}^{\mathrm{b}}_{\{1\}}(X) \simeq \cate{D}^{\mathrm{b}}(X)$ and $\mathbf{oblv}^K_{\{1\}} \simeq \mathbf{oblv}$.

\begin{theorem}[A.\ Beilinson, see {\cite[Theorem 2.13]{BL95}}]\label{prop:Beilinson-equiv}
	Let $X$ be a smooth affine $K$-variety. There is an equivalence between triangulated categories
	\[ \varepsilon: {}^{\mathrm{h}}\cate{D}^{\mathrm{b}}(D_X, K) \to \cate{D}^{\mathrm{b}}_K(X) \]
	with the following properties.
	\begin{enumerate}[(i)]
		\item For every subgroup $T \subset K$, the diagram
		\[\begin{tikzcd}
			{}^{\mathrm{h}}\cate{D}^{\mathrm{b}}(D_X, K) \arrow[d] \arrow[r, "\varepsilon"] & \cate{D}^{\mathrm{b}}_K(X) \arrow[d, "{\mathbf{oblv}^K_T}"] \\
			{}^{\mathrm{h}}\cate{D}^{\mathrm{b}}(D_X, T) \arrow[r, "\varepsilon"'] & \cate{D}^{\mathrm{b}}_T(X)
		\end{tikzcd}\]
		commutes up to isomorphism; in particular, $\varepsilon$ commutes with ${}^{\mathrm{h}}\cate{D}^{\mathrm{b}}(D_X, K) \to \cate{D}^{\mathrm{b}}(X)$ and $\mathbf{oblv}$.
		\item It is $t$-exact.
		\item It induces identity on hearts, which are the category of equivariant $D_X$-modules on both sides.
	\end{enumerate}
\end{theorem}
\begin{proof}
	The equivalence is stated in \textit{loc.\ cit.}; the case here is simpler since $X$ is affine and we do not consider monodromic structures yet. Let us sketch the definition of $\varepsilon$ now. Given a resolution $(P, p)$, consider the functors
	\begin{equation}\label{eqn:epsilon-construction}
		\begin{tikzcd}
			{}^{\mathrm{h}} \cate{D}^{\mathrm{b}}(D_X, K) \arrow[r, "{p^\bullet}"] & {}^{\mathrm{h}} \cate{D}^{\mathrm{b}}(D_P, K) & \cate{D}^{\mathrm{b}}((D_P, K)\dcate{Mod}) \arrow[l, "\sim"'] & \cate{D}^{\mathrm{b}}(\overline{P}) \arrow[l, "\sim"'] .
		\end{tikzcd}
	\end{equation}

	The middle equivalence is the functor $\alpha$ in the step 2 of proof of \cite[Theorem 2.13]{BL95}. The rightmost equivalence comes from $D_{\overline{P}}\dcate{Mod} \simeq (D_P, K)\dcate{Mod}$, as $P \to \overline{P}$ is a $K$-torsor; it is also the functor in the step 1 of the cited proof.
	
	Suppose that an object $M$ of ${}^{\mathrm{h}} \cate{D}^{\mathrm{b}}(D_X, K)$ is given. When $(P, p)$ varies, we obtain a compatible collection of objects of $\cate{D}^{\mathrm{b}}(\overline{P})$. Hence \eqref{eqn:epsilon-construction} yields the functor $\varepsilon$. It is triangulated since all the functors in \eqref{eqn:epsilon-construction} are.
	
	Consider the property (i). Given $T \subset K$, define $\varepsilon_{T \subset K}: {}^{\mathrm{h}}\cate{D}^{\mathrm{b}}(D_X, T) \to \cate{D}^{\mathrm{b}}_{T \subset K}(X)$ by the same recipe \eqref{eqn:epsilon-construction}, but allowing only resolutions restricted from $K$. By the earlier review on $\mathbf{oblv}^K_T$, it suffices to commute
	\[\begin{tikzcd}
		{}^{\mathrm{h}}\cate{D}^{\mathrm{b}}(D_X, K) \arrow[d] \arrow[r, "\varepsilon"] & \cate{D}^{\mathrm{b}}_K(X) \arrow[d, "{\mathbf{oblv}^K_{T \subset K}}"] \\
		{}^{\mathrm{h}}\cate{D}^{\mathrm{b}}(D_X, T) \arrow[r, "{\varepsilon_{T \subset K}}"'] & \cate{D}^{\mathrm{b}}_{T \subset K}(X)
	\end{tikzcd}\]
	up to isomorphism. But this is immediate by construction.
	
	As $\varepsilon$ is seen to commute with ${}^{\mathrm{h}}\cate{D}^{\mathrm{b}}(D_X, K) \to \cate{D}^{\mathrm{b}}(X)$ and $\mathbf{oblv}$, it is also $t$-exact, whence (ii).
	
	Consider (iii). If $M$ is an equivariant $D_X$-module, then its image under \eqref{eqn:epsilon-construction} is the descent to $\overline{P}$ of its inverse image in $(D_P, K)\dcate{Mod}$. When $(P, p)$ varies, such a collection also gives rise to an object in the heart of $\cate{D}^{\mathrm{b}}_K(X)$ that corresponds to $M$; for details, see \cite[Theorem 6.4.10]{Ac21} and its proof. Beware of the shifts in the constructible setting considered in \cite{Ac21}.
\end{proof}

\begin{corollary}\label{prop:Beilinson-equiv-Ext}
	Let $X$ be a smooth affine $K$-variety. There are canonical isomorphisms
	\[ {}^{\mathrm{h}} \Ext^n_{D_X, K}(M, N) \rightiso \Ext^n_{\cate{D}^{\mathrm{b}}_K(X)}(\varepsilon M, \varepsilon N) \]
	for all objects $M, N$ of ${}^{\mathrm{h}}\cate{D}^{\mathrm{b}}(D_X, K)$ and all $n \in \Z$, compatibly with long exact sequences for $\Ext$.
\end{corollary}
\begin{proof}
	Analogous to Corollary \ref{prop:BL-equiv-Ext}.
\end{proof}

The inverse image functor is also defined on equivariant derived categories, denoted by the same symbol $f^\bullet$. They match the h-counterpart under $\varepsilon$ by the next result.

\begin{proposition}\label{prop:inverse-image-compatibility}
	Assume that $K$ is reductive. Let $f: X \to Y$ be a morphism between smooth affine $K$-varieties. The following diagram commutes up to a canonical isomorphism:
	\[\begin{tikzcd}
		{}^{\mathrm{h}} \cate{D}^{\mathrm{b}}(D_Y, K) \arrow[r, "\varepsilon"] \arrow[d, "{f^\bullet}"'] & \cate{D}^{\mathrm{b}}_K(Y) \arrow[d, "{f^\bullet}"] \\
		{}^{\mathrm{h}} \cate{D}^{\mathrm{b}}(D_X, K) \arrow[r, "\varepsilon"'] & \cate{D}^{\mathrm{b}}_K(X) .
	\end{tikzcd}\]
\end{proposition}
\begin{proof}
	Given a resolution $q: Q \to Y$, we form the Cartesian square
	\[\begin{tikzcd}
		P \arrow[r, "{\tilde{f}}"] \arrow[d, "p"'] & Q \arrow[d, "q"] \\
		X \arrow[r, "f"'] & Y
	\end{tikzcd}\]
	then $p$ is also a resolution. Let $\overline{f}: \overline{P} \to \overline{Q}$ denote the induced morphism. Then
	\begin{equation*}
		\mathrm{rd}(p) = \mathrm{rd}(q), \quad \mathrm{rd}(\overline{f}) = \mathrm{rd}(f) = \mathrm{rd}(\tilde{f}).
	\end{equation*}
	
	Recall the construction \eqref{eqn:epsilon-construction} of $\varepsilon$, and consider the following diagram
	\[\begin{tikzcd}
		{}^{\mathrm{h}} \cate{D}^{\mathrm{b}}(D_Y, K) \arrow[d, "{f^\bullet}"'] \arrow[r, "{q^\bullet}"] & {}^{\mathrm{h}} \cate{D}^{\mathrm{b}}(D_Q, K) \arrow[d, "{\tilde{f}^\bullet}"] & \cate{D}^{\mathrm{b}}((D_Q, K)\dcate{Mod}) \arrow[l, "\sim"'] & \cate{D}^{\mathrm{b}}(\overline{Q}) \arrow[l, "\sim"'] \arrow[d, "{\overline{f}^\bullet}"] \\
		{}^{\mathrm{h}} \cate{D}^{\mathrm{b}}(D_X, K) \arrow[r, "{p^\bullet}"'] & {}^{\mathrm{h}} \cate{D}^{\mathrm{b}}(D_P, K) & \cate{D}^{\mathrm{b}}((D_P, K)\dcate{Mod}) \arrow[l, "\sim"] & \cate{D}^{\mathrm{b}}(\overline{P}) \arrow[l, "\sim"] .
	\end{tikzcd}\]
	
	The left square commutes by Lemma \ref{prop:pullback-composite}. As for the right rectangle, consider a bounded complex $\overline{M}$ of $D_{\overline{Q}}$-modules. We may and do take a quasi-isomorphism $\overline{F} \to \overline{M}$ such that $\overline{F}$ is K-flat when viewed over $O_{\overline{Q}} := \Gamma(\overline{Q}, \mathscr{O}_{\overline{Q}})$. Then their inverse images as complexes of $(D_Q, K)$-modules $F \to M$ is still a quasi-isomorphism, and $F$ is K-flat over $O_Q$ since
	\[ F \dotimes{O_Q} (\cdot) \simeq \overline{F} \dotimes{O_{\overline{Q}}} O_Q \dotimes{O_Q} (\cdot) \simeq \overline{F} \dotimes{O_{\overline{Q}}} (\cdot). \]
	
	The image of $F$ in ${}^{\mathrm{h}} \cate{C}(D_Q, K)$ can be used to compute $\tilde{f}^\bullet$, whilst $\overline{F}$ can be used to compute $\overline{f}^\bullet$. This concludes the commutativity since
	\[\begin{tikzcd}
		Q \arrow[r] & \overline{Q} \\
		P \arrow[r] \arrow[u, "f"] & \overline{P} \arrow[u, "{\overline{f}}"']
	\end{tikzcd} \quad \text{commutes.} \]
	
	When $(Q, q)$ varies, the resulting $(P, p)$ determine objects of $\cate{D}^{\mathrm{b}}_K(X)$ by \cite[Lemma 6.4.8]{Ac21}. The assignment
	\[ \left( M_Q, \ldots \right)_{(Q, q)} \mapsto \left( \overline{f}^\bullet M_Q, \ldots\right)_{(P, p)}, \]
	where $M_Q$ are objects of $\cate{D}^{\mathrm{b}}(\overline{Q})$, is exactly the recipe of Bernstein--Lunts \cite{BL94} for defining $f^{\bullet}$ or its shift $f^!$ on equivariant derived categories; see \cite[\S 6.5]{Ac21}. This completes the proof.
\end{proof}

Finally, suppose that $N \lhd K$ and $X$ is a smooth affine $K/N$-variety. There is an inflation functor
\[ \mathrm{Infl}^K_{K/N}: \cate{D}^{\mathrm{b}}_{K/N}(X) \to \cate{D}^{\mathrm{b}}_K(X). \]
The construction is based on the fact that if $p: P \to X$ is a resolution as $K$-varieties, then so is its quotient $p': P' := P/N \to X$ as $K/N$-varieties, and $\overline{P} \rightiso \overline{P'}$. It is $t$-exact.

\begin{proposition}\label{prop:Infl-epsilon}
	The diagram
	\[\begin{tikzcd}
		{}^{\mathrm{h}}\cate{D}^{\mathrm{b}}(D_X, K/N) \arrow[d, "{\mathrm{Infl}^K_{K/N}}"'] \arrow[r, "\varepsilon"] & \cate{D}^{\mathrm{b}}_{K/N}(X) \arrow[d, "{\mathrm{Infl}^K_{K/N}}"] \\
		{}^{\mathrm{h}}\cate{D}^{\mathrm{b}}(D_X, K) \arrow[r, "\varepsilon"'] & \cate{D}^{\mathrm{b}}_K(X)
	\end{tikzcd}\]
	commutes up to isomorphism, where the $\mathrm{Infl}^K_{K/N} := \mathrm{Infl}^{D_X, K}_{D_X, K/N}$ on the left is defined in \S\ref{sec:inv-coinv}.
\end{proposition}
\begin{proof}
	Given a resolution $(P, p)$ as $K$-varieties, define $(P', p')$ as above. By the construction \eqref{eqn:epsilon-construction} of $\varepsilon$, it suffices to check the commutativity of
	\begin{equation*}
		\begin{tikzcd}
			{}^{\mathrm{h}} \cate{D}^{\mathrm{b}}(D_X, K/N) \arrow[r, "{(p')^\bullet}"] \arrow[d, "{\mathrm{Infl}^K_{K/N}}"'] & {}^{\mathrm{h}} \cate{D}^{\mathrm{b}}(D_{P'}, K/N) \arrow[d] & \cate{D}^{\mathrm{b}}((D_{P'}, K/N)\dcate{Mod}) \arrow[l, "\sim"'] \arrow[d] & \cate{D}^{\mathrm{b}}(\overline{P'}) \arrow[l, "\sim"'] \arrow[d] \\
			{}^{\mathrm{h}} \cate{D}^{\mathrm{b}}(D_X, K) \arrow[r, "{p^\bullet}"] & {}^{\mathrm{h}} \cate{D}^{\mathrm{b}}(D_P, K) & \cate{D}^{\mathrm{b}}((D_P, K)\dcate{Mod}) \arrow[l, "\sim"'] & \cate{D}^{\mathrm{b}}(\overline{P}) \arrow[l, "\sim"']
		\end{tikzcd}
	\end{equation*}
	up to isomorphism, where the vertical arrows except the first one are inverse images via $P \to P'$ and $\overline{P} \rightiso \overline{P'}$. This is clear.
\end{proof}


\section{Localization and higher regularity}\label{sec:Loc}
\subsection{Derived localization functor}\label{sec:Loc-functor}
Let $G$ be a connected reductive group, and let $K \subset G$ be a reductive subgroup. Let $X$ be an affine smooth $G$-variety.

From this we obtain a pair $(\mathfrak{g}, K)$ as in \S\ref{sec:gK-basic}: the homomorphism $\Ad: K \to \Aut(\mathfrak{g})$ is just the adjoint action, and $\iota: \mathfrak{k} \to \mathfrak{g}$ is the inclusion. Recall that we write $(\mathfrak{g}, K)$ instead of $(U(\mathfrak{g}), K)$ for the pairs.

On the other hand, we also obtain $(D_X, G)$ and its subpair $(D_X, K)$. The constructions in \S\ref{sec:D-basic} are applicable.

\begin{proposition}\label{prop:j-pair}
	The map $j: U(\mathfrak{g}) \to D_X$ induced from the $G$-action on $X$ induces a morphism $(\mathfrak{g}, K) \to (D_X, K)$ of pairs in the sense of \S\ref{sec:HC-dga}.
\end{proposition}
\begin{proof}
	The map $j$ is $K$-equivariant since $K \subset G$. The requirement that $\mathfrak{k} \to U(\mathfrak{g}) \xrightarrow{j} D_X$ composes to $j|_{\mathfrak{k}}: \mathfrak{k} \to D_X$ is trivial.
\end{proof}

Therefore we have the change-of-algebra dg-functor in Proposition \ref{prop:adjoint-h-oblv} (i):
\[ D_X \dotimes{U(\mathfrak{g})} (\cdot): {}^{\mathrm{h}} \cate{C}(\mathfrak{g}, K) \to {}^{\mathrm{h}} \cate{C}(D_X, K). \]
Being left adjoint to oblivion, it preserves K-projectives.

\begin{definition}
	Denote by $\mathbf{Loc}_X = \mathbf{Loc}_{X, K}: {}^{\mathrm{h}} \cate{D}(\mathfrak{g}, K) \to {}^{\mathrm{h}} \cate{D}(D_X, K)$ the h-derived functor of $D_X \dotimes{U(\mathfrak{g})} (\cdot)$, called the \emph{h-derived localization functor}.
\end{definition}

Since $K$ is reductive, the definition makes sense: ${}^{\mathrm{h}} \cate{K}(\mathfrak{g}, K)$ has enough K-projectives, namely the ones from standard resolutions. The functor $\mathbf{Loc}_X$ is right $t$-exact.

\begin{remark}
	For general smooth $G$-variety, one can also define $\mathbf{Loc}_X$ with the sheafified pair $(\mathscr{D}_X, K)$. The Beilinson--Bernstein localization corresponds to the case when $X = \mathcal{B}$ is the flag variety, allowing twisted differential operators, and considering only $\Hm^0 \mathbf{Loc}_{\mathcal{B}}$.
\end{remark}

\begin{proposition}
	The triangulated functor $\mathbf{Loc}_X$ has amplitude in $[-\dim G, 0]$. In particular, it restricts to
	\[ \mathbf{Loc}_X: {}^{\mathrm{h}} \cate{D}^{\mathrm{b}}(\mathfrak{g}, K) \to {}^{\mathrm{h}} \cate{D}^{\mathrm{b}}(D_X, K). \]
\end{proposition}
\begin{proof}
	Use standard resolutions (Theorem \ref{prop:N-resolution-gen}).
\end{proof}

\begin{proposition}\label{prop:Loc-oblv}
	Let $T$ be a reductive subgroup of $K$. The diagram below commutes up to canonical isomorphism:
	\[\begin{tikzcd}[column sep=large]
		{}^{\mathrm{h}} \cate{D}(\mathfrak{g}, K) \arrow[r, "{\mathbf{Loc}_{X, K}}"] \arrow[d] & {}^{\mathrm{h}} \cate{D}(D_X, K) \arrow[d] \\
		{}^{\mathrm{h}} \cate{D}(\mathfrak{g}, T) \arrow[r, "{\mathbf{Loc}_{X, T}}"'] & {}^{\mathrm{h}} \cate{D}(D_X, T).
	\end{tikzcd}\]
	In particular, the same holds for
	\[\begin{tikzcd}[column sep=large]
		{}^{\mathrm{h}} \cate{D}(\mathfrak{g}, K) \arrow[r, "{\mathbf{Loc}_{X, K}}"] \arrow[d] & {}^{\mathrm{h}} \cate{D}(D_X, K) \arrow[d] \\
		\cate{D}(\mathfrak{g}\dcate{Mod}) \arrow[r, "{D_X \otimesL[U(\mathfrak{g})] (\cdot)}"'] & \cate{D}(X) .
	\end{tikzcd}\]
\end{proposition}
\begin{proof}
	The starting point is the commutativity of
	\[\begin{tikzcd}
		{}^{\mathrm{h}} \cate{C}(\mathfrak{g}, K) \arrow[r] \arrow[d] & {}^{\mathrm{h}} \cate{C}(D_X, K) \arrow[d] \\
		{}^{\mathrm{h}} \cate{C}(\mathfrak{g}, T) \arrow[r] & {}^{\mathrm{h}} \cate{C}(D_X, T).
	\end{tikzcd}\]
	The vertical arrows are exact and preserve K-projectives by Corollary \ref{prop:change-group-K-proj}. Taking left h-derived functors via Corollary \ref{prop:derived-functor-composite} gives the desired result.
\end{proof}

\begin{proposition}
	Let $V$ be a $(\mathfrak{g}, K)$-module. There is a canonical isomorphism of $K$-equivariant $D_X$-modules
	\[ \Hm^0 \mathbf{Loc}_X(V) \simeq D_X \dotimes{U(\mathfrak{g})} V; \]
	on the right-hand side, $D_X$ acts by left multiplication and $K$ acts diagonally.
\end{proposition}
\begin{proof}
	As $K$ is reductive, one can take the standard resolution $P := V \otimes N\mathfrak{g} \to V$. Now take the h-complex $D_X \dotimes{U(\mathfrak{g})} P$. Its $\Hm^0$ yields $D_X \dotimes{U(\mathfrak{g})} V$.
\end{proof}

In view of this result, $\mathbf{Loc}_X$ or its cohomologies can be seen as ``higher localization''.

The localization functor has the following symmetries. Let $D_X^G$ be the algebra of $G$-invariant differential operators on $X$. For every $z \in D_X^G$ and every h-complex $P$ over $(D_X, K)$, we obtain an endomorphism $\mathcal{R}_z$ of the h-complex $D_X \dotimes{U(\mathfrak{g})} P$ given by
\begin{align*}
	\mathcal{R}_z^n: D_X \dotimes{U(\mathfrak{g})} P^n & \to D_X \dotimes{U(\mathfrak{g})} P^n \\
	D \otimes p & \mapsto Dz \otimes p.
\end{align*}
This is functorial in $P$ and passes to the h-derived category, giving rise to a homomorphism of algebras
\begin{equation}\label{eqn:Z-action}
	\mathcal{R}: (D_X^G)^{\mathrm{op}} \to \End_{\mathrm{functors}}(\mathbf{Loc}_X).
\end{equation}
In fact, the construction works over any $G$-variety $X$ in the sheafified context. It is also compatible with oblivion, relative to the diagram in Proposition \ref{prop:Loc-oblv}.

\subsection{Localization and co-invariants}\label{sec:Loc-coinv}
Consider a homogeneous $G$-space $G = H \backslash G$, and let $x$ be the point $H \cdot 1$ of $X$.

The arguments in \S\ref{sec:Ext-application} will rely crucially on the following fact relating localizations and co-invariants. We do not assume $G$ is reductive or $X$ is affine here.

\begin{proposition}\label{prop:Loc-coinv}
	For every $\mathfrak{g}$-module $V$, we have the canonical isomorphism
	\begin{align*}
		V/\mathfrak{h}V & \rightiso \CC \dotimes{\mathscr{O}_x} \left( \mathscr{D}_X \dotimes{U(\mathfrak{g})} V \right)_x \\
		v + \mathfrak{h}V & \mapsto 1 \otimes (1 \otimes v),
	\end{align*}
	where $\mathscr{O}_x$ is the local ring at $x$ and the $\otimes$ is relative to the evaluation map at $x$.
\end{proposition}
\begin{proof}
	See \cite[Lemma 2.2]{BZG19}.
\end{proof}

Let $K^H$ be a subgroup of $H$, acting trivially on $\CC$. Several observations are in order.
\begin{itemize}
	\item If $V$ is a weak $(\mathfrak{g}, K^H)$-module, the isomorphism in Proposition \ref{prop:Loc-coinv} is $K^H$-equivariant. Thus we obtain an isomorphism between two dg-functors from ${}^{\mathrm{w}} \cate{C}(\mathfrak{g}, K^H)$ to ${}^{\mathrm{w}} \cate{C}(\CC, K^H)$.
	\item Furthermore, the functors $V \mapsto V/\mathfrak{h}V$ and $V \mapsto \CC \dotimes{\mathscr{O}_x} \left( \mathscr{D}_X \dotimes{U(\mathfrak{g})} V \right)_x$ both lift to the level of h-complexes
	\[ {}^{\mathrm{h}} \cate{C}(\mathfrak{g}, K^H) \to {}^{\mathrm{h}} \cate{C}(\CC, K^H). \]
	In fact, the dg-functor induced by $V \mapsto V/\mathfrak{h}V$ equals the $\mathrm{coInv}^{U(\mathfrak{h}), K^H}_{\CC, K^H}$ from Proposition \ref{prop:Inv-coInv} (in the h-setting), where $\CC$ is viewed as the quotient of $U(\mathfrak{h})$ by augmentation ideal.
	
	\item Since the data $i_\xi$ in an h-complex over $(\mathfrak{g}, K^H)$ are $\mathfrak{g}$-linear, and the isomorphism in Proposition \ref{prop:Loc-coinv} is natural in the $\mathfrak{g}$-module $V$, it also yields an isomorphism between these functors on the level of h-complexes.
\end{itemize}

Observe that $\CC = D_{\mathrm{pt}}$, and $\CC \dotimes{\mathscr{O}_x} (\cdot)_x$ equals the non-derived inverse image of $\mathscr{D}_X$-modules via the inclusion $i_x: \mathrm{pt} \to X$ of $X$. The next result is immediate.

\begin{lemma}\label{prop:Loc-coInv-prep}
	Assume that $X$ is affine. Let $K^H$ be a subgroup of $H$. The following diagram commutes up to a canonical isomorphism.
	\[\begin{tikzcd}[column sep=large]
		{}^{\mathrm{h}} \cate{C}(\mathfrak{g}, K^H) \arrow[d] \arrow[r, "{D_X \dotimes{U(\mathfrak{g})} (\cdot)}"] & {}^{\mathrm{h}} \cate{C}(D_X, K^H) \arrow[d, "{\CC \dotimes{\mathscr{O}_x} (\cdot)_x}"] \\
		{}^{\mathrm{h}} \cate{C}(\mathfrak{h}, K^H) \arrow[r, "{\mathrm{coInv}^{U(\mathfrak{h}), K^H}_{\CC, K^H}}"' inner sep=0.6em] & {}^{\mathrm{h}} \cate{C}(D_{\mathrm{pt}}, K^H)
	\end{tikzcd}\]
\end{lemma}

When $K^H$ is reductive, one can take the left h-derived functor $\Lder \left(\mathrm{coInv}^{U(\mathfrak{h}), K^H}_{\CC, K^H}\right)$ on the bounded-above h-derived categories.

\begin{proposition}\label{prop:Loc-coInv}
	Assume that $X$ is affine. Let $K$ (resp.\ $K^H$) be a reductive subgroup of $G$ (resp.\ $H$), and denote $\mathrm{oblv}^K_{K^H}$ the oblivion ${}^{\mathrm{h}} \cate{D}(D_X, K) \to {}^{\mathrm{h}} \cate{D}(D_X, K^H)$. There are canonical isomorphisms
	\begin{align*}
		i_x^\bullet \mathbf{Loc}_{X, K^H} & \simeq \Lder \left(\mathrm{coInv}^{U(\mathfrak{h}), K^H}_{\CC, K^H}\right), \\
		i_x^\bullet \left( \mathrm{oblv}^K_{K^H} \mathbf{Loc}_{X, K}(M)\right) & \simeq \Lder \left(\mathrm{coInv}^{U(\mathfrak{h}), K^H}_{\CC, K^H}\right)(M|_H),
	\end{align*}
	where $M$ stands for an object of ${}^{\mathrm{h}} \cate{D}^-(\mathfrak{g}, K)$ and $M|_H$ denotes its image in ${}^{\mathrm{h}} \cate{D}^-(\mathfrak{h}, K^H)$.
\end{proposition}
\begin{proof}
	It suffices to prove the first isomorphism, since the second follows by commuting $\mathbf{Loc}_X$ and oblivion using Proposition \ref{prop:Loc-oblv}.
	
	Note that the functors emitting from ${}^{\mathrm{h}} \cate{C}(\mathfrak{g}, K^H)$ in the diagram of Lemma \ref{prop:Loc-coInv-prep} preserve K-projectives, by Lemma \ref{prop:restriction-lemma}. It remains to take left h-derived functors in two ways, which give isomorphic results.
\end{proof}

Finally, $D_X^G$ acts on the right of $i_x^\bullet \mathbf{Loc}_{X, K^H}$ through its action \eqref{eqn:Z-action} on $\mathbf{Loc}_{X, K^H}$. Let us relate it to co-invariants.

\begin{proposition}
	Let $z \in D_X^G$. Via Proposition \ref{prop:Loc-coInv}, the $z$-action on $i_x^\bullet \mathbf{Loc}_{X, K^H}$ arises from the following endomorphism of the functor $\mathrm{coInv}^{U(\mathfrak{k}), K^H}_{\CC, K^H}$. Near $x$, we can express $z$ as the image of some $f \otimes u$ under the homomorphism
	\[ \mathscr{O}_X \otimes U(\mathfrak{g}) \to \mathscr{D}_X, \quad f \otimes u \mapsto f j(u) \]
	where $j: U(\mathfrak{g}) \to D_X$ is the natural map. For every h-complex $P$ over $(\mathfrak{h}, K^H)$, we let $z$ act in each degree $n$ by
	\begin{align*}
		P^n / \mathfrak{h} P^n & \to P^n / \mathfrak{h}P^n \\
		p + \mathfrak{h}P^n & \mapsto f(x) \cdot up + \mathfrak{h} P^n.
	\end{align*}
\end{proposition}
\begin{proof}
	The action $\identity \otimes \mathcal{R}_z^n$ on $\CC \dotimes{\mathscr{O}_x} D_X \dotimes{U(\mathfrak{g})} P^n$ depends only on the behavior of $z$ near $x$. Therefore it maps
	\[ 1 \otimes 1 \otimes p \mapsto 1 \otimes f j(u) \otimes p = f(x) \otimes 1 \otimes up. \]
	It remains to compare this with Proposition \ref{prop:Loc-coinv}.
\end{proof}

\subsection{Statement of higher regularity}
The assumptions on $(\mathfrak{g}, K)$ and $X$ from \S\ref{sec:Loc-functor} remain in force.

\begin{definition}
	An object of ${}^{\mathrm{h}} \cate{D}(D_X, K)$ is said to be \emph{holonomic} (resp.\ \emph{regular holonomic}) if all its cohomologies are holonomic (resp.\ regular holonomic) as $D_X$-modules. This gives rise to a full triangulated subcategory ${}^{\mathrm{h}} \cate{D}_{\mathrm{h}}(D_X, K)$ (resp.\ ${}^{\mathrm{h}} \cate{D}_{\mathrm{rh}}(D_X, K)$) of ${}^{\mathrm{h}} \cate{D}(D_X, K)$, and one may impose boundedness conditions as well.
\end{definition}

Inside the equivariant derived category $\cate{D}^{\mathrm{b}}_K(X)$, we also have $\cate{D}^{\mathrm{b}}_{K, \mathrm{h}}(X)$ (resp.\ $\cate{D}^{\mathrm{b}}_{K, \mathrm{rh}}(X)$), the full triangulated subcategory of objects with holonomic (resp.\ regular holonomic) cohomologies. It matches ${}^{\mathrm{h}} \cate{D}^{\mathrm{b}}_{\mathrm{h}}(D_X, K)$ (resp.\ ${}^{\mathrm{h}} \cate{D}^{\mathrm{b}}_{\mathrm{h}}(D_X, K)$) under Beilinson's equivalence (Theorem \ref{prop:Beilinson-equiv}).

For every $D_X$-module $L$, denote by $\mathrm{Ch}(L) \subset T^* X$ its characteristic variety. We have the \emph{moment map}
\[ \bm{\mu}: T^* X \to \mathfrak{g}^*. \]

By choosing a base point, every homogeneous $G$-spaces take the form $H \backslash G$, and $T^*(H \backslash G) \simeq \mathfrak{h}^\perp \utimes{H} G$. Then $\bm{\mu}$ maps $[\lambda, g]$ to $\Ad^*(g^{-1}) \lambda$ (co-adjoint action) for all $\lambda \in \mathfrak{h}^\perp$ and $g \in G$.

Let $\mathcal{N} \subset \mathfrak{g}^*$ denote the nilpotent cone.

\begin{definition}\label{def:rh-plus}
	Let ${}^{\mathrm{h}} \cate{D}^{\mathrm{b}}_{\mathrm{rh}+}(D_X, K)$ be the full subcategory of ${}^{\mathrm{h}} \cate{D}^{\mathrm{b}}_{\mathrm{rh}}(D_X, K)$ consisting of objects $L$ such that
	\[ \mathrm{Ch}\left( \Hm^n(L) \right) \subset \bm{\mu}^{-1}\left( \mathcal{N} \cap \mathfrak{k}^\perp \right) \]
	for all $n \in \Z$, where $\Hm^n(L)$ is viewed merely as a $D_X$-module.
\end{definition}

This is a full triangulated subcategory by standard properties of characteristic varieties \cite[\S 2.2]{HTT08}. Again, it has the counterpart $\cate{D}^{\mathrm{b}}_{K, \mathrm{rh}+}(X)$ inside $\cate{D}^{\mathrm{b}}_{K, \mathrm{rh}}(X)$.

\begin{definition}
	A normal $G$-variety is said to be \emph{spherical} if there is an open $B$-orbit in $X$ for some (equivalently, any) Borel subgroup $B \subset G$. A subgroup $H \subset G$ is said to be \emph{spherical} if the homogeneous $G$-space $H \backslash G$ is spherical.
\end{definition}

We are ready to state the higher regularity of localizations.

\begin{theorem}\label{prop:regularity}
	Let $G$ be a connected reductive group. Suppose that
	\begin{itemize}
		\item $X$ is an affine spherical homogeneous $G$-space,
		\item $K$ is a reductive spherical subgroup of $G$.
	\end{itemize}

	Let $V$ be a Harish-Chandra module over $(\mathfrak{g}, K)$ (Definition \ref{def:HC-module}), then $\mathbf{Loc}_X(V)$ lies in ${}^{\mathrm{h}} \cate{D}^{\mathrm{b}}_{\mathrm{rh}+}(D_X, K)$.
\end{theorem}

The proof will be split into three chunks in \S\ref{sec:regularity-criterion}, the hardest one among which will be settled in \S\ref{sec:end-of-regularity}.

\begin{corollary}\label{prop:regularity-gen}
	Under the assumptions of Theorem \ref{prop:regularity}, let ${}^{\mathrm{h}}\cate{D}^{\mathrm{b}}_{\mathrm{HC}}(\mathfrak{g}, K)$ be the full triangulated subcategory of ${}^{\mathrm{h}}\cate{D}^{\mathrm{b}}(\mathfrak{g}, K)$ consisting of objects whose cohomologies are Harish-Chandra $(\mathfrak{g}, K)$-modules. Then $\mathbf{Loc}_X$ restricts to
	\[ {}^{\mathrm{h}}\cate{D}^{\mathrm{b}}_{\mathrm{HC}}(\mathfrak{g}, K) \to {}^{\mathrm{h}}\cate{D}^{\mathrm{b}}_{\mathrm{rh}+}(D_X, K). \]
\end{corollary}
\begin{proof}
	Since $\mathbf{Loc}_X$ is triangulated, one can truncate and shift to reduce to the case of Harish-Chandra $(\mathfrak{g}, K)$-modules.
\end{proof}

We remark that in view of Theorems \ref{prop:BL-equiv} and \ref{prop:Beilinson-equiv}, the result above can be rephrased in terms of ``usual'' derived categories as
\[ \mathbf{Loc}_X: \cate{D}^{\mathrm{b}}_{\mathrm{HC}}(\mathfrak{g}, K) \to \cate{D}^{\mathrm{b}}_{K, \mathrm{rh}+}(X). \]
However, the construction of $\mathbf{Loc}_X$ passes through h-derived categories.

\begin{remark}
	Take $X = H \backslash G$ for some subgroup $H \subset G$. By Matsushima's criterion \cite[Theorem 3.8]{Ti11}, $H \backslash G$ is affine if and only if $H$ is reductive. The conditions in Theorem \ref{prop:regularity} are thus symmetric in $H$ and $K$: both are required to be spherical and reductive. One can also view $\mathbf{Loc}_{H \backslash G}(V)$ as objects of the derived category of the stack $H \backslash G / K$.
\end{remark}

\begin{example}
	In order to appreciate the property of characteristic varieties in Definition \ref{def:rh-plus}, let us take
	\begin{itemize}
		\item $H$: a connected reductive group, embedded diagonally in $G := H \times H$,
		\item $K := H$ as a subgroup of $G$,
		\item $X := H$ with $G$ acting by $x(h_1, h_2) = h_2^{-1} x h_1$.
	\end{itemize}

	Note that $X \simeq H \backslash G$ by choosing $1$ as the base point, and $H \backslash G / K$ becomes the adjoint quotient stack $\frac{H}{H}$.
	
	The $(\mathfrak{g}, K)$-modules in this case give rise to \textit{Harish-Chandra bimodules}. They are closely related to harmonic analysis on $H(\CC)$.
	
	All conditions in Theorem \ref{prop:regularity} are met in this case. Given a Harish-Chandra $(\mathfrak{g}, K)$-module $V$ and $n \in \Z$, the irreducible constituents of the $\Ad$-equivariant $D_H$-module $\Hm^n \mathbf{Loc}_H(V)$ are actually \emph{character $D_H$-modules}, i.e.\ regular holonomic $D_H$-modules corresponding to \emph{character sheaves} via Riemann--Hilbert (see \cite[\S 2.1]{MV88}, and also \cite{Gin89}). Indeed, this follows from the condition on $\mathrm{Ch}\left(\Hm^n \mathbf{Loc}_H(V)\right)$ and \cite[Theorem 4.4]{MV88}.
\end{example}

\subsection{Criterion of regularity}\label{sec:regularity-criterion}
Let $K \subset G$ and $X$ be as in Theorem \ref{prop:regularity}. Our proof of Theorem \ref{prop:regularity} is based on the result below from \cite{Li22}, which is a variant of Ginzburg's \cite[Corollary 8.9.1]{Gin89}.

\begin{theorem}\label{prop:reg-criterion}
	Suppose that $M$ is a $D_X$-module with the following properties:
	\begin{enumerate}[(R1)]
		\item $M$ is finitely generated over $D_X$;
		\item $M$ carries a structure of $K$-equivariant $D_X$-module;
		\item $M$ is locally $\mathcal{Z}(\mathfrak{g})$-finite, where $\mathcal{Z}(\mathfrak{g})$ acts through the homomorphism $j: U(\mathfrak{g}) \to D_X$ of algebras.
	\end{enumerate}
	Then $M$ is regular holonomic, and $\mathrm{Ch}(M) \subset \bm{\mu}^{-1}(\mathcal{N} \cap \mathfrak{k}^\perp)$.
\end{theorem}
\begin{proof}
	This is an instance of \cite[Proposition 3.4, Corollary 5.7]{Li22}.
\end{proof}

We can now depict the strategy for proving Theorem \ref{prop:regularity}.

\begin{proof}[Proof of Theorem \ref{prop:regularity}]
	Let $V$ be a Harish-Chandra $(\mathfrak{g}, K)$-module. Set $\mathcal{L} := \mathbf{Loc}_{X, \{1\}}(V)$, which is also the image of $\mathbf{Loc}_{X, K}(V)$ under ${}^{\mathrm{h}} \cate{D}(D_X, K) \to \cate{D}(X)$ by Proposition \ref{prop:Loc-oblv}. We shall verify (R1) --- (R3) for the cohomologies of $\mathcal{L}$.
	
	First, we identify $\mathcal{L}$ with $D_X \otimesL[U(\mathfrak{g})] V$. Since $V$ is finitely generated over $\mathfrak{g}$ and $U(\mathfrak{g})$ is left (and right) Noetherian, there is a free resolution over $\mathfrak{g}$:
	\[ \cdots \to U(\mathfrak{g})^{\oplus n_1} \to U(\mathfrak{g})^{\oplus n_0} \to V \to 0 , \quad n_i \in \Z_{\geq 0}. \]
	Hence $D_X \otimesL[U(\mathfrak{g})] V$ is represented by the complex
	\[ \cdots \to D_X^{\oplus n_{i+1}} \to D_X^{\oplus n_i} \to \cdots . \]
	
	Since $D_X$ is left (and right) Noetherian by \cite[Proposition 1.4.6]{HTT08}, the cohomologies of the complex above are finitely generated $D_X$-modules. This verifies (R1).
	
	As for (R2), from \eqref{eqn:oblv-comm} we see that $\Hm^n(\mathcal{L})$ is the image of $\Hm^n(\mathbf{Loc}_{X, K}(V))$, for all $n$.
	
	What remains is (R3); this will be settled by Proposition \ref{prop:local-Zg-finiteness}.
\end{proof}

Note that the case of zeroth cohomology, i.e.\ the non-derived $D_X \dotimes{U(\mathfrak{g})} V$, has been addressed in \cite[Example 5.5 (iii)]{Li22}.

\subsection{End of the proof}\label{sec:end-of-regularity}
Hereafter, $G$ is a connected reductive group, and $X$ is a smooth affine $G$-variety.

We say a left module $M$ over a commutative algebra $A$ is \emph{locally $A$-finite} if it it is the union of finite-dimensional $A$-submodules.

\begin{lemma}\label{prop:Zg-finite-ses}
	Suppose $A$ is a finitely generated commutative algebra. For every short exact sequence $0 \to N' \to N \to N'' \to 0$ of $A$-modules, we have: $N$ is locally finite if and only if $N'$ and $N''$ are both locally finite.
\end{lemma}
\begin{proof}
	It suffices to explain the ``if'' part. Let $x \in N$. Its image in $N''$ is annihilated by an ideal $I$ of finite codimension (as vector subspace), hence $Ix \subset N'$. Since $A$ is Noetherian, $I$ is finitely generated, hence there exist ideals $I_1, \ldots, I_k$ of finite codimension such that $I_1 \cdots I_k I x = 0$. However $I_1 \cdots I_k I$ is also of finite codimension, by the structure of $A$.
\end{proof}

\begin{lemma}\label{prop:alpha-surj}
	Let $M$ be a $D_X$-module, thus also a $U(\mathfrak{g})$-submodule via the homomorphism $j: U(\mathfrak{g}) \to D_X$. Let $M^\natural \subset M$ be a $\mathfrak{g}$-submodule. Then the map
	\[\begin{tikzcd}[row sep=tiny]
		\alpha: D_X \otimes M^\natural \arrow[r] & M \\
		P \otimes m \arrow[r, mapsto] & Pm
	\end{tikzcd}\]
	is a homomorphism of $\mathfrak{g}$-modules, if we let $\theta \in \mathfrak{g}$ act on $P \in D_X$ by $\theta \odot P := [j(\theta), P]$.
\end{lemma}
\begin{proof}
	Omit $j$ to simplify notation. Given $\theta \in \mathfrak{g}$, we have
	\begin{align*}
		\theta \left(\alpha(P \otimes m)\right) & = \theta(Pm) = (\theta P) m - (P\theta) m + P(\theta m) \\
		& = [\theta, P] m + P(\theta m) \\
		& = \alpha\left( (\theta \odot P) \otimes m \right) + \alpha(P \otimes \theta m) \\
		& = \alpha\left( \theta (P \otimes m) \right)   
	\end{align*}
	by the standard definition of the $\mathfrak{g}$-module structure on $D_X \otimes M$.
\end{proof}

\begin{lemma}\label{prop:tensor-Kostant}
	Let $M$ be a $D_X$-module as before, and let $M^\flat \subset M$ be a $\mathcal{Z}(\mathfrak{g})$-submodule that is locally finite, and generates $M$ over $D_X$. Then $M$ is locally $\mathcal{Z}(\mathfrak{g})$-finite.
\end{lemma}
\begin{proof}
	Take the $\mathfrak{g}$-submodule $M^\natural := U(\mathfrak{g}) M^\flat$ of $M$. Note that $M^\natural$ is locally $\mathcal{Z}(\mathfrak{g})$-finite, and generates $M$ over $D_X$. Note that $G$ acts algebraically on $D_X$ by $P \xmapsto{g} gPg^{-1}$; this is indeed clear, and is included in Proposition \ref{prop:j-pair}. Its derivative is the action $\odot$ in Lemma \ref{prop:alpha-surj}, hence $D_X$ is a union of finite-dimensional $\mathfrak{g}$-submodules.
	
	Consider now the $\mathfrak{g}$-linear surjection $\alpha: D_X \otimes M^\natural \twoheadrightarrow M$ from Lemma \ref{prop:alpha-surj}. By a theorem of Kostant \cite[Theorem 7.133]{KV95} (and taking $\varinjlim$), the $\mathfrak{g}$-module $D_X \otimes M^\natural$ is seen to be locally $\mathcal{Z}(\mathfrak{g})$-finite. Hence so is $M$.
\end{proof}

Note that in Lemma \ref{prop:tensor-Kostant}, if $M^\flat$ has infinitesimal character $\chi$, then $M$ as a $\mathcal{Z}(\mathfrak{g})$-module is supported on $\chi + \{\text{weights}\}$. Here we describe $\Spec\mathcal{Z}(\mathfrak{g})$ via Harish-Chandra's isomorphism.

For every homomorphism of algebras $\chi: \mathcal{Z}(\mathfrak{g}) \to \CC$, we have the ideal
\[ \mathfrak{m}_\chi := \Ker(\chi) \subset \mathcal{Z}(\mathfrak{g}), \]
and we say a $\mathfrak{g}$-module has \emph{infinitesimal character} $\chi$ if $\mathcal{Z}(\mathfrak{g})$ acts through $\chi$. For example, the $\mathfrak{g}$-module below has infinitesimal character $\chi$:
\begin{equation}\label{eqn:Mchi}
	M_\chi := U(\mathfrak{g}) / \mathfrak{m}_\chi U(\mathfrak{g}).
\end{equation}

For all $\mathfrak{g}$-module $V$, define the $D_X$-modules
\[ \Tor^{U(\mathfrak{g})}_n\left( D_X, V \right) := \Hm^{-n}\left( D_X \otimesL[U(\mathfrak{g})] V \right), \quad n \in \Z. \]

\begin{definition}\label{def:ZX}
	Following \cite[pp.254--255]{Kn94}, we define $\mathcal{Z}(X)$ to be the center of $D_X^G$. It is known to be the whole $D_X^G$ when $X$ is spherical, see \textit{loc.\ cit.}
\end{definition}

Note that $j: U(\mathfrak{g}) \to D_X$ restricts to $\mathcal{Z}(\mathfrak{g}) \to \mathcal{Z}(X)$.

\begin{lemma}\label{prop:Mchi-Zg-finiteness}
	Given $\chi$, the $D_X$-modules $\Tor^{U(\mathfrak{g})}_n\left( D_X, M_\chi \right)$ are locally $\mathcal{Z}(\mathfrak{g})$-finite for all $n \in \Z$, where $M_\chi$ is as in \eqref{eqn:Mchi}.
\end{lemma}
\begin{proof}
	This relies critically on some results of B.\ Kostant and F.\ Knop. We will perform change-of-rings through the commutative diagram
	\[\begin{tikzcd}[column sep=small, row sep=small]
		& D_X & \\
		\mathcal{Z}(X) \arrow[hookrightarrow, ru, "\text{free}"] & & U(\mathfrak{g}) \arrow[lu, "j"'] \\
		& \mathcal{Z}(\mathfrak{g}) \arrow[lu, "j"] \arrow[hookrightarrow, ru, "\text{free}"'] &
	\end{tikzcd}\]
	of algebras; here ``free'' means free as left and right modules. Indeed,
	\begin{itemize}
		\item the freeness of $U(\mathfrak{g})$ over $\mathcal{Z}(\mathfrak{g})$ is due to Kostant --- see \cite[Theorem 7.114]{KV95},
		\item the freeness of $D_X$ over $\mathcal{Z}(X)$ is due to Knop \cite[Theorem 9.5 (c)]{Kn94}.
	\end{itemize}
	
	Using the $\mathcal{Z}(\mathfrak{g})$-flatness of $U(\mathfrak{g})$, inside $\cate{D}(\mathfrak{g}\dcate{Mod})$ we have
	\begin{equation*}
		M_\chi \simeq U(\mathfrak{g}) \dotimes{\mathcal{Z}(\mathfrak{g})} \frac{\mathcal{Z}(\mathfrak{g})}{\mathfrak{m}_\chi} \simeq U(\mathfrak{g}) \otimesL[\mathcal{Z}(\mathfrak{g})] \frac{\mathcal{Z}(\mathfrak{g})}{\mathfrak{m}_\chi}.
	\end{equation*}

	Set $N_\chi := \mathcal{Z}(\mathfrak{g}) / \mathfrak{m}_\chi$. Since change-of-ring preserves K-projectives, performing $\otimesL$ in stages leads to
	\begin{align*}
		D_X \otimesL[U(\mathfrak{g})] M_\chi & \simeq D_X \otimesL[U(\mathfrak{g})] \left( U(\mathfrak{g}) \otimesL[\mathcal{Z}(\mathfrak{g})] N_\chi \right) \\
		& \simeq D_X \otimesL[\mathcal{Z}(\mathfrak{g})] N_\chi \\
		& \simeq D_X \otimesL[\mathcal{Z}(X)] \left( \mathcal{Z}(X) \otimesL[\mathcal{Z}(\mathfrak{g})] N_\chi \right) \;\quad \text{in}\; \cate{D}(D_X\dcate{Mod}).
	\end{align*}
	Hence by the $\mathcal{Z}(X)$-flatness of $D_X$, for all $n \geq 0$ we have in $D_X\dcate{Mod}$
	\begin{equation}\label{eqn:Mchi-Zg-finiteness-aux0}
		\Tor^{U(\mathfrak{g})}_n\left( D_X, M_\chi \right) \simeq D_X \dotimes{\mathcal{Z}(X)} \Tor^{\mathcal{Z}(\mathfrak{g})}_n\left( \mathcal{Z}(X), N_\chi \right).
	\end{equation}

	We contend that for all $n \geq 0$,
	\begin{equation}\label{eqn:Mchi-Zg-finiteness-aux1}
		\mathcal{Z}(\mathfrak{g}) \;\text{acts on}\; \Tor^{\mathcal{Z}(\mathfrak{g})}_n \left( \mathcal{Z}(X), N_\chi\right) \;\text{through}\; \chi.
	\end{equation}
	Here $\mathcal{Z}(\mathfrak{g})$ acts through $j: \mathcal{Z}(\mathfrak{g}) \to \mathcal{Z}(X)$. The action of $z \in \mathcal{Z}(\mathfrak{g})$ is given as follows. Take a projective resolution $\cdots \to Q_1 \to Q_0 \to N_\chi \to 0$, then there exist $f_i \in \End_{\mathcal{Z}(\mathfrak{g})}(Q_i)$ making
	\[\begin{tikzcd}
		\cdots \arrow[r] & Q_1 \arrow[r] \arrow[d, "f_1"] & Q_0 \arrow[r] \arrow[d, "f_0"] & N_\chi \arrow[d, "z"] \arrow[r] & 0 \\
		\cdots \arrow[r] & Q_1 \arrow[r] & Q_0 \arrow[r] & N_\chi \arrow[r] & 0
	\end{tikzcd}\]
	commutative; different choices of $(f_i)_{i \geq 0}$ are related by $\mathcal{Z}(\mathfrak{g})$-linear homotopies. Take $\mathcal{Z}(X) \dotimes{\mathcal{Z}(\mathfrak{g})} (\cdot)$ on the whole diagram, then the induced action on $\Hm_n$ is the desired one; the commutativity of $\mathcal{Z}(X)$ is crucial here.
	\begin{itemize}
		\item Naturally, one can take $f_i = z$. 
		\item On the other hand, one can also take $f_i = \chi(z) \identity$. This proves \eqref{eqn:Mchi-Zg-finiteness-aux1}.
	\end{itemize}
	
	By \eqref{eqn:Mchi-Zg-finiteness-aux0}, $\Tor^{U(\mathfrak{g})}_n\left( D_X, M_\chi \right)$ is generated by $M^\flat := 1 \otimes \Tor^{\mathcal{Z}(\mathfrak{g})}_n\left( \mathcal{Z}(X), N_\chi \right)$ as a $D_X$-module. By combining \eqref{eqn:Mchi-Zg-finiteness-aux1} and Lemma \ref{prop:tensor-Kostant}, the local $\mathcal{Z}(\mathfrak{g})$-finiteness follows.
\end{proof}

\begin{proposition}\label{prop:local-Zg-finiteness}
	Let $V$ be a $\mathfrak{g}$-module that is finitely generated and locally $\mathcal{Z}(\mathfrak{g})$-finite. Then $\Tor^{U(\mathfrak{g})}_n( D_X, V)$ is locally $\mathcal{Z}(\mathfrak{g})$-finite for all $n \in \Z$.
\end{proposition}
\begin{proof}
	Using the assumptions on $V$, one produces a filtration $0 = V_0 \subset \cdots \subset V_n = V$ such that each subquotient has an infinitesimal character. By the long exact sequence (of $D_X$-modules) for $\Tor^{U(\mathfrak{g})}_n(D_X, \cdot)$ and Lemma \ref{prop:Zg-finite-ses}, we are thus reduced to the case that $V$ has infinitesimal character $\chi$.
	
	We argue by dimension shifting. The case $n < 0$ is trivial. By taking generators of $V$ and using \eqref{eqn:Mchi}, there is a short exact sequence of $\mathfrak{g}$-modules
	\[ 0 \to W \to M_\chi^{\oplus I} \to V \to 0 \]
	where $I$ is some (small) set; all modules have infinitesimal character $\chi$. In the long exact sequence we have the piece	
	\[ \Tor^{U(\mathfrak{g})}_n( D_X, M_\chi^{\oplus I}) \to \Tor^{U(\mathfrak{g})}_n( D_X, V) \to \Tor^{U(\mathfrak{g})}_{n-1}( D_X, W). \]
	The third term is locally $\mathcal{Z}(\mathfrak{g})$-finite by recursion; so is the first term by Lemma \ref{prop:Mchi-Zg-finiteness} and the fact that $\Tor^{U(\mathfrak{g})}_n(D_X, \cdot)$ commutes with direct sums. This completes the proof.
\end{proof}

\begin{remark}
	When $V$ has infinitesimal character $\chi$, the support of $\Tor^{U(\mathfrak{g})}_n(D_X, V)$ as a $\mathcal{Z}(\mathfrak{g})$-module can be loosely controlled as discussed after Lemma \ref{prop:tensor-Kostant}.
\end{remark}

We record a by-product of the proofs above, which implies that $D_X \dotimes{U(\mathfrak{g})} (\cdot)$ is usually non-exact.

\begin{proposition}\label{prop:non-exactness}
	Define $M_\chi$ by \eqref{eqn:Mchi} for all $\chi$. We have
	\[ \left[ \forall \chi, \; \Tor^{U(\mathfrak{g})}_1(D_X, M_\chi) = 0 \right] \iff \mathcal{Z}(X) \;\text{is flat over}\; \mathcal{Z}(\mathfrak{g}). \]
\end{proposition}
\begin{proof}
	Let $N_\chi := \mathcal{Z}(\mathfrak{g})/\mathfrak{m}_\chi$.
	By the local criterion of flatness, $\mathcal{Z}(X)$ is flat over $\mathcal{Z}(\mathfrak{g})$ if and only if $\Tor_1^{\mathcal{Z}(\mathfrak{g})}(\mathcal{Z}(X), N_\chi) = 0$ for all $\chi$. In turn, this is equivalent to $\Tor^{U(\mathfrak{g})}_1(D_X, M_\chi) = 0$ by \eqref{eqn:Mchi-Zg-finiteness-aux0} together with the freeness of $D_X$ over $\mathcal{Z}(X)$.
\end{proof}

\begin{remark}\label{rem:non-exactness}
	In view of the structure theorem in \cite{Kn94} and the notation therein, $\mathcal{Z}(X)$ is flat over $\mathcal{Z}(\mathfrak{g})$ if and only if the natural map between categorical quotients
	\[ (\rho + \mathfrak{a}_X^*) \sslash W_X \to \mathfrak{t}^* \sslash W \]
	is flat, which rarely holds; a necessary condition is that $\mathrm{rank}(X) = \mathrm{rank}(G)$. Therefore $D_X$ is rarely flat over $U(\mathfrak{g})$, by Proposition \ref{prop:non-exactness}.
	
	However, flatness of $\mathcal{Z}(X)$ does hold when $X$ is the affine closure of $U \backslash G$ where $U$ is a maximal unipotent subgroup: this is a combination of \cite[Lemma 6.4]{Kn94} and Pittie--Steinberg theorem. Note that $\mathcal{Z}(Y)$ is defined for all normal $G$-varieties $Y$ in \textit{loc.\ cit.}, and it is an equivariant-birational invariant of $Y$.
\end{remark}

The same arguments also lead to the following fact about the right $\mathcal{Z}(X)$-action on $\mathbf{Loc}_X$, which might be of independent interest.

\begin{proposition}\label{prop:ZX-locally-finite}
	Under the assumptions of Proposition \ref{prop:local-Zg-finiteness}, $\Tor_n^{U(\mathfrak{g})}(D_X, V)$ is locally $\mathcal{Z}(X)$-finite under the action \eqref{eqn:Z-action}, for all $n \in \Z$.
	
	Furthermore, if $V$ has infinitesimal character $\chi$, then the support of $\Tor_n^{U(\mathfrak{g})}(D_X, V)$ as a $\mathcal{Z}(X)$-module is a subset of the fiber of $\chi$ under $\Spec(\mathcal{Z}(X)) \to \Spec(\mathcal{Z}(\mathfrak{g}))$, which is finite.
\end{proposition}
\begin{proof}
	First, Lemma \ref{prop:Zg-finite-ses} applies to $\mathcal{Z}(X)$-modules, by describing the structure of $\mathcal{Z}(X)$ using Knop's theorem.
	
	As before, one reduces to the case when $V$ has infinitesimal character $\chi$, and then to the case $V = M_\chi$. Let $\mathcal{Z}(\mathfrak{g})$ act on $\Tor^{U(\mathfrak{g})}_n(D_X, M_\chi)$ via
	\[ \mathcal{Z}(\mathfrak{g}) \to \mathcal{Z}(X) \xrightarrow{\eqref{eqn:Z-action}} \End_{D_X}\left(\Tor^{U(\mathfrak{g})}_n(D_X, M_\chi)\right). \]
	
	By inspecting the proof of Lemma \ref{prop:Mchi-Zg-finiteness}, we see that the action above is simply $\chi$. We conclude by the fact that $\mathcal{Z}(X)$ is a finite $\mathcal{Z}(\mathfrak{g})$-module; see \cite[p.254]{Kn94}.
\end{proof}

\section{Application to \texorpdfstring{$\Ext$}{Ext}-branching}\label{sec:Ext-application}
\subsection{General setting of branching laws}\label{sec:general-branching}
Suppose we are given a pair $(\mathfrak{i}, K^I)$ (see \S\ref{sec:gK-basic}) and its subpair $(\mathfrak{h}, K^H)$, where $\mathfrak{h} \subset \mathfrak{i}$ and $K^H \subset K^I$; see \S\ref{sec:restriction-lemma} for the precise conditions on subpairs. This induces the diagonal embedding of pairs
\[ (\mathfrak{h}, K^H) \to (\mathfrak{i} \times \mathfrak{h}, K^I \times K^H). \]
On the first component it is $U(\mathfrak{h}) \to U(\mathfrak{i}) \otimes U(\mathfrak{h})$, given by $\theta \mapsto \theta \otimes 1 + 1 \otimes \theta$ for all $\theta \in \mathfrak{h}$. On the second component it is $k \mapsto (k, k)$ for all $k \in K^H$.

Consider h-complexes $M$ over $(\mathfrak{i}, K^I)$ and $N$ over $(\mathfrak{h}, K^H)$. Taking tensor product yields an h-complex over $(\mathfrak{i} \times \mathfrak{h}, K^I \times K^H)$, which we denote as $M \boxtimes N$.

\begin{definition}\label{def:M-H}
	Denote by $M|_H$ the restriction of $M$ to an h-complex over $(\mathfrak{h}, K^H)$. Define $(M \boxtimes N)|_H$ similarly via the diagonal embedding. Ditto for the functors induced on derived categories.
\end{definition}

By the general formalism of \S\ref{sec:derived-categories}, we can define $\RHom$ between h-complexes over these pairs; they will be denoted as $\RHom_{\mathfrak{h}, K^H}$ and so on.

\begin{proposition}
	Assume $K^H$ and $K^I$ are reductive. Let $M$ (resp.\ $N$) be an h-complex over $(\mathfrak{i}, K^I)$ (resp.\ $(\mathfrak{h}, K^H)$). Denote by $N^\vee$ the contragredient of $N$, as in Definition \ref{def:internal-Hom}. There is a canonical isomorphism in $\cate{D}(\CC)$
	\[ \RHom_{\mathfrak{h}, K^H}\left( M|_H, N^\vee \right) \simeq \RHom_{\mathfrak{h}, K^H}\left((M \boxtimes N)|_H, \CC \right). \]
	The same isomorphism also holds for complexes of $(\mathfrak{i}, K^I)$-modules and $(\mathfrak{h}, K^H)$-modules, with the corresponding $\RHom$.
\end{proposition}
\begin{proof}
	We will only sketch the case of h-complexes; the other case is similar and well-known.
	
	Take a K-projective resolution $P \to M$ in ${}^{\mathrm{h}} \cate{C}(\mathfrak{i}, K^I)$. Then so is $P|_H \to M|_H$  in ${}^{\mathrm{h}} \cate{C}(\mathfrak{h}, K^H)$ by Lemma \ref{prop:restriction-lemma}. We claim that $(P \boxtimes N)|_H \to (M \boxtimes N)|_H$ is a K-projective resolution as well.
	
	Since $((\cdot) \boxtimes N)|_H \simeq (\cdot)|_H \otimes N$ is exact, the above is indeed a quasi-isomorphism. By \cite[Lemma 3.2]{Pan05}, $(\cdot)|_H \otimes N$ has a right adjoint dg-functor given by internal $\Hom$ (see \S\ref{sec:gK-basic})
	\[ \Hom^\bullet_{\CC}(N, (\cdot)|_H)^{K_H\text{-alg}}. \]
	As seen in the proof of Lemma \ref{prop:restriction-lemma}, this right adjoint is exact, thus the claim follows.
	
	All in all, $\RHom_{\mathfrak{h}, K^H}\left( M|_H, N^\vee \right)$ is represented by the complex
	\[ {}^{\mathrm{h}} \Hom^\bullet_{\mathfrak{h}, K^H}\left(P|_H , N^\vee \right) \simeq {}^{\mathrm{h}} \Hom^\bullet_{\mathfrak{h}, K^H}\left(P|_H \otimes N, \CC \right) \]
	where the aforementioned dg-adjunction is applied once again. By the claim, the right-hand side represents $\RHom_{\mathfrak{h}, K^H}\left((M \boxtimes N)|_H, \CC \right)$.
\end{proof}

Cf.\ \cite[Proposition 2.6]{Pra18} for the case of $p$-adic groups.

Under the assumptions of reductivity, the study of general $\Ext$-branching
\[ {}^{\mathrm{h}} \Ext^n_{\mathfrak{h}, K^H}(M|_H, N^\vee) \]
for various $M$, $N$ thus reduces to the special case ${}^{\mathrm{h}} \Ext^n_{\mathfrak{h}, K^H}(M|_H, \CC)$, once we replace $\mathfrak{i}$ (resp.\ $K^I$) by $\mathfrak{g} := \mathfrak{i} \times \mathfrak{h}$ (resp.\ $K := K^I \times K^H$). This recipe is certainly well-known, at least in degree $n=0$ and for complexes of $(\mathfrak{g}, K)$-modules instead of h-complexes.

We record another easy fact about $\Ext^n_{\mathfrak{h}, K^H}(\cdot, \CC)$.

\begin{proposition}[{\cite[Corollary 3.2]{KV95}}]\label{prop:Ext-H}
	Consider a subpair $(\mathfrak{h}, K^H)$ of $(\mathfrak{g}, K)$ and assume $K^H$ and $K$ are reductive. Let $V$ be a $(\mathfrak{g}, K)$-module. There are canonical isomorphisms
	\[ \Ext^n_{\mathfrak{h}, K^H}(V|_H, \CC) \simeq \Hm_n\left(\mathfrak{h}, K^H; V|_H \right)^* \]
	for each $n \in \Z$, where $\Hm_n\left(\mathfrak{h}, K^H; \cdot \right)$ is the relative Lie algebra homology.
\end{proposition}

Note that ${}^{\mathrm{h}} \Ext_{\mathfrak{h}, K^H} \simeq \Ext_{\mathfrak{h}, K^H}$ by Corollary \ref{prop:BL-equiv-Ext}. The upcoming results about $\Ext^n$ can all be viewed as assertions about $\Hm^n\left(\mathfrak{h}, K^H; V \right)^*$.

\subsection{Branching and localization}
Consider a connected reductive group $G$ and its subgroups
\begin{equation}\label{eqn:four-groups}\begin{tikzcd}
	H \arrow[phantom, r, "\subset" description] & G \\
	K^H \arrow[phantom, u, "\subset" description, sloped] \arrow[phantom, r, "\subset" description] & K \arrow[phantom, u, "\subset" description, sloped].
\end{tikzcd}\end{equation}
We assume that $H$, $K$ and $K^H$ are all reductive.

\begin{lemma}\label{prop:Hom12}
	The following diagram commutes up to canonical isomorphism
	\[\begin{tikzcd}[column sep=large]
		{}^{\mathrm{h}} \cate{C}(\mathfrak{h}, K^H) \arrow[r, "{\mathrm{coInv}^{\mathfrak{h}, K^H}_{\CC, K^H}}"] \arrow[rd, "{\Hom_1}"'] & {}^{\mathrm{h}} \cate{C}(\CC, K^H) \arrow[d, "{\Hom_2}"] \\
		& \cate{C}(\CC)^{\mathrm{op}}
	\end{tikzcd}\]
	where
	\begin{align*}
		\Hom_1 & := {}^{\mathrm{h}} \Hom^\bullet_{\mathfrak{h}, K^H}(\cdot, \CC), \\
		\Hom_2 & := {}^{\mathrm{h}} \Hom^\bullet_{\CC, K^H}(\cdot, \CC),
	\end{align*}
	and all actions upon $\CC$ are defined to be trivial.
\end{lemma}
\begin{proof}
	All functors in view are dg-functors, and it amounts to establishing a canonical isomorphism in $\cate{C}(\CC)$
	\[ {}^{\mathrm{h}} \Hom^\bullet_{\CC, K^H}\left( \mathrm{coInv}^{\mathfrak{h}, K^H}_{\CC, K^H}(N), \CC\right) \simeq {}^{\mathrm{h}} \Hom^\bullet_{\mathfrak{h}, K^H}(N, \CC) \]
	where $N$ is any h-complex over $(\mathfrak{h}, K^H)$. Indeed, this follows from the adjunction between co-invariants and inflation (Proposition \ref{prop:Inv-coInv}); the adjunction extends to the dg-level, cf.\ the proof of Proposition \ref{prop:K-injectives-adjunction}.
\end{proof}

Consider the affine homogeneous $G$-space
\[ X := H \backslash G, \quad x := H \cdot 1. \]
Thus $x$ is a $K^H$-fixed point in $X$, and we have the corresponding $K^H$-equivariant morphism
\[ i_x: \mathrm{pt} \to X. \]
Note that $D_{\mathrm{pt}} = \CC$, with trivial $K^H$-action.

To save space, we will omit the oblivion ${}^{\mathrm{h}}\cate{D}(D_X, K) \to {}^{\mathrm{h}}\cate{D}(D_X, K^H)$ in the statements below, and adopt the notation $M \mapsto M|_H$ of Definition \ref{def:M-H} on the level of h-derived categories.

\begin{proposition}\label{prop:RHom-RHom}
	For all objects $M$ of ${}^{\mathrm{h}} \cate{D}^-(\mathfrak{g}, K)$, there are canonical isomorphisms in $\cate{D}^+(\CC)$:
	\[ \RHom_{\mathfrak{h}, K^H}(M|_H, \CC) \simeq \RHom_{D_{\mathrm{pt}}, K^H}\left( i_x^\bullet( \mathbf{Loc}_X(M) ), \CC \right). \]
	Here $i_x^\bullet: {}^{\mathrm{h}} \cate{D}^-(D_X, K^H) \to {}^{\mathrm{h}} \cate{D}^-(D_{\mathrm{pt}}, K^H)$ is the inverse image functor introduced in \S\ref{sec:inverse-image}. 
\end{proposition}
\begin{proof}
	The functor ${}^{\mathrm{h}} \cate{C}(\mathfrak{g}, K) \to {}^{\mathrm{h}} \cate{C}(\mathfrak{h}, K^H)$ preserves K-projectives by Lemma \ref{prop:restriction-lemma}, hence 
	the left-hand side is the left h-derived functor of the composition of
	\[ {}^{\mathrm{h}} \cate{C}(\mathfrak{g}, K) \to {}^{\mathrm{h}} \cate{C}(\mathfrak{h}, K^H) \xrightarrow{\mathrm{coInv}^{\mathfrak{h}, K^H}_{\CC, K^H}} {}^{\mathrm{h}} \cate{C}(\CC, K^H) \xrightarrow{\Hom_2} \cate{C}(\CC)^{\mathrm{op}}, \]
	in the notation of Lemma \ref{prop:Hom12}.
	
	Since $\mathrm{coInv}^{\mathfrak{h}, K^H}_{\CC, K^H}$ also preserves K-projectives by Proposition \ref{prop:Inv-coInv-K}, one can take left h-derived functors in stages. Using Proposition \ref{prop:Loc-coInv}, the result is isomorphic to the composition of
	\[ i_x^\bullet \left( \mathbf{Loc}_X \right): {}^{\mathrm{h}} \cate{D}^-(\mathfrak{g}, K) \to {}^{\mathrm{h}} \cate{D}^-(\CC, K^H) \]
	with $\Lder(\Hom_2): {}^{\mathrm{h}} \cate{D}^-(\CC, K^H) \to \cate{D}^+(\CC)^{\mathrm{op}}$. However, $\Lder(\Hom_2)$ is just $\RHom_{D_{\mathrm{pt}}, K^H}(\cdot, \CC)$.
\end{proof}

\begin{corollary}
	For all objects $M$ of $\cate{D}^{\mathrm{b}}(\mathfrak{g}, K)$ and all $n \in \Z$, there are canonical isomorphisms:
	\[ \Ext^n_{\mathfrak{h}, K^H}(M|_H, \CC) \simeq \Ext^n_{\cate{D}^{\mathrm{b}}_{K^H}(\mathrm{pt})}\left( i_x^\bullet( \varepsilon \mathbf{Loc}_X(M) ), \CC \right). \]
	Here $\varepsilon$ is the equivalence in Theorem \ref{prop:Beilinson-equiv}.
\end{corollary}
\begin{proof}
	Combine the Bernstein--Lunts equivalence (Corollary \ref{prop:BL-equiv-Ext}) with Beilinson's equivalence $\varepsilon$ (Corollary \ref{prop:Beilinson-equiv-Ext}), noting that $\varepsilon$ is compatible with inverse images (Proposition \ref{prop:inverse-image-compatibility}).
\end{proof}

One can also interpret the relative Lie algebra homologies in terms of localization.

\begin{proposition}\label{prop:H-coInv}
	Let $V$ be a $(\mathfrak{g}, K)$-module. There are canonical isomorphisms
	\[ \Hm_n(\mathfrak{h}, K^H; V|_H) \simeq \Hm^{-n}\Lder\left( \mathrm{coInv}^{\CC, K^H}_{\CC, \{1\}} \right) \left(i_x^\bullet \mathbf{Loc}_X(V)\right), \quad n \in \Z. \].
\end{proposition}
\begin{proof}
	The arguments from Proposition \ref{prop:RHom-RHom} shows that the right-hand side is the $\Hm^{-n}$ of the composition
	\begin{equation*}
		{}^{\mathrm{h}} \cate{D}^-(\mathfrak{g}, K) \to {}^{\mathrm{h}} \cate{D}^-(\mathfrak{h}, K^H) \xrightarrow{\Lder\left(\mathrm{coInv}^{\mathfrak{h}, K^H}_{\CC, K^H}\right)}
		{}^{\mathrm{h}} \cate{D}^-(\CC, K^H) \xrightarrow{\Lder\left(\mathrm{coInv}^{\CC, K^H}_{\CC, \{1\}}\right)} \cate{D}^-(\CC).
	\end{equation*}

	Taking co-invariants preserves K-projectives. By the transitivity \eqref{eqn:Inv-coInv-transitive}, the composition above folds into
	\[ {}^{\mathrm{h}} \cate{D}^-(\mathfrak{g}, K) \to {}^{\mathrm{h}} \cate{D}^-(\mathfrak{h}, K^H) \xrightarrow{\Lder\left(\mathrm{coInv}^{\mathfrak{h}, K^H}_{\CC, \{1\}}\right)}
	\cate{D}^-(\CC). \]
	It remains to show that after taking $\Hm^{-n}$, the second arrow gives $\Hm_n(\mathfrak{h}, K^H; W)$ when applied to any $(\mathfrak{h}, K^H)$-module $W$.
	
	Use transitivity to break $\mathrm{coInv}^{\mathfrak{h}, K^H}_{\CC, \{1\}}$ into two stages
	\[ {}^{\mathrm{h}} \cate{C}(\mathfrak{h}, K^H) \xrightarrow{\mathrm{coInv}_1} \cate{C}(\mathfrak{h}, K^H) \xrightarrow{\mathrm{coInv}_2} \cate{C}(\CC), \]
	cf.\ Example \ref{eg:h-inflation} for $\mathrm{coInv}_1$. The left derived functors are decomposed accordingly. Now take the standard resolution $W \otimes N\mathfrak{h} \to W$ in ${}^{\mathrm{h}} \cate{C}(\mathfrak{h}, K^H)$ (see \S\ref{sec:std-resolution}). Unwinding the construction of $\mathrm{coInv}_1$, one obtains
	\[ \mathrm{coInv}_1(W \otimes N\mathfrak{g}) \simeq W \otimes \mathrm{coInv}_1(N\mathfrak{g}) \]
	in $\cate{C}(\mathfrak{h}, K^H)$. It is shown in \cite[Proposition 3.2.7]{Pan07} that $\mathrm{coInv}_1(N\mathfrak{g})$ is the standard resolution of $\CC$ in $\cate{C}(\mathfrak{h}, K^H)$, i.e.\ the relative standard complex. Hence $\mathrm{coInv}_2 (W \otimes \mathrm{coInv}_1(N\mathfrak{g}))$ represents the relative Lie algebra homologies of $W$, as desired.
\end{proof}

\subsection{Some consequences of regularity}\label{sec:consequence-regularity}
Throughout this section, we let $G$ be a connected reductive group, $X$ be a spherical affine homogeneous $G$-space, and $K$ be a spherical reductive subgroup of $G$. We consider a point $x$ of $X$ and a reductive subgroup $K^H$ of $K$ that stabilizes $x$. In particular, $i_x: \mathrm{pt} \hookrightarrow X$ is $K^H$-equivariant.

By taking $H := \Stab_G(x)$ to identify $X$ with $H \backslash G$, the situation is like \eqref{eqn:four-groups}; the extra assumption here is that $H$ and $K$ are both spherical.

For every smooth $K^H$-variety $Y$, let $\cate{D}^{\mathrm{b}}_{K^H, \mathrm{cons}}(Y)$ denote the bounded $K^H$-equivariant derived category of constructible sheaves defined by Bernstein--Lunts \cite{BL94}, equipped with perverse $t$-structure. Assuming $Y$ is affine, we have equivalences
\[ {}^{\mathrm{h}} \cate{D}^{\mathrm{b}}_{\mathrm{rh}}(D_Y, K^H) \simeq \cate{D}^{\mathrm{b}}_{K^H, \mathrm{rh}}(Y) \simeq \cate{D}^{\mathrm{b}}_{K^H, \mathrm{cons}}(Y). \]
\begin{itemize}
	\item The first one is Beilinson's equivalence (Theorem \ref{prop:Beilinson-equiv}).
	\item The second one is the equivariant Riemann--Hilbert correspondence; see \cite[4.2]{BL94} or \cite[Theorem 4.6.2]{Ka08}.
\end{itemize}

\begin{definition}
	Let
	$\begin{tikzcd}
		{}^{\mathrm{h}} \cate{D}^{\mathrm{b}}_{\mathrm{rh}}(D_X, K^H) \arrow[r, shift left, "{i_x^*, i_x^!}"] &
		{}^{\mathrm{h}} \cate{D}^{\mathrm{b}}_{\mathrm{rh}}(D_{\mathrm{pt}}, K^H) \arrow[shift left, l, "{i_{x, *}}"]
	\end{tikzcd}$
	be the functors that correspond to the synonymous functors between $\cate{D}^{\mathrm{b}}_{K^H, \mathrm{cons}}(X)$ and $\cate{D}^{\mathrm{b}}_{K^H, \mathrm{cons}}(\mathrm{pt})$.
\end{definition}

Therefore we have
\begin{align*}
	i_x^! & = i_x^\bullet[-\dim X], \\
	i_x^* & = i_x^![2\dim X] = i_x^\bullet [\dim X] \quad \text{when}\; K^H = \{1\},
\end{align*}
and $i_x^*$ is left adjoint to $i_{x, *}$. For the relation between $i_x^!$ and $i_x^\bullet$, see eg.\ \cite[Theorem 7.1.1]{HTT08}. For the relation between the $i_x^*$ and $i_x^!$ in the non-equivariant case, see eg.\ \cite[p.8]{BL94}.

\begin{theorem}\label{prop:RHom-Loc}
	Let $M$ be an object of the category ${}^{\mathrm{h}}\cate{D}^{\mathrm{b}}_{\mathrm{HC}}(\mathfrak{g}, K)$ introduced in Corollary \ref{prop:regularity-gen}. There is a canonical isomorphism
	\begin{equation*}
		\RHom_{\mathfrak{h}, K^H}(M|_H, \CC) \simeq \RHom_{D_{\mathrm{pt}}, K^H}\left( i_x^! \mathbf{Loc}_X(M)[\dim X], \CC \right).
	\end{equation*}
	When $K^H = \{1\}$, this is also isomorphic to
	\begin{equation*}
		\RHom_{D_{\mathrm{pt}}}\left( i_x^* \mathbf{Loc}_X(M)[-\dim X], \CC \right)
		\simeq \RHom_{D_X}\left( \mathbf{Loc}_X(M), i_{x, *}(\CC)[\dim X] \right).
	\end{equation*}
	All these complexes of $D$-modules are bounded with regular holonomic cohomologies.
\end{theorem}
\begin{proof}
	By Corollary \ref{prop:regularity-gen}, $\mathbf{Loc}_X(M)$ is in ${}^{\mathrm{h}} \cate{D}^{\mathrm{b}}_{\mathrm{rh}}(D_X, K)$. It remains to apply Proposition \ref{prop:RHom-RHom} and the relations recalled above.
\end{proof}

Note that by \cite[Examples 1.5.23 and 1.6.4]{HTT08}, $i_{x, *}(\CC)$ is the $D_X$-module generated by the Dirac measure at $x$.

In order to apply Theorem \ref{prop:RHom-Loc} to concrete problems, one needs a deeper, quantitative understanding of $\mathbf{Loc}_X(M)$. We only give some crude applications below, which bypasses this issue.

\begin{corollary}\label{prop:Ext-consequence-1}
	Let $M$ be as in Theorem \ref{prop:RHom-Loc}. Then ${}^{\mathrm{h}} \Ext^n_{\mathfrak{h}, K^H}(M|_H, \CC)$ is finite-dimensional for all $n \in \Z$. It vanishes for $|n| \gg 0$.
\end{corollary}
\begin{proof}
	To prove the vanishing for $|n| \gg 0$, simply take the standard resolution of $M$.
	
	To prove the finiteness of ${}^{\mathrm{h}}\Ext^n_{\mathfrak{h}, K^H}(M|_H, \CC)$, it suffices to show that for every object $N$ of ${}^{\mathrm{h}}\cate{D}^{\mathrm{b}}_{\mathrm{rh}}(D_{\mathrm{pt}}, K^H)$, we have
	\[ \dim {}^{\mathrm{h}}\Ext^n_{D_{\mathrm{pt}}, K^H}\left(N, \CC\right) < +\infty. \]

	There are at least two ways to see this. (i) Working in $\cate{D}^{\mathrm{b}}_{K^H, \mathrm{cons}}(\mathrm{pt})$, use the constructibility of $\mathrm{R}\mathscr{H}\mathit{om}$ (see \cite[Desideratum 6.4.1]{Ac21}) together with the functor $\mathrm{Inv}^{\leq m}_{K^H, *}$ of truncated invariants introduced in \cite[Proposition 6.6.7]{Ac21} ($m \gg 0$) to reach $\Ext^n$. (ii) Use $m$-acyclic resolutions and holonomicity to access $\Ext^n$: see \cite[Proposition 2.13.1]{BL95} together with the proof of \cite[Theorem 2.13]{BL95}.
\end{proof}

In particular, the Euler--Poincaré characteristic for branching laws
\begin{equation}\label{eqn:EP-branching}
	\mathrm{EP}_{\mathfrak{h}, K^H}(M|_H, \CC) := \sum_n (-1)^n \dim {}^{\mathrm{h}}\Ext^n_{\mathfrak{h}, K^H}(M|_H, \CC)
\end{equation}
is well-defined, for all object $M$ of ${}^{\mathrm{h}}\cate{D}^{\mathrm{b}}_{\mathrm{HC}}(\mathfrak{g}, K)$.

\begin{theorem}\label{prop:local-index}
	Take $K^H = \{1\}$ in the formalism above. Let $M$ be in ${}^{\mathrm{h}}\cate{D}^{\mathrm{b}}_{\mathrm{HC}}(\mathfrak{g}, K)$ and set $\mathcal{L} := \mathbf{Loc}_{X, \{1\}}(M)$. Define the solution complex $\mathrm{Sol}_X(\mathcal{L}) \in \cate{D}^{\mathrm{b}}_{\mathrm{cons}}(X)$ of $\mathcal{L}$ by \cite[p.118]{HTT08}. Then
	\[ \RHom_{\mathfrak{h}}(M|_H, \CC) \simeq i_x^* \mathrm{Sol}_X(\mathcal{L}). \]
	Consequently, the $\mathrm{EP}_{\mathfrak{h}, \{1\}}(M|_H, \CC)$ in \eqref{eqn:EP-branching} equals the local Euler--Poincaré characteristic
	\[ \chi_x\left( \mathrm{Sol}_X(\mathcal{L}) \right) \]
	of $\mathrm{Sol}_X(\mathcal{L})$ at $x$, which can be expressed in terms of characteristic cycles and Euler obstructions by Kashiwara's local index theorem \cite[Theorem 4.6.7]{HTT08}.
\end{theorem}
\begin{proof}
	By \cite[Proposition 4.7.4]{HTT08} we have
	\[ \mathrm{Sol}_X(\mathcal{L})[\dim X] \simeq \mathrm{DR}_X\left( \mathbb{D}_X \mathcal{L} \right) \]
	canonically, where $\mathrm{DR}_X$ (resp.\ $\mathbb{D}_X$) denotes the non-equivariant de Rham functor (resp.\ duality endo-functor) for $X$. Since $\mathrm{DR}_{\mathrm{pt}} = \identity$, the Riemann--Hilbert correspondence leads to
	\[ i_x^* \mathrm{Sol}_X(\mathcal{L}) \simeq \left( \mathbb{D}_{\mathrm{pt}} i_x^! \mathcal{L} \right)[-\dim X] \simeq \mathbb{D}_{\mathrm{pt}} \left( i_x^! \mathcal{L}[\dim X] \right). \]
	
	Now put $K^H = \{1\}$ in Theorem \ref{prop:RHom-Loc} to infer that $i_x^* \mathrm{Sol}_X(\mathcal{L}) \simeq \RHom_{\mathfrak{h}}(M|_H, \CC)$. The remaining assertions follow at once.
\end{proof}

The isomorphism in Theorem \ref{prop:local-index} generalizes \cite[Proposition 10.2]{Li22}.

We now specialize to the case when $M$ is concentrated in degree zero.

\begin{corollary}\label{prop:Ext-consequence-2}
	For every Harish-Chandra $(\mathfrak{g}, K)$-module $V$ and $n \geq 0$, the dimension of
	\[ {}^{\mathrm{h}} \Ext^n_{\mathfrak{h}, K^H}(V|_H, \CC) \simeq \Ext^n_{\mathfrak{h}, K^H}(V|_H, \CC) \]
	is finite. It also equals $\dim \Hm_n(\mathfrak{h}, K^H; V|_H)$.
\end{corollary}
\begin{proof}
	The isomorphism in the first part is Corollary \ref{prop:BL-equiv-Ext}. The second part follows from Corollary \ref{prop:Ext-H}.
\end{proof}

The finiteness in Corollary \ref{prop:Ext-consequence-2} is also a consequence of a more general result of M.\ Kitagawa \cite[Fact 4.7]{Ki21}; a uniform bound is given in \cite[Corollary 7.17]{Ki21}. In the case $K^H = \{1\}$, the finiteness of $\dim \Hm_n(\mathfrak{h}; V|_H)$ is shown in \cite[Proposition 4.2.2]{AGKL16} when $K$ is a symmetric subgroup in good position relative to $H$.

Likewise, Proposition \ref{prop:H-coInv} and Corollary \ref{prop:regularity-gen} lead to the following statement for relative Lie algebra homologies.

\begin{proposition}\label{prop:H-Loc}
	Let $V$ be a Harish-Chandra $(\mathfrak{g}, K)$-module. There are canonical isomorphisms
	\begin{equation*}
		\Hm_n(\mathfrak{h}, K^H; V|_H) \simeq \Hm^{- n + \dim X}\Lder\left( \mathrm{coInv}^{\CC, K^H}_{\CC, \{1\}} \right) \left(i_x^! \mathbf{Loc}_X(V)\right)
	\end{equation*}
	for all $n \in \Z$; when $K^H = \{1\}$, it is also isomorphic to
	\begin{equation*}
		\Hm^{-n - \dim X} \left(i_x^* \mathbf{Loc}_X(V)\right).
	\end{equation*}
\end{proposition}

To better understand the effect of $\Hm^{-n + \dim X} \Lder\left( \mathrm{coInv}^{\CC, K^H}_{\CC, \{1\}} \right)$, we remark that on the constructible side, it matches the composition
\[ \cate{D}^{\mathrm{b}}_{K^H, \mathrm{cons}}(\mathrm{pt}) \xrightarrow{\mathrm{Inv}^{\geq m}_{K^H, !}} {}^{\mathrm{p}} \cate{D}^-_{\mathrm{cons}}(\mathrm{pt})^{\geq m} \xrightarrow{{}^{\mathrm{p}} \Hm^{-n \pm \dim X}} \left\{\text{f.d.}\; \CC \text{-vector spaces} \right\} \]
for $m \leq -n + \dim X$, where $\mathrm{Inv}^{\geq m}_{K^H, !}$ stands for the functor of truncated co-invariants in \cite[Proposition 6.6.7]{Ac21}. This can be deduced from Proposition \ref{prop:Infl-epsilon} and adjunction. For a detailed discussion, see \cite[\S 6.6]{Ac21}.

From this, the finiteness of $\dim \Hm_n(\mathfrak{h}, K^H; V|_H)$ can be deduced directly from Proposition \ref{prop:H-Loc}.

\subsection{The monodromic setting}\label{sec:monodromic}
Let $G$ and $X = H \backslash G$ be as in \S\ref{sec:consequence-regularity}, so $H$ is reductive. In addition, we fix a character of Lie algebra
\[ \chi: \mathfrak{h} \to \CC. \]
It must factor through the Lie algebra of a torus quotient, say $S = H/\underline{H}$ where $\underline{H} \lhd H$. Accordingly, there is a morphism of $G$-varieties
\[ \pi: \tilde{X} := \underline{H} \backslash G \to H \backslash G = X. \]
Note that $S$ acts on the left of $\tilde{X}$, commuting with $G$ and makes $\pi$ into an $S$-torsor. Since $S$ is a torus, $\pi$ is even Zariski-locally trivial. Also note that $\tilde{X}$ is affine since $X$ is. In \cite[2.2]{BL95}, $\pi$ is said to make $X$ into an \emph{$S$-monodromic $G$-variety}.

The character $\chi$ corresponds to an element of $\mathfrak{s}^*$. Let $\mathfrak{m}_\chi \subset \Sym(\mathfrak{s})$ be the corresponding maximal ideal; its image in $\widetilde{\mathscr{D}}_X := (\pi_* \mathscr{D}_{\tilde{X}})^S$ is central. Recall the following notion from \cite[p.24]{BB93} (see also \cite{Li22}).

\begin{definition}
	Set $\mathscr{D}_{X, \chi} :=  \widetilde{\mathscr{D}}_X / \mathfrak{m}_\chi \widetilde{\mathscr{D}}_X$ be the sheaf of algebras of twisted differential operators (TDO's) attached to $\chi$. Put $D_{X, \chi} := \Gamma(X, \mathscr{D}_{X, \chi})$.
\end{definition}

By trivializing $\pi$ locally over $X$, we see that $\mathscr{D}_{X, \chi}$ is a locally trivial sheaf of TDO's. Taking $\chi=0$ reverts to the untwisted version $\mathscr{D}_X$.

\begin{remark}\label{rem:monodromic-inverse-image}
	For later use, we remark that taking inverse images realizes an equivalence from $\mathscr{D}_{X, \chi}\dcate{Mod}$ to the category of $(S, \chi)$-monodromic $\mathscr{D}_{\tilde{X}}$-modules, the latter being a twisted version of $S$-equivariance; see \cite[(2.3) and Proposition 2.8]{Li22} or \cite[1.8.10 Lemma]{BB93} for details. A further inverse image via $G \to \underline{H} \backslash G = \tilde{X}$ brings us to $(H^{\mathrm{op}}, \chi)$-monodromic $\mathscr{D}_G$-modules.
\end{remark}

The actions of $G$ and $S$ commute, hence we obtain $G$-equivariant homomorphisms $U(\mathfrak{g}) \to \widetilde{\mathscr{D}}_X \to \mathscr{D}_{X, \chi}$. Since $X$ is affine, we may and will work with $D_{X, \chi}$-modules instead. The previous constructions now become
\[ D_{X, \chi} \simeq D_{\tilde{X}}^S / \mathfrak{m}_\chi D_{\tilde{X}}^S. \]
We obtain a natural $G$-equivariant homomorphism $j: U(\mathfrak{g}) \to D_{X, \chi}$.

Given a reductive subgroup $K \subset G$, there is also a monodromic version of Beilinson's equivalence (Theorem \ref{prop:BL-equiv}), identifying ${}^{\mathrm{h}} \cate{D}^{\mathrm{b}}(D_{X, \chi}, K)$ with the monodromic equivariant derived category $\cate{D}^{\mathrm{b}}_{K, \chi}(X)$ of $D$-modules; see \cite[2.14]{BL95}.

\begin{remark}
	Since we fixed $\chi$, there is no need to consider $(\tilde{\mathscr{D}}_X, F|K)$-modules as in \cite[2.5]{BL95} where $F := K \times S$. This corresponds to taking $I = \mathfrak{m}_\chi$ in the formalism of \cite[2.3]{BL95}.
\end{remark}

In the context of TDO's arising from monodromic structures, the notions of holonomicity and regularity can be found in \cite[\S 7.14]{Ka08}\footnote{The author is indebted to Masatoshi Kitagawa for clarifications on this point.}. They reduce to the untwisted theory by taking inverse images via $\pi$. We define ${}^{\mathrm{h}} \cate{D}^{\mathrm{b}}_{\mathrm{rh}+}(D_{X, \chi}, K)$ accordingly, cf.\ Definition \ref{def:rh-plus}.

The derived monodromic localization functor
\begin{equation}
	\mathbf{Loc}_{X, \chi} = \mathbf{Loc}_{X, K, \chi}: {}^{\mathrm{h}} \cate{D}(\mathfrak{g}, K) \to {}^{\mathrm{h}} \cate{D}(D_{X, \chi}, K)
\end{equation}
can still be defined, with amplitude in $[-\dim G, 0]$.

The following construction is needed for the monodromic counterpart of Theorem \ref{prop:regularity}. Define the commutative algebra $\mathcal{Z}(\tilde{X})$ for the $S \times G$-variety $\tilde{X}$ by applying Definition \ref{def:ZX}. There is a homomorphism
\[ \Sym(\mathfrak{s}) \otimes \mathcal{Z}(\mathfrak{g}) \simeq \mathcal{Z}(\mathfrak{s} \times \mathfrak{g}) \to \mathcal{Z}(\tilde{X}) \subset D_{\tilde{X}}^{S \times G} \subset D_{\tilde{X}}^S \]
of algebras. Define
\begin{equation*}
	\mathcal{Z}_\chi(X) := \mathcal{Z}(\tilde{X}) / \mathfrak{m}_\chi \mathcal{Z}(\tilde{X}).
\end{equation*}
We obtain homomorphisms
\[ \mathcal{Z}(\mathfrak{g}) \to \mathcal{Z}_\chi(X) \to D_{X, \chi}^G \subset D_{X, \chi}. \]

Thus $\mathcal{Z}_\chi(X)$ acts on the right of $\mathbf{Loc}_{X, \chi}$ through $\mathcal{Z}_\chi(X) \to D_{X, \chi}^G$, as in the non-twisted case \eqref{eqn:Z-action}.

\begin{lemma}\label{prop:Knop-monodromic}
	The algebra $\mathcal{Z}_\chi(X)$ is commutative. Moreover, $D_{X, \chi}$ is projective over $\mathcal{Z}_\chi(X)$ and $\mathcal{Z}_\chi(X)$ is finitely generated over $\mathcal{Z}(\mathfrak{g})$, both as left and right modules.
\end{lemma}
\begin{proof}
	We know that $\mathcal{Z}(\tilde{X})$ is commutative, hence so is $\mathcal{Z}_\chi(X)$.
	
	Next, the projectivity and finite generation as left and right modules hold on the level of
	\[ \mathcal{Z}(\mathfrak{s} \times \mathfrak{g}) \to \mathcal{Z}(\tilde{X}) \hookrightarrow D_{\tilde{X}} \]
	by \cite{Kn94}. Note that $\mathcal{Z}(\tilde{X})$ is actually included in $D_{\tilde{X}}^S$. We claim that $D_{\tilde{X}}^S$ is also projective as left and right $\mathcal{Z}(\tilde{X})$-modules.
	
	Since $S$ is reductive, acting algebraically on $D_{\tilde{X}}$ and trivially on $\mathcal{Z}(\tilde{X})$, one can realize $D_{\tilde{X}}^S$ canonically as a direct $\mathcal{Z}(\tilde{X})$-summand of $D_{\tilde{X}}$, thus the claim follows at once. This can also be seen by identifying $S$-action on $D_{\tilde{X}}$ with a grading by the cocharacter lattice of $S$.
	
	By the claim, we may tensor the homomorphisms
	\[ \mathcal{Z}(\mathfrak{s} \times \mathfrak{g}) \to \mathcal{Z}(\tilde{X}) \hookrightarrow D_{\tilde{X}}^S \]
	with $\Sym(\mathfrak{s}) / \mathfrak{m}_\chi$ over $\Sym(\mathfrak{s})$ on the left or right. The required projectivity and finite generation persist.
\end{proof}

\begin{theorem}\label{prop:regularity-monodromic}
	Suppose that $X$ is affine and spherical, $K$ is a reductive spherical subgroup of $G$, and $V$ is a Harish-Chandra $(\mathfrak{g}, K)$-module. Then $\mathbf{Loc}_{X, \chi}(V)$ lies in ${}^{\mathrm{h}} \cate{D}^{\mathrm{b}}_{\mathrm{rh}+}(D_{X, \chi}, K)$.
\end{theorem}
\begin{proof}	
	Recall the strategy in the non-monodromic case in \S\S\ref{sec:regularity-criterion}--\ref{sec:end-of-regularity}. The key criterion from \cite{Li22} applies in the monodromic setting. In fact, after ascending to $G$ by Remark \ref{rem:monodromic-inverse-image}, it boils down to showing that a finitely generated, $(H^{\mathrm{op}} \times K, \chi \otimes \mathrm{triv})$-monodromic and locally $\mathcal{Z}(\mathfrak{g} \times \mathfrak{g})$-finite $D_G$-module is regular holonomic, and this is indeed covered by \cite[Theorem 5.6]{Li22}. Note that $\chi$ is required to be trivial on the nilpotent radical of $\mathfrak{h}$ in \textit{loc.\ cit.}, but $H$ is reductive here.

	The other ingredients in the proof carry over, by using Lemma \ref{prop:Knop-monodromic}. The bound on the characteristic varieties are obtained in the same way as \cite[Proposition 3.4]{Li22}.
\end{proof}

Likewise, Corollary \ref{prop:regularity-gen} and Proposition \ref{prop:ZX-locally-finite} also admit monodromic versions.

Given the base-point $x$ of $X$, the inclusion map $i_x$ gives rise to an inclusion of $S$-monodromic $K^H$-varieties, given by the Cartesian square
\[\begin{tikzcd}
	\pi^{-1}(x) \arrow[r] \arrow[d] & \tilde{X} \arrow[d, "\pi"] \\
	\mathrm{pt} \arrow[r, "{i_x}"'] & X.
\end{tikzcd}\]

Next, suppose that
\begin{itemize}
	\item $K^H$ is a reductive subgroup of $K$ that fixes $x$,
	\item a character $K^H \to \Gm$ is given, whose derivative coincides with $\chi|_{\mathfrak{k}^H}$.
\end{itemize}
Therefore $\chi$ can be viewed as a $(\mathfrak{h}, K^H)$-module.

Note that $D_{\mathrm{pt}, \chi} \simeq \Sym(\mathfrak{s}) / \mathfrak{m}_\chi$. Define the $K^H$-equivariant $D_{\mathrm{pt}, \chi}$-module $\CC_\chi$ by letting $\mathfrak{s}$ (resp.\ $K^H$) act on $\CC$ through $\chi$ (resp.\ the given $K^H \to \Gm$). Below is the monodromic counterpart of Proposition \ref{prop:RHom-RHom}.

\begin{proposition}
	For all objects $M$ of ${}^{\mathrm{h}} \cate{D}^-(\mathfrak{g}, K)$, there are canonical isomorphisms in $\cate{D}^+(\CC)$:
	\[ \RHom_{\mathfrak{h}, K^H}(M|_H, \chi) \simeq \RHom_{D_{\mathrm{pt}, \chi}, K^H}\left( i_x^\bullet( \mathbf{Loc}_{X, \chi}(M) ), \CC_\chi \right) \]
	where $i_x^\bullet$ is the h-derived inverse image for $K^H$-equivariant monodromic $D$-modules.
\end{proposition}
\begin{proof}
	Same as the non-monodromic case. It suffices to replace $\CC$ by $\CC_\chi$, then apply the following analog of Proposition \ref{prop:Loc-coinv}.
\end{proof}

\begin{proposition}\label{prop:coinv-Loc-monodromic}
	For every $\mathfrak{g}$-module $V$, its space of $(\mathfrak{h}, \chi)$-co-invariants is isomorphic to $\CC \dotimes{\mathscr{O}_{X, x}} \left( \mathscr{D}_{X, \chi} \dotimes{U(\mathfrak{g})} V\right)$, by sending the image of $v \in V$ to $1 \otimes (1 \otimes v)$.
\end{proposition}
\begin{proof}
	Same as the one given in \cite[Lemma 2.2]{BZG19}. It suffices to replace the map $\mathscr{O}_X \otimes \mathfrak{g} \to \mathscr{D}_X$ by $\mathscr{O}_X \otimes \mathfrak{g} \to \mathscr{D}_{X, \chi}$. 
\end{proof}

Below is the monodromic counterpart of Theorem \ref{prop:RHom-Loc}. Define the functors $i_x^!$, $i_x^*$ and $i_{x, *}$ in the same manner as in \S\ref{sec:consequence-regularity}.

\begin{theorem}
	Let $M$ be an object of the category ${}^{\mathrm{h}}\cate{D}^{\mathrm{b}}_{\mathrm{HC}}(\mathfrak{g}, K)$. There are canonical isomorphisms
	\begin{equation*}
		\RHom_{\mathfrak{h}, K^H}(M|_H, \chi) \simeq \RHom_{D_{\mathrm{pt}, \chi}, K^H}\left( i_x^! \mathbf{Loc}_{X, \chi}(M)[\dim X], \CC_\chi \right).
	\end{equation*}
	When $K^H = \{1\}$, this is also isomorphic to
	\begin{multline*}
		\RHom_{D_{\mathrm{pt}, \chi}}\left( i_x^* \mathbf{Loc}_{X, \chi}(M)[-\dim X], \CC_\chi \right) \\
		\simeq \RHom_{D_{X, \chi}}\left( \mathbf{Loc}_{X, \chi}(M), i_{x, *}(\CC_\chi)[\dim X] \right).
	\end{multline*}
	All these complexes of $D$-modules are bounded with regular holonomic cohomologies.
\end{theorem}
\begin{proof}
	Since the Riemann--Hilbert correspondence has a monodromic version \cite[Theorem 3.16.2]{Ka89}, and similarly in the equivariant case, one can repeat the proof of Theorem \ref{prop:RHom-Loc}. Caution: the monodromic Riemann---Hilbert correspondence lands in the derived category of ``twisted sheaves''.
\end{proof}

Accordingly, we deduce the consequences about:
\begin{itemize}
	\item ${}^{\mathrm{h}} \Ext^n_{\mathfrak{h}, K^H}(M|_H, \chi)$ (cf.\ Corollary \ref{prop:Ext-consequence-1}),
	\item $\Ext^n_{\mathfrak{h}, K^H}(V|_H, \chi)$, $\Hm_n(\mathfrak{h}, K^H; V|_H \otimes \chi^\vee)$ (cf.\ Corollary \ref{prop:Ext-consequence-2}), where $V$ is a Harish-Chandra $(\mathfrak{g}, K)$-module and $\chi^\vee$ means the contragredient of $\chi$.
\end{itemize}
The finiteness of $\dim \Hm_n(\mathfrak{h}, K^H; V|_H \otimes \chi^\vee)$ is still covered by the work of M.\ Kitagawa \cite[Fact 4.7]{Ki21}.

Finally, the Proposition \ref{prop:H-Loc} also has a monodromic analogue:
\begin{equation*}
	\Hm_n(\mathfrak{h}, K^H; V|_H \otimes \chi^\vee) \simeq \Hm^{- n + \dim X}\Lder\left( \mathrm{coInv}^{\CC, K^H}_{\CC, \{1\}} \right) \left(i_x^! \mathbf{Loc}_{X, \chi}(V)\right).
\end{equation*}
When $K^H = \{1\}$, it is isomorphic to $\Hm^{-n - \dim X}\left(i_x^* \mathbf{Loc}_{X, \chi}(V)\right)$.

\section{Comparison with the analytic picture}\label{sec:analytic}
\subsection{Schwartz homologies}
In this subsection, we work on the analytic side of representation theory. Therefore, we take
\begin{itemize}
	\item $G$: an almost linear Nash group,
	\item $K$: a maximal compact subgroup of $G$.
\end{itemize}
They are both Lie groups, and we will consider continuous representations of them.

Let $E$ be a smooth Fréchet representation of $G$ of moderate growth. In \cite{CS21}, Y.\ Chen and B.\ Sun defined the \emph{Schwartz homologies} $\Hm^{\mathcal{S}}_n(G; E)$ for $n \in \Z_{\geq 0}$. These are locally convex topological vector spaces which are possibly non-Hausdorff, functorial in $E$, and satisfy various desirable properties such as long exact sequences and Shapiro's lemma. In particular,
\[ \Hm^{\mathcal{S}}_0(G; E) = E_G := E \big/ \sum_{g \in G} (g - \identity)(E) \; + \;\text{quotient topology}. \]
We refer to \cite{CS21} for all the details and terminologies.

Let $\mathfrak{g}$ be the complexified Lie algebra of $G$. Every smooth representation $E$ of $G$ gives rise to an $(\mathfrak{g}, K)$-module in a generalized sense: it is a vector space equipped with compatible structures of  $\mathfrak{g}$-module and $K$-module; unlike \S\ref{sec:gK-basic}, the $K$-action here is not assumed to be algebraic or locally finite, but only smooth.

Let $\mathfrak{p} := \mathfrak{g}/\mathfrak{k}$, on which $K$ acts; we may and do choose a decomposition $\mathfrak{g} = \mathfrak{k} \oplus \mathfrak{p}$ as $K$-modules, and denote by $\mathcal{P}: \mathfrak{g} \to \mathfrak{p}$ the corresponding projection.

\begin{definition}
	For an $(\mathfrak{g}, K)$-module $E$, the relative Lie algebra homologies of $E$ are
	\[ \Hm_n(\mathfrak{g}, K; E) := \Hm_n\left( C(\mathfrak{g}, K; E) \right) \]
	where the \emph{standard complex} $C(\mathfrak{g}, K; E)$ for $E$ is defined by
	\[
		C_n(\mathfrak{g}, K; E) := (\bigwedge^n \mathfrak{p}) \dotimes{K} E, \quad n \in \Z_{\geq 0}.
	\]
	Here we make the left $K$-module $\bigwedge^n \mathfrak{p}$ into a right one by $u k := k^{-1} u$, so $\otimes_{K}$ is simply the tensor product over the abstract group algebra $\CC[K]$. Alternatively, $\bigwedge^n \mathfrak{p} \dotimes{K} E$ is also the space of co-invariants (algebraically) of $\bigwedge^n \mathfrak{p} \otimes E$ under the diagonal $K$-action.
	
	The boundary maps in the complex
	\[ \partial_n: C_n(\mathfrak{g}, K; E) \to C_{n-1}(\mathfrak{g}, K; E) \]
	for $n \geq 1$ are the usual ones (cf.\ \cite[(2.126b)]{KV95}):
	\begin{multline*}
		\partial_n((\xi_1 \wedge \cdots \wedge \xi_n) \otimes v) = \sum_{i=1}^n (-1)^i (\cdots \wedge \widehat{\xi_i} \wedge \cdots) \otimes \xi_i v \\
		+ \sum_{p<q} (-1)^{p+q} (\mathcal{P}[\xi_p, \xi_q] \wedge \cdots \widehat{\xi_p} \cdots \widehat{\xi_q} \cdots ) \otimes v.
	\end{multline*}
\end{definition}

Now let $E$ be a smooth Fréchet representation of $G$ of moderate growth. Then each $C_n(\mathfrak{g}, K; E)$ gets topologized, each $\partial_n$ is continuous, and the homologies are endowed with quotient topologies. In \cite[Theorem 7.7]{CS21} it is proved that
\begin{equation}\label{eqn:HS-Lie}
	\Hm^{\mathcal{S}}_n(G; E) \simeq \Hm_n(\mathfrak{g}, K; E), \quad n \in \Z_{\geq 0}
\end{equation}
canonically and topologically.

\begin{definition}
	Let $E$ be a smooth Fréchet representation of $G$ of moderate growth. Denote by $E^{K\text{-fini}}$ the subspace of $K$-finite vectors in $E$.
\end{definition}

Therefore $E^{K\text{-fini}}$ is an $(\mathfrak{g}, K)$-module in the usual algebraic sense. By functoriality and \eqref{eqn:HS-Lie}, we obtain canonical linear maps
\[ \Hm_n\left(\mathfrak{g}, K; E^{K\text{-fini}}\right) \to \Hm^{\mathcal{S}}_n(G; E), \quad n \in \Z_{\geq 0}. \]

\begin{proposition}\label{prop:HS-int}
	Let $E$ be a smooth Fréchet representation of $G$ of moderate growth. For every $n$, we have
	\[ \Hm_n\left(\mathfrak{g}, K; E^{K\text{-fini}}\right) \rightiso \Hm^{\mathcal{S}}_n(G; E) \]
	as vector spaces.
\end{proposition}
\begin{proof}
	Let $\mathrm{Irr}(K)$ be the set of isomorphism classes of finite-dimensional irreducible representations of $K$. Let $\mathcal{M}(K)$ be the unital algebra of bounded complex Borel measures on $K$ under convolution. Let $R(K)$ be the non-unital subalgebra comprised of left $K$-finite (equivalently, right $K$-finite) elements, which is also spanned by the matrix coefficients from $\mathrm{Irr}(K)$ once we fix a Haar measure. See \cite[I.2]{KV95}.
	
	By continuity and completeness, $\mathcal{M}(K)$ acts on $E$, thus so does $R(K)$.
	
	For each $\gamma \in \mathrm{Irr}(K)$, define
	\[ \chi_\gamma := \dim\gamma \cdot \Theta_{\gamma^\vee} \cdot \frac{\dd k}{\mathrm{mes}(K)} \; \in R(K), \]
	where $\Theta_{\gamma^\vee}$ is the character of the contragredient $\gamma^\vee$; see \cite[(1.23)]{KV95}. This is an idempotent in $R(K)$ which is $K$-invariant under conjugation.
	
	For every finite subset $\Gamma$ of $\mathrm{Irr}(K)$, we define $\chi_\Gamma := \sum_{\gamma \in \Gamma} \chi_\gamma$. This is still a $K$-invariant idempotent. When applied to locally $K$-finite representations, its effect is to project to the $\Gamma$-isotypic part.
	
	The evident maps $C_n(\mathfrak{g}, K; E^{K\text{-fini}}) \to C_n(\mathfrak{g}, K; E)$ give a morphism between complexes. It suffices to show it is an isomorphism in each degree $n$.
	
	First, we prove that $C_n(\mathfrak{g}, K; E^{K\text{-fini}}) \to C_n(\mathfrak{g}, K; E)$ is surjective. Take $\Gamma$ that contains all irreducible constituents of $\bigwedge^n \mathfrak{p}$. Let $r \mapsto r'$ be the anti-involution of $R(K)$ induced by $k \mapsto k^{-1}$. For every $\omega \otimes v \in \bigwedge^n \mathfrak{p} \dotimes{K} E$, we claim that
	\[ \omega \otimes v = \chi_\Gamma \omega \otimes v = \omega \otimes \chi'_\Gamma v. \]
	
	Indeed, the first equality follows from the choice of $\Gamma$. As for the second one, in the proof of \cite[Theorem 7.7]{CS21}, a ``strong projective resolution'' (45) of $E$ is shown to be topologically isomorphic to $C(\mathfrak{g}, K; E)$ after taking $(\cdot)_G$. It is part of the theory of Schwartz homologies, especially \cite[Theorem 5.9]{CS21}, that taking $(\cdot)_G$ of such a resolution produces Fréchet spaces in each degree. Hence the algebraic co-invariants in the formation of $\bigwedge^n \mathfrak{p} \dotimes{K} v$ are actually topological, i.e.
	\[ \bigwedge^n \mathfrak{p} \dotimes{K} E = (\bigwedge^n \mathfrak{p} \otimes E) \bigg/ \overline{\lrangle{k\omega \otimes kv = \omega \otimes v : k \in K, \ldots }}. \]
	This implies $\chi_\Gamma \omega \otimes v = \omega \otimes \chi'_\Gamma v$ in $\bigwedge^n \mathfrak{p} \dotimes{K} E$, by approximating $\chi_{\Gamma}$ by Dirac measures. Surjectivity follows since $\chi_{\Gamma} E = E^\Gamma$, the $\Gamma$-isotypic part of $E^{K\text{-fini}}$.
	
	Let us show injectivity. By general properties of $\otimes$ over $\CC[K]$, we have
	\[ \bigwedge^n \mathfrak{p} \dotimes{K} (E^{K\text{-fini}}) = \bigwedge^n \mathfrak{p} \dotimes{K} \varinjlim_{\Gamma} E^\Gamma \leftiso \varinjlim_{\Gamma} \left( \bigwedge^n \mathfrak{p} \dotimes{K} E^\Gamma \right) \]
	where the $\varinjlim$ over finite subsets $\Gamma \subset \mathrm{Irr}(K)$ is filtered. The map from the leftmost term to $\bigwedge^n \mathfrak{p} \dotimes{K} E$ is determined by the compatible family of evident maps
	\[ \bigwedge^n \mathfrak{p} \dotimes{K} E^\Gamma \to \bigwedge^n \mathfrak{p} \dotimes{K} E. \]
	
	Fixing $\Gamma$, the map above is a split injection since $\chi_\Gamma$ is a $K$-invariant idempotent, so
	\[ E = \chi_\Gamma E \oplus (\identity - \chi_\Gamma) E = E_\Gamma \oplus (\identity - \chi_\Gamma) E \]
	as $\CC[K]$-modules. As the filtered $\varinjlim$ of injections is still injective, this shows the injectivity of $\bigwedge^n \mathfrak{p} \dotimes{K} (E^{K\text{-fini}}) \to \bigwedge^n \mathfrak{p} \dotimes{K} E$.

	Note that the same method shows that $\bigwedge^n \mathfrak{p} \dotimes{K} E^\Gamma \to \bigwedge^n \mathfrak{p} \dotimes{K} E^{\Gamma'}$ is split injective whenever $\Gamma \subset \Gamma'$.
\end{proof}

\subsection{Comparison map}\label{sec:comparison}
We now consider almost linear Nash groups
\begin{equation}\label{eqn:four-groups-an}\begin{tikzcd}
	H \arrow[phantom, r, "\subset" description] & G \\
	K^H \arrow[phantom, u, "\subset" description, sloped] \arrow[phantom, r, "\subset" description] & K \arrow[phantom, u, "\subset" description, sloped].
\end{tikzcd}\end{equation}
analogously to the algebraic setting of \eqref{eqn:four-groups}, the assumptions now being
\begin{itemize}
	\item $G$ is reductive,
	\item $K$ (resp.\ $K^H$) is a maximal compact subgroup of $G$ (resp.\ $H$); this implies $K^H = H \cap K$.
\end{itemize}

We consider \emph{Casselman--Wallach representations} of $G$, i.e.\ smooth admissible Fréchet representation of moderate growth; taking $K$-finite vectors establishes an equivalence between this category and the category of Harish-Chandra $(\mathfrak{g}, K)$-modules. See \cite{BK14}.

For a smooth Fréchet representation $E$ of $G$ of moderate growth, its restriction $E|_H$ to $H$ is still smooth Fréchet of moderate growth, by unraveling definitions. However, admissibility is usually lost.

In view of \eqref{eqn:HS-Lie} applied to $H$ and $K^H$, the inclusion $E^{K\text{-fini}} \subset E$ induces natural linear maps
\begin{equation}\label{eqn:H-vs-HS}
	c_n(E): \Hm_n\left(\mathfrak{h}, K^H; E^{K\text{-fini}}\right) \to \Hm^{\mathcal{S}}_n(H; E|_H), \quad n \in \Z_{\geq 0}
\end{equation}
between vector spaces.

\begin{definition}
	We call \eqref{eqn:H-vs-HS} the \emph{comparison maps} for the Casselman--Wallach representation $E$ relative to the data \eqref{eqn:four-groups-an}.
\end{definition}

The following question is thus natural.
\begin{center}\fbox{\begin{minipage}{0.8\textwidth}
	For what data \eqref{eqn:four-groups-an}, representations $E$ and $n$ is $c_n(E)$ an isomorphism?
\end{minipage}}\end{center}

\begin{remark}\label{rem:comparison-K}
	In view of Proposition \ref{prop:HS-int}, the map $c_n(E)$ is the same as the
	\[ \Hm_n\left(\mathfrak{h}, K^H; E^{K\text{-fini}}\right) \to \Hm_n\left(\mathfrak{h}, K^H; E^{K^H\text{-fini}}\right) \]
	induced by $E^{K\text{-fini}} \subset E^{K^H\text{-fini}}$.
\end{remark}

For such questions, it is usual to focus on the case when $H$ is a real spherical subgroup of $G$. An affirmative answer for all $E$ will have strong consequences. For example, in degree $n=0$ it amounts to \emph{automatic continuity} for invariant linear functionals (cf.\ \cite[Theorem A]{AGKL16}), which are hard-core results available only in sporadic cases, for example when (i) $n=0$ and $H \subset G$ is a symmetric subgroup \cite{BD88}, or (ii) when $n$ is arbitrary and $H \subset G$ is a maximal unipotent subgroup \cite{HT98, LLY21}.

In order to apply the earlier results, let us assume further that
\begin{itemize}
	\item all the groups $G$, $K$, $H$, $K^H$ arise from the algebraic setting \eqref{eqn:four-groups}, namely they arise from $\R$-points of affine groups;
	\item $H$ is a reductive spherical subgroup of $G$, when viewed as complex groups.
\end{itemize}

Since $V := E^{K\text{-fini}}$ is a Harish-Chandra $(\mathfrak{g}, K)$-module, by Corollary \ref{prop:Ext-consequence-2} or results of \cite{Ki21}, the source of $c_n(E)$ is finite-dimensional. Under these assumptions, whether $c_n(E)$ is an isomorphism or not is still wide open. In \S\ref{sec:example-comparison}, we will consider some cases in which $c_n(E)$ turn out to be isomorphisms for all $n$.

\begin{remark}
	A smooth Fréchet representation $F$ of $H$ of moderate growth
	is said to be \emph{homologically finite} (resp.\ \emph{homologically separated}) if $\Hm^{\mathcal{S}}_n(H; F)$ is finite-dimensional (resp.\ Hausdorff) for all $n$; homological finiteness implies homological separation. See \cite{BC21} for an overview. With all the previous assumptions (algebraicity, sphericity), it is still unknown whether $E|_H$ is homologically finite or not, except for specific choices of $H$ or $E$, or for $n=0$. If $c_n(E)$ is surjective, then $E|_H$ will be homologically finite. In particular, homological finiteness will hold in all the examples of \S\ref{sec:example-comparison}.
\end{remark}

\begin{remark}
	In parallel with the monodromic setting considered in \S\ref{sec:monodromic}, one should also consider $\Hm^{\mathcal{S}}_n(H; E|_H \otimes \chi^{-1})$ when $\chi$ is a smooth character of $H$ factoring through the unipotent radical of $H$, and modify \eqref{eqn:H-vs-HS} accordingly.
\end{remark}

\subsection{Examples of comparison isomorphisms}\label{sec:example-comparison}
Consider the groups in \eqref{eqn:four-groups-an}; we assume that they are all affine groups in what follows.

\begin{proposition}\label{prop:disc-decomp}
	Let $E$ be a Casselman--Wallach representation of $G$. If $E^{K^H\text{-fini}} = E^{K\text{-fini}}$, then the comparison map $c_n(E)$ is an isomorphisms for all $n$.
\end{proposition}
\begin{proof}
	Immediate from Remark \ref{rem:comparison-K}.
\end{proof}

\begin{example}\label{eg:admissible-restriction}
	Suppose $E$ is of the form $E = B^\infty$, where $B$ is an irreducible unitary representation of $G$ and $B^\infty$ denotes its smooth part. We say $B$ is \emph{$K^H$-admissible} if $B|_{K^H}$ decomposes discretely into a Hilbert direct sum of irreducibles with finite multiplicities. By \cite[Proposition 1.6]{Ko98}, this implies that $B^{K\text{-fini}}$ is \emph{discretely decomposable} as an $(\mathfrak{h}, K^H)$-module. The notions of admissibility and discrete decomposability under restriction have been extensively studied by T.\ Kobayashi and his collaborators; see \cite{Ko98} or \cite[\S 4.1]{Ko15} for an overview.
	
	We claim that when $B$ is $K^H$-admissible, $c_n(E)$ is an isomorphism for all $n$.
	
	Indeed, by \cite[Proposition 1.6]{Ko98}, the $K^H$-admissibility implies
	\[ B^{K\text{-fini}} = (B|_H)^{\infty, K^H\text{-fini}}; \]
	the left-hand side is $E^{K\text{-fini}}$ whilst the right-hand side contains $B^{\infty, K^{H\text{-fini}}} = E^{K^H\text{-fini}}$. Hence Proposition \ref{prop:disc-decomp} can be applied.
	
	Discrete decomposability is a rare phenomenon, so the scope of the condition above is limited. Nonetheless, it covers some interesting families: suppose $H \subset G$ is a symmetric subgroup, and $H/K^H \to G/K$ is a holomorphic embedding of Hermitian symmetric domains, then every unitary highest weight module (eg.\ holomorphic discrete series) is $K^H$-admissible, hence discretely decomposable over $(\mathfrak{h}, K^H)$. See \cite[Fact 5.4]{Ko98}.
\end{example}

\begin{example}\label{eg:HK-admissibility}
	The $K^H$-admissibility mentioned above always holds if $H = K$.
\end{example}

Before stating the next example, we record a standard reduction due to Hecht--Taylor.

\begin{proposition}\label{prop:HT-reduction}
	If $c_n(E)$ is an isomorphism for all $n$ and all principal series representations $E$, then the same is true for all Casselman--Wallach representations $E$.
\end{proposition}
\begin{proof}
	In view of the isomorphisms of abstract vector spaces
	\[ \Hm_n\left(\mathfrak{h}, K^H; E^{K^H\text{-fini}} \right) \simeq \Hm_n\left(\mathfrak{h}, K^H; E\right) \simeq \Hm^{\mathcal{S}}_n(H; E|_H) \]
	of Proposition \ref{prop:HS-int}, one can reiterate the proof of \cite[Proposition 3]{HT98}, which proceeds by resolving the $(\mathfrak{g}, K)$-module $E^{K\text{-fini}}$ into a complex of principal series via Casselman's Subrepresentation Theorem, and analyzing the corresponding map between two double complexes, each computing $(\mathfrak{h}, K^H)$-homologies in the vertical direction.
	
	The point here is that $(\bigwedge^p \mathfrak{p}) \dotimes{K^H} (\cdot)$ is an exact functor from locally $K^H$-finite representations to vector spaces, for every $p \in \Z$. Indeed, taking $K^H$-co-invariants is an exact functor on such representations, as follows from \cite[Proposition 1.18]{KV95}. In \cite{HT98} one took $\otimes$ instead of $\otimes_{K^H}$, but the exactness is all one needs to analyze the spectral sequence $E_2^{q, -p}$ therein.
\end{proof}

\begin{example}\label{eg:SL2}
	Take $G = \SL(2)$, $K = \SO(2)$, $H \simeq \Gm$ being the diagonal torus, and $K^H = K \cap H$; note that $H$ is a symmetric subgroup of $G$, and $K^H = \mu_2 := \{\pm 1\}$. We also take $U$ to be the upper-triangular unipotent subgroup, and identify all these groups with their $\R$-points. By identifying $\R^2$ with row vectors and $U \backslash G$ with $\R^2 \smallsetminus \{0\}$, the principal series can be described as follows, on the level of Casselman--Wallach representations.

	Let $\lambda \in \CC$ and $\epsilon \in \{0, 1\}$. Define
	\[ V^\epsilon_\lambda := \left\{\begin{array}{r|l}
		f: \R^2 \smallsetminus \{0\} \to \CC & \forall t \in \R^\times, \; x \in \R^2 \smallsetminus \{0\} \\
		\text{smooth} & f(tx) = \sgn(t)^\epsilon |t|^\lambda f(x) 
	\end{array}\right\}. \]
	It is topologized in the standard way, and $G$ acts by $(gf)(x) = f(xg)$.
	
	In particular, $\bigl(\begin{smallmatrix} a & \\ & a^{-1} \end{smallmatrix}\bigr)$ maps $f$ to $(x_1, x_2) \mapsto f(ax_1, a^{-1}x_2)$, and $-1 \in K^H$ maps $f$ to $(-1)^\epsilon f$. We also remark that $V^\epsilon_\lambda$ is the normalized parabolic induction of
	\[ \begin{pmatrix} a & * \\ & a^{-1} \end{pmatrix} \mapsto \sgn(a)^\epsilon |a|^{-\lambda - 1}, \quad a \in \R^\times. \]
	
	Let $\theta := \bigl(\begin{smallmatrix} 1 & \\ & -1\end{smallmatrix}\bigr)$, which is $K^H$-invariant and generates $\mathfrak{h}$. The $\theta$-invariants in $V^\epsilon_\lambda$ are precisely the functions $f$ which are locally constant on the real curve $x_1 x_2 = c$, for each $c \in \R$.
	
	Clearly, $(V^\epsilon_\lambda)^{K^H\text{-fini}} = V^\epsilon_\lambda$ is strictly larger than $(V^\epsilon_\lambda)^{K\text{-fini}}$.
	
	For $L \in \{K^H, K\}$, the complex computing $\Hm_\bullet(\mathfrak{h}, K^H; (V^\epsilon_\lambda)^{L\text{-fini}})$ is
	\begin{equation}\label{eqn:SL2-ps}\begin{aligned}
		\CC\theta \dotimes{K^H} (V^\epsilon_\lambda)^{L\text{-fini}} & \to \CC \dotimes{K^H} (V^\epsilon_\lambda)^{L\text{-fini}}, \quad \deg = -1, 0, \\
		\theta \otimes v & \mapsto 1 \otimes \theta v.
	\end{aligned}\end{equation}
	
	When $\epsilon = 0$, the $\Hm_1$ is $\{f \in (V^0_\lambda)^{L\text{-fini}} : \theta f = 0 \}$. On each open quadrant, $f$ takes the form $c |x_1 x_2|^{\lambda /2}$ for some $c \in \CC$. In order that they glue to a smooth function on $\R^2 \smallsetminus \{0\}$, the constants $c$ must coincide and we must have $\lambda \in 4\Z_{\geq 0}$, in which case $f(x_1 x_2) = c (x_1 x_2)^{\lambda /2}$ is polynomial, hence $K$-finite. Therefore
	\[ \Hm_1\left(\mathfrak{h}, K^H; (V^0_\lambda)^{L\text{-fini}} \right) = \begin{cases}
		0, & \text{if}\; \lambda \notin 4\Z_{\geq 0} \\
		\CC (x_1 x_2)^{\lambda/2}, & \text{if}\; \lambda \in 4\Z_{\geq 0}.
	\end{cases}\]

	When $\epsilon = 1$, we have $\CC\theta \dotimes{K^H} (V^1_\lambda)^{L\text{-fini}} = \CC \dotimes{K^H} (V^1_\lambda)^{L\text{-fini}} = 0$, thus
	\[ \Hm_n\left(\mathfrak{h}, K^H; (V^1_\lambda)^{L\text{-fini}} \right) = 0, \quad n \in \Z. \]
	
	These computations show that $c_1(V_\lambda^\epsilon)$ is always an isomorphism. The same also holds for $c_0(V_\lambda^\epsilon)$, by either invoking the automatic continuity theorem for $H \subset G$, or a direct computation as before. By Proposition \ref{prop:HT-reduction}, it follows that $c_n(E)$ is an isomorphism for all Casselman--Wallach representations $E$ and all $n$.
\end{example}

The example above shows that, even in the simplest case, it may happen simultaneously that
\begin{itemize}
	\item $E^{K^H\text{-fini}} \supset E^{K\text{-fini}}$ strictly;
	\item $c_n(E)$ is an isomorphism for all $n$;
	\item some higher homologies of $E|_H$ are nonzero.
\end{itemize}

In particular, the condition in Proposition \ref{prop:disc-decomp} is sufficient but not necessary in general.

\begin{example}\label{eg:SL2-more}
	The proper reductive spherical subgroups $H$ of $G = \SL(2)$, identified with their $\R$-points, are
	\[ \begin{pmatrix} * & \\ & * \end{pmatrix}, \quad \left\langle \begin{pmatrix} * & \\ & * \end{pmatrix}, \begin{pmatrix} & 1 \\ -1 & \end{pmatrix}\right\rangle, \quad \SO(2) \]
	up to conjugacy. Take $K = \SO(2)$, so that $K^H := K \cap H$ is maximal compact in $H$ in each case.
	
	In Example \ref{eg:SL2} it is shown that $c_n(E)$ is an isomorphism in the first case, for all $E$ and $n$. We will settle the remaining cases below.
	
	For the second case (the normalizer of the diagonal torus), retain the notations from Example \ref{eg:SL2} and consider the principal series $V_\lambda^\epsilon$, which suffices. One has $K^H \simeq \Z/4\Z$ whose generators act on $\theta \in \mathfrak{h}$ by $-1$. For $L \in \{K^H, K\}$ we may decompose $(V_\lambda^\epsilon)^{L\text{-fini}}$ into $K^H$-isotypic components. The complex \eqref{eqn:SL2-ps} becomes
	\[ (V_\lambda^\epsilon)^{L\text{-fini}}[\rho] \to (V_\lambda^\epsilon)^{L\text{-fini}}[\mathrm{triv}], \quad v \mapsto \theta v, \]
	where $\rho$ (resp.\ $\mathrm{triv}$) denotes the character of $K^H$ mapping the generators to $-1$ (resp.\ $1$), and $[\cdots]$ are the isotypic components. As in Example \ref{eg:SL2}, the homologies turn out to be independent of $L$. Hence $c_n(E)$ is an isomorphism for all $E$ and $n$.
	
	Finally, the third case $H = \SO(2)$ is covered by Example \ref{eg:HK-admissibility}.
\end{example}

\bibliographystyle{abbrv}
\bibliography{Loc}


\vspace{1em}
\begin{flushleft} \small
	Beijing International Center for Mathematical Research / School of Mathematical Sciences, Peking University. No.\ 5 Yiheyuan Road, Beijing 100871, People's Republic of China. \\
	E-mail address: \href{mailto:wwli@bicmr.pku.edu.cn}{\texttt{wwli@bicmr.pku.edu.cn}}
\end{flushleft}

\end{document}